\documentclass[11pt,a4paper,twoside]{amsart}
\usepackage[activeacute,english]{babel}
\usepackage[latin1]{inputenc}
\usepackage{amsfonts}
\usepackage{amsmath,amsthm}
\usepackage{amssymb}
\usepackage{graphics}
\usepackage{pdfsync}
 
\newcommand{\HH}{{\mathcal H}}
\newcommand{\OO}{{\mathcal O}}
\newcommand{\VV}{{\mathcal V}}
\newcommand{\FF}{{\mathcal F}}
\newcommand{\GG}{{\mathcal G}}
\newcommand{\CC}{{\mathcal C}}
\newcommand{\II}{{\mathcal I}}

\newcommand{\TT}{{\mathcal T}}
\newcommand{\LL}{{\mathcal L}}
\newcommand{\SSS}{{\mathcal S}}

\newcommand{\RR}{{\mathcal R}}
\newcommand{\DD}{{\mathcal D}}

\newcommand{\N}{{\mathbb N}}
\newcommand{\Z}{{\mathbb Z}}
\newcommand{\R}{{\mathbb R}}
\newcommand{\C}{{\mathbb C}}
\newcommand{\wit}{\widetilde}

\newcommand{\supp}{{\operatorname{supp}}}
\newcommand{\diam}{{\operatorname{diam}}}

\newcommand{\Lip}{{\operatorname{Lip}}}
\newcommand{\dist}{{\operatorname{dist}}}
\newtheorem{teo}{Theorem}[section]

\newtheorem{lema}[teo]{Lemma}

\newtheorem{prob}[teo]{Problem}

\newtheorem{claim}[teo]{Claim}
{\theoremstyle{remark} \newtheorem{remark}[teo]{Remark}}
\allowdisplaybreaks[1]

\newtheorem*{teo*}{Theorem}

\makeatletter
\@addtoreset {equation}{section}
\makeatother

\title[Variation for singular integrals on Lipschitz graphs]
{Variation and oscillation for\\singular integrals with odd
kernel\\on Lipschitz graphs}
\author[A. MAS AND X. TOLSA]{ALBERT MAS\; AND\; XAVIER TOLSA}
\date{March, 2011}
\subjclass[2010]{Primary 42B20, 42B25.} 
\keywords{$\rho$-variation,
oscillation, Calder\'{o}n-Zygmund singular integrals.}
\thanks{Both authors are partially supported by grants
2009SGR-000420 (Generalitat de
Catalunya) and MTM2010-16232 (Spain). Albert Mas is also
supported by grant AP2006-02416 (FPU program, Spain).}
\address{Albert Mas. Departament de Matem\`atiques, Universitat
Aut\`onoma de Bar\-ce\-lo\-na, 08193 Catalonia, Spain} \email{amblesa@mat.uab.cat}
\address{Xavier Tolsa. Instituci\'{o} Catalana de Recerca
i Estudis Avan\c{c}ats (ICREA) and Departament de
Ma\-te\-m\`a\-ti\-ques, Universitat Aut\`onoma de Bar\-ce\-lo\-na, 08193
Catalonia} \email{xtolsa@mat.uab.cat}

\begin{document}

\begin{abstract}
We prove that, for $\rho>2$, the $\rho$-variation and oscillation for the smooth truncations of the Cauchy
transform on Lipschitz graphs are bounded in $L^p$ for $1<p<\infty$.
The analogous result holds for the $n$-dimensional
Riesz transform on $n$-dimensional Lipschitz graphs, as well as for
other singular integral operators with odd kernel. In particular,
our results strengthen the classical theorem on the $L^2$ boundedness of the Cauchy transform on Lipschitz graphs by Coifman, McIntosh, and Meyer.
\end{abstract}

\maketitle

\section{Introduction}

The $\rho$-variation and oscillation for martingales and some
families of operators have been studied in many recent papers on
probability, ergodic theory, and harmonic analysis (see
\cite{Lepingle}, \cite{Bourgain}, \cite{JKRW-ergodic},
\cite{CJRW-Hilbert}, and \cite{JSW}, for example). The purpose of
this paper is to establish some new results concerning the
$\rho$-variation and oscillation for families of singular integral
operators defined on Lipschitz graphs. In particular, our results
include the $L^p$ boundedness of the $\rho$-variation and the
oscillation for the smooth truncations of the Cauchy transform and the $n$-dimensional Riesz
transform on Lipschitz graphs, for $1<p<\infty$ and $\rho>2$.

Given a Borel measure $\mu$ in $\R^d$, one defines the
$n$-dimensional Riesz transform of a function $f\in L^1(\mu)$ by
$R^\mu f(x)=\lim_{\epsilon\searrow0}R^\mu_\epsilon f(x)$
(whenever the limit exists), where
$$R^\mu_\epsilon f(x)=\int_{|x-y|>\epsilon}\frac{x-y}
{|x-y|^{n+1}}\,f(y)\,d\mu(y),\qquad x\in\R^d.$$
When $d=2$ (i.e., $\mu$ is a Borel measure in $\C$), one defines the
Cauchy transform of  $f\in L^1(\mu)$ by $C^\mu
f(x)=\lim_{\epsilon\searrow0}C^\mu_\epsilon f(x)$ (whenever the
limit exists), where
$$C^\mu_\epsilon f(x)=\int_{|x-y|>\epsilon}\frac{f(y)}
{x-y}\,d\mu(y),\qquad x\in\C.$$
To avoid the problem of existence of the preceding limits, it is useful to consider the maximal
operators $R^\mu_* f(x)=\sup_{\epsilon>0}|R^\mu_\epsilon f(x)|$
and $C^\mu_* f(x)=\sup_{\epsilon>0}|C^\mu_\epsilon f(x)|$.

The Cauchy and Riesz transforms are two very important examples
of singular integral operators with a Calder\'{o}n-Zygmund kernel.
The kernels $K:\R^d\setminus\{0\}\to\R$ that we consider in this
paper satisfy
\begin{equation}\label{eq333}
|K(x)|\leq \frac{C}{|x|^{n}},\quad|\partial_{x^i}K(x)|\leq
\frac{C}{|x|^{n+1}}\quad\text{and}\quad|\partial_{x^i}\partial_{x^j}K(x)|\leq
\frac{C}{|x|^{n+2}},
\end{equation}
for all $1\leq i,j\leq d$ and
$x=(x^1,\ldots,x^d)\in\R^d\setminus\{0\}$, where $0<n<d$ is some
integer and $C>0$ is some constant; and moreover $K(-x)=-K(x)$ for
all $x\neq0$ (i.e. $K$ is odd). Notice that the $n$-dimensional
Riesz transform corresponds to the vector kernel
$(x^1,\ldots,x^d)/|x|^{n+1}$, and the Cauchy transform
to $(x^1,-x^2)/|x|^{2}$ (so, we may consider $K$
to be any scalar component of these vector kernels).

Given an odd kernel $K$ satisfying \eqref{eq333} and a finite Borel
measure $\mu$ in $\R^d$, for each $\epsilon>0$, we consider the $\epsilon$-truncated operator
$$T_\epsilon \mu(x)=\int_{|x-y|>\epsilon}K(x-y)\,d\mu(y),\qquad x\in\R^d,$$
and then we set $T\mu(x)=\lim_{\epsilon\searrow0}T_\epsilon \mu(x)$
whenever the limit makes sense, and
$T_*\mu(x)=\sup_{\epsilon>0}|T_\epsilon \mu(x)|$. Finally, given
$f\in L^1(\mu)$, we define $T^\mu_\epsilon
f(x):=T_\epsilon(f\mu)(x)$, $T^\mu f(x):=T(f\mu)(x)$ and $T_*^\mu
f(x):=T_*(f\mu)(x)$. Thus, for a suitable choice of $K$, the operator $T^\mu$ coincides with the Cauchy or Riesz transforms.

Besides the operator $T_\epsilon$ defined above, one can consider another $\epsilon$-truncated
variant that we proceed to define. First we need some additional notation.
 Given $x=(x^1,\ldots,x^d)\in\R^d,$ we use the
notation $\widetilde x:=(x^1,\ldots,x^n)\in\R^n$.
Let $\varphi_\R:[0,\infty)\to[0,\infty)$ be a non decreasing $\CC^2$
function such that
$\chi_{[3\sqrt{n},\infty)}\leq\varphi_\R\leq\chi_{[2.1\sqrt{n},\infty)}$
(the numbers $3\sqrt{n}$ and $2.1\sqrt{n}$ are chosen just for definiteness and they are not important).
Given $\epsilon>0$ and $x\in\R^d$, we denote
\begin{gather*}
\varphi_{\epsilon}(x):=\varphi_\R(|\widetilde x|/\epsilon)
\quad\text{and}\quad\varphi:=\{\varphi_{\epsilon}\}_{\epsilon>0}.
\end{gather*}

Given $K$ as above, $x\in\R^{d}$, $0<\epsilon$, and a finite Borel measure $\mu$,
we set
\begin{equation*}
(K\varphi_{\epsilon}*\mu)(x) :=\int\varphi_{\epsilon}(x-y)K(x-y)\,d\mu(y).
\end{equation*}
 We also denote $(K\varphi*\mu)(x):=\{(K\varphi_\epsilon*\mu)(x)\}_{\epsilon>0}$.
Finally, given
$f\in L^1(\mu)$, we define $T^\mu_{\varphi_\epsilon}f(x):=(K\varphi_{\epsilon}*(f\mu))(x)$, $T^\mu_{\varphi}f(x):=\lim_{\epsilon\to0}T^\mu_{\varphi_\epsilon}f(x)$ (whenever the limit makes sense),
$T^\mu_{\varphi_*}f(x):=\sup_{\epsilon>0}|T^\mu_{\varphi_\epsilon}f(x)|$, and $\TT^\mu_{\varphi}f(x):=\{T^\mu_{\varphi_\epsilon}f(x)\}_{\epsilon>0}$.

Let $\II$ be a subset of $\R$ (in this paper, we will always have
$\II=(0,\infty)$ or $\II=\Z$),  and let
$\mathcal{F}:=\{F_\epsilon\}_{\epsilon\in\II}$ be a family of
functions defined on $\R^d$. Given  $\rho>0$, the $\rho$-{\em
variation} of $\FF$ at $x\in\R^d$ is defined by
\begin{equation*}
\VV_{\rho}(\FF)(x):=\sup_{\{\epsilon_{m}\}}\bigg(\sum_{m\in\Z}
|F_{\epsilon_{m+1}}(x)-F_{\epsilon_{m}}(x)|^{\rho}\bigg)^{1/\rho},
\end{equation*}
where the pointwise supremum is taken over all decreasing
sequences $\{\epsilon_{m}\}_{m\in\Z}\subset\II$. Fix a decreasing
sequence $\{r_{m}\}_{m\in\Z}\subset\II$. The {\em oscillation} of $\FF$
at $x\in\R^d$ is defined by
\begin{equation*}
\OO(\FF)(x):=\sup_{\{\epsilon_m\},\{\delta_{m}\}}\bigg(\sum_{m\in\Z}
|F_{\epsilon_m}(x)-F_{\delta_m}(x)|^{2}\bigg)^{1/2},
\end{equation*}
where the pointwise supremum is taken over all sequences
$\{\epsilon_m\}_{m\in\Z}\subset\II$ and $\{\delta_{m}\}_{m\in\Z}\subset\II$
such that $r_{m+1}\leq\epsilon_{m}\leq\delta_{m}\leq r_{m}$ for all
$m\in\Z$.

In this paper we are interested in studying the $\rho$-variation and oscillation for the family
$\TT^\mu_{\varphi}f$. That is, we will deal with
\begin{equation*}
\begin{split}
&(\VV_{\rho}\circ\TT^\mu_{\varphi})f(x):=\VV_{\rho}(\TT^\mu_{\varphi}f)(x)
=\VV_{\rho}(K\varphi*(f\mu))(x)\quad\text{and}\\
&(\OO\circ\TT^\mu_{\varphi})f(x):=\OO(\TT^\mu_{\varphi}f)(x)
=\OO(K\varphi*(f\mu))(x),
\end{split}
\end{equation*}
for a Borel measure $\mu$ and $f\in L^1(\mu)$.
Although it is not clear from the definitions, these operators are
$\mu$-measurable (see \cite{CJRW-Hilbert}, \cite{JSW}).

Given $E\subset \R^d$, we denote by $\HH^n_E$ the $n$-dimensional
Hausdorff measure restricted to $E$.
Let $\Gamma:=\{x\in \R^d\,:\,x=(\widetilde x,A(\widetilde
x))\}$ be the graph of a Lipschitz function $A:\R^n\to\R^{d-n}$ with Lipschitz constant $\Lip(A)$. Let
$H^1(\HH^{n}_{\Gamma})$ and $BMO(\HH^{n}_{\Gamma})$ be the (atomic)
Hardy space and the space of functions with bounded mean
oscillation, respectively, with respect to the measure
$\HH^{n}_{\Gamma}$. The following is our main result.

\begin{teo}\label{main theorem}
Let $\rho>2$, let $K$ be a kernel satisfying  \eqref{eq333}, and set $\mu:=\HH^{n}_{\Gamma}$. The operators
$\VV_{\rho}\circ\TT^{\mu}_{\varphi}$ and
$\OO\circ\TT^{\mu}_{\varphi}$ are bounded
\begin{itemize}
\item in $L^p(\mu)$ for $1<p<\infty$,
\item from $H^1(\mu)$ to $L^1(\mu)$,
\item from $L^1(\mu)$ to $L^{1,\infty}(\mu)$, and
\item from $L^\infty(\mu)$ to $BMO(\mu)$.
\end{itemize}
In all the cases above, the norm of $\OO\circ\TT^{\mu}_{\varphi}$ is bounded independently of the sequence that defines $\OO$.
\end{teo}

Let us recall that the $L^2(\HH^1_\Gamma)$ boundedness of the Cauchy
transform on Lipschitz graphs $\Gamma\subset\C$ with slope small
enough was proved by A.\ P.\ Calder\'{o}n in his celebrated  paper
\cite{Calderon}. The $L^2$ boundedness on Lipschitz graphs in full
generality was proved later on by R. Coifman, A. McIntosh, and Y.
Meyer \cite{CMM}.

Consider the Cauchy kernel $K(z)=1/z$ ($z\in\C$), and set $\mu:=\HH^1_\Gamma$, so $C_\epsilon^\mu=T_{\epsilon}^\mu$. By standard Calder\'on-Zygmund theory
(namely, Cotlar's inequality), the $L^2(\mu)$ boundedness of the
Cauchy transform $C^{\mu}$ is equivalent to the $L^2(\mu)$
boundedness of the maximal operator $C_*^{\mu}$. Let $M^\mu$ denote the Hardy-Littlewood maximal operator with respect to the measure $\mu$. It is easy to check
that, for $f\in L^1(\mu)$ with compact support, there exists some constant $C_0>0$ such that
$$C_\epsilon^{\mu}f(x)\leq T_{\varphi_\epsilon}^\mu f(x)+C_0 M^\mu f(x)\leq (\VV_\rho\circ \TT_{\varphi}^\mu)f(x)+C_0 M^\mu f(x)$$ for all $\epsilon>0$, thus
$(\VV_\rho\circ\TT_{\varphi}^\mu)+C_0M^\mu$ controls the maximal operator $C_*^{\mu}$
and, in this sense, Theorem \ref{main theorem} (together with the known $L^p(\mu)$ boundedness of $M^\mu$) strengthens the
results of \cite{Calderon} and \cite{CMM}. Analogous conclusions
hold for the $n$-dimensional Riesz transform and
the maximal operator $R_*^\mu$.

The operator $\VV_\rho\circ\TT_\varphi^\mu$ is also related to an
important open problem posed by G. David and S. Semmes which actually is our main motivation to prove Theorem \ref{main theorem}. We need some definitions to state it.

Recall that a measure $\mu$ is said to be $n$-dimensional Ahlfors-David
regular, or simply AD regular, if there exists some constant $C$
such that $C^{-1}r^n\leq\mu(B(x,r))\leq Cr^n$ for all $x\in\supp\mu$
and $0<r\leq\diam(\supp\mu)$. It  is not difficult to
see that such a measure $\mu$ must be of the form
$\mu=h\,\HH^n_{\supp\mu}$, where $h$ is some positive function
bounded above and away from zero. A Borel set $E\subset\R^d$ is
called AD regular if the measure $\HH^n_{E}$ is AD regular.
One says that $\mu$ is uniformly $n$-rectifiable, or simply
uniformly rectifiable, if there exist $\theta,M>0$ so that, for each
$x\in\supp\mu$ and $R>0$, there is a Lipschitz mapping $g$ from
the $n$-dimensional ball $B^n(0,R)\subset\R^n$ into $\R^d$ such that
$\Lip(g)\leq M$ and
$\mu\big(B(x,R)\cap g(B^n(0,R))\big) \geq \theta R^n,$
where $\Lip(g)$ stands for the Lipschitz constant of $g$. In the
language of \cite{DS2}, this means that {\em $\supp\mu$  has big
pieces of Lipschitz images of $\R^n$}. A Borel set $E\subset\R^d$ is
called uniformly $n$-rectifiable if $\HH^n_{E}$ is $n$-uniformly
rectifiable. Of course, the $n$-dimensional graph of a Lipschitz
function is uniformly $n$-rectifiable.

David and Semmes asked the following question, which is still open (see \cite[Chapter 7]{Pajot}):

\begin{prob}\label{problema DS}
Is it true that
an $n$-dimensional AD regular measure $\mu$ is $n$-uniformly
rectifiable if and only if $R_*^\mu$ is bounded in $L^2(\mu)$?
\end{prob}

It is proved in \cite{DS1} that if $\mu$ is uniformly rectifiable,
then $R_*^\mu$ is  bounded in $L^2(\mu)$. However, the converse
implication has been proved only in the case $n=1$ and
$d=2$, by P. Mattila, M. Melnikov and J. Verdera \cite{MMV}, using
the notion of curvature of measures (which seems to be useful only
in this case).

Set $\RR^\mu:=\{R^\mu_\epsilon\}_{\epsilon>0}$. By combining some techniques from \cite{DS2}
and \cite{To}, in our forthcoming paper \cite{MT} we show that the $L^2(\mu)$ boundedness of
$\VV_\rho\circ\RR^\mu$ implies that $\mu$ is uniformly $n$-rectifiable. Moreover, 
we also prove that $\VV_\rho\circ\RR^\mu$ is bounded in $L^2(\mu)$ for all AD regular uniformly $n$-rectifiable measures $\mu$. So we obtain the following theorem, which might be considered as a first 
approach to a possible solution of Problem \ref{problema DS}:

\begin{teo}
Let $\rho>2$. An $n$-dimensional AD regular measure $\mu$ is uniformly
$n$-rectifiable if and only if $\VV_\rho\circ\RR^\mu$ is a bounded operator in
$L^2(\mu)$.
\end{teo}

An essential ingredient for the proof of this result is
Theorem \ref{main theorem} above. The arguments and techniques used to derive the $L^2$ boundedness of $\VV_\rho\circ\RR^\mu$ on uniformly rectifiable measures from the $L^2$ boundedness of $\VV_\rho\circ\RR_\varphi^\mu$
on Lipschitz graphs are quite delicate ($\RR^\mu_\varphi$ is defined as $\RR^\mu$ but using the family $\varphi$ for the truncations). In particular, they involve the corona type decomposition introduced in \cite{DS1}. For this reason, the proof of the preceding
theorem is out of the scope of this paper and will appear in \cite{MT}.

Concerning the background on the $\rho$-variation and oscillation, a
fundamental result is L\'epingle's inequality \cite{Lepingle}, from
which the $L^p$ boundedness of the $\rho$-variation and oscillation
for martingales follows, for $\rho>2$ and $1<p<\infty$ (see Theorem \ref{osc-var
martingala} below for more details). From this result on
martingales, one deduces that the $\rho$-variation and oscillation
are also bounded in $L^p$ for the averaging operators (also called
differentiation operators, see \cite{JKRW-ergodic}):
\begin{equation}\label{eqaver11}
D_\epsilon
f(x)=\frac1{|B(x,\epsilon)|}\int_{B(x,\epsilon)}f(y)\,dy,\qquad
x\in\R.
\end{equation}

As far as we know, the first work dealing with the $\rho$-variation
and oscillation for singular integral operators is the one of J.\
Campbell, R.\ L.\ Jones, K.\ Reinhold and M.\ Wierdl
\cite{CJRW-Hilbert}, where the $L^p$ and weak $L^1$ boundedness of
the $\rho$-variation (for $\rho>2$) and oscillation for the Hilbert transform was
proved. Recall that, for $f\in L^p(\R)$ and $x\in\R$,
$$H_\epsilon f(x) = \frac1\pi \int_{|x-y|>\epsilon}\frac{1}{x-y}\,f(y)\,dy,$$
and then the Hilbert transform of $f$ is defined by
$Hf(x)=\lim_{\epsilon\to0}H_\epsilon f(x)$, whenever the limit
exists.
Later on, there appeared other papers showing the $L^p$
boundedness of the $\rho$-variation and oscillation for singular
integrals in $\R^d$ (\cite{CJRW-singular integrals}), with weights
(\cite{GT}), or for other operators such as the spherical averaging
operator or singular integral operators on parabolas (\cite{JSW}).
Finally, we remark that, very recently, the
case of the Carleson operator has been considered too (\cite{Lacey},
\cite{OSTTW}).

Notice that the Hilbert transform is one of the simplest examples
where Theorem \ref{main theorem} applies (one sets
$\Gamma=\R$, i.e. $A\equiv0$), and so one obtains a new proof of the $L^p$
boundedness of the $\rho$-variation and oscillation for the Hilbert
transform. In the original proof in \cite{CJRW-Hilbert}, a key
ingredient was the following classical identity, which follows via
the Fourier transform:
\begin{equation}\label{hilbert}
Q_\epsilon =P_\epsilon*H,
\end{equation}
where $P_\epsilon$ is the Poisson kernel and $Q_\epsilon$ is the
conjugated Poisson kernel. Using this identity and the close
relationship between the operators $Q_\epsilon$ and $H_\epsilon$,
Campbell {\it et al.} derived the $L^p$ boundedness of the
$\rho$-variation and oscillation for the Hilbert transform from the
one of the family $\{D_\epsilon(H f)\}_{\epsilon>0}$, where
$D_\epsilon$ is the averaging operator in \eqref{eqaver11} (notice
that $P_\epsilon$ can be written as a convex combination of
operators $D_\delta$, $\delta>0$).

In most of the previous results concerning $\rho$-variation and
oscillation of families of operators from harmonic analysis, the
Fourier transform is a fundamental tool. However, this is not useful
in order to prove Theorem \ref{main theorem}, since the graph
$\Gamma$ is not invariant under translations in general. Moreover,
even for the Cauchy transform, there is no formula like
(\ref{hilbert}), which relates the truncations of a singular
integral operator with an averaging operator applied to a singular
integral operator, when $\Gamma$ is a general Lipschitz graph.

The main ingredients of our proof of Theorem \ref{main theorem} are the known results on the $\rho$-variation and oscillation
for martingales (L\'{e}pingle's inequality \cite{Lepingle}) and a
multiscale analysis which stems from the geometric proof of the
$L^2$ boundedness of the Cauchy transform on Lipschitz graphs by P.\
W.\ Jones \cite{Jones-Escorial} and his celebrated work
\cite{Jones-salesman} on quantitative rectifiability in the plane,
using the so called $\beta$ coefficients. Some of the techniques in
these papers were further developed in higher dimensions by David
and Semmes \cite{DS1} for Ahlfors-David regular sets. More recently,
in \cite{To} some coefficients denoted by $\alpha$, in the spirit of
the Jones' $\beta$'s, were introduced, and they were shown to be
useful for the study of the $L^p$-boundedness of Calder\'on-Zygmund
operators on Lipschitz graphs and on uniformly rectifiable sets (see
the definition below Theorem \ref{teoappli}). In our paper, the
$\alpha$ and $\beta$ coefficients play a fundamental role.

Let us remark that L\'{e}pingle's inequality, which asserts the
$L^p$ boundedness of the $\rho$-variation of martingales, fails if
one assumes $\rho\leq2$ (see \cite{Qi} and \cite{JW}, for example). Moreover,
this fact can be brought to the $\rho$-variation of averaging
operators and singular integral operators, thus it is essential to
assume $\rho>2$ in Theorem \ref{main theorem}. Analogous
conclusions hold if one replaces the $\ell^2$-norm by and
$\ell^\rho$-norm with $\rho<2$ in the definition of $\OO$. See
\cite{CJRW-Hilbert}, or \cite{Akcoglu} for the case of martingales.

Concerning the direct applications of Theorem \ref{main theorem}, it
is easily seen that the $L^p$ boundedness of
$\VV_\rho\circ\TT_\varphi^\mu$ yields a new proof of the existence of
the principal values $T_{\varphi}^{\mu}
f(x):=\lim_{\epsilon\to0}T_{\varphi_\epsilon}^{\mu}f(x)$  for all $f\in L^p(\mu)$ and almost all $x\in\Gamma$,
without using a dense class of functions in $L^p(\mu)$ as
in the classical proof. Moreover, from Theorem \ref{main theorem} one also gets some information on the speed of convergence.
In fact, a classical result derived from variational inequalities is
the boundedness of the {\em $\lambda$-jump operator} $N_\lambda\circ\TT_\varphi^\mu$ and
the {\em $(a,b)$-upcrossings operator} $N_a^b\circ\TT_\varphi^\mu$. Given $\lambda>0$,
$f\in L^1_{loc}(\mu)$ and $x\in\R^d$, one defines
$(N_\lambda\circ\TT_\varphi^\mu)f(x)$ as the supremmum of all integers $N$ for which
there exists
$0<\epsilon_1<\delta_1\leq\epsilon_2<\delta_2\leq\cdots\leq\epsilon_N<\delta_N$
so that
$$|T^{\mu}_{\varphi_{\epsilon_i}}f(x)-T^{\mu}_{\varphi_{\delta_i}}f(x)|>\lambda$$
for each $i=1,\ldots,N$. Similarly, given $a<b$, one defines $(N_a^b\circ\TT_\varphi^\mu)
f(x)$ to be the supremmum of all integers $N$ for which there exists
$0<\epsilon_1<\delta_1\leq\epsilon_2<\delta_2\leq\cdots\leq\epsilon_N<\delta_N$
so that
$T^{\mu}_{\varphi_{\epsilon_i}}f(x)<a$ and $T^{\mu}_{\varphi_{\delta_i}}f(x)>b$
for each $i=1,\ldots,N$. Using Theorem \ref{main theorem} one
obtains (by the same arguments as in \cite[Theorem 1.3 and Corollary
7.1]{CJRW-Hilbert}) the following:

\begin{teo}\label{teoappli}
Let $\rho>2$, $\lambda>0$, and let $K$, and $\mu$ be as in
Theorem \ref{main theorem}. For $1<p<\infty$, there exist
constants $C_1$ and $C_2$ depending on $\rho$, $n$, $d$, $K$, and
$\Lip(A)$ (and on $p$ for the case of $C_1$) such that
\begin{gather*}
\|\big((N_\lambda\circ\TT_\varphi^\mu)f\big)^{1/\rho}\|_{L^p(\mu)}\leq\frac{C_1}
{\lambda}\,\|f\|_{L^p(\mu)}\,\text{ and }\\
\mu(\{x\in\Gamma:\,(N_\lambda\circ\TT_\varphi^\mu)f(x)>m\})\leq\frac{C_2}{\lambda
m^{1/\rho}}\,\|f\|_{L^1(\mu)}.
\end{gather*}
\end{teo}
Trivially, $(N_a^b\circ\TT_\varphi^\mu)f\leq (N_{b-a}\circ\TT_\varphi^\mu)f$,
thus Theorem \ref{teoappli} also holds replacing $\lambda$ by $b-a$
and $N_\lambda$ by $N_a^b$. In \cite{JSW} it is shown that the
results of Theorem \ref{teoappli} still hold when $\rho=2$ for the
particular case of the Hilbert transform. In our paper we do not pursue this endpoint result.


\section{Preliminaries}\label{s preliminaries}
Throughout all the paper, $n$ and $d$ are two fixed integers such
that $0<n<d$. Given a point $x=(x^1,\ldots,x^d)\in\R^d,$ we use the
notation $\widetilde x:=(x^1,\ldots,x^n)\in\R^n$. Given a
function $f:\R^m\to\R$, we denote by $\nabla f$ its gradient (when
it makes sense), and by $\nabla^2f$ the matrix of second derivatives
of $f$. If $f$ depends on different points $x_1,x_2,\ldots\in\R^m$,
then $\nabla_{x_i}f$ denotes the gradient of $f$ with respect to the
$x_i$ variable, and analogously for $\nabla_{x_i}^2f$.

For two sets $F_1,F_2\subset\R^d$, we denote by $\dist_\HH(F_1,F_2)$
the Hausdorff distance between $F_1$ and $F_2$. We denote by $\LL^n$
the Lebesgue measure on $\R^n$, and for the sake of simplicity, we
set $\|\cdot\|_{p}:=\|\cdot\|_{L^p(\LL^n)}$ for $1\leq p\leq\infty$,
and $dy:=d\LL^n(y)$ for $y\in\R^n$.

In the paper, when we refer to the angle between two affine
$n$-planes in $\R^d$, we mean the angle between the $n$-dimensional subspaces associated to the $n$-planes. As usual, the letter `$C$' stands
for some constant which may change its value at different
occurrences, and which quite often only depends on $n$ and $d$. The notation $A\lesssim B$ ($A\gtrsim B$) means that
there is some fixed constant $C$ such that $A\leq CB$ ($A\geq CB$),
with $C$ as above. Also, $A\approx B$ is equivalent to $A\lesssim B
\lesssim A$.

\subsection{More about the family $\varphi$}
Given $x\in\R^{d}$, $0<\epsilon\leq\delta$, and a finite Borel measure $\mu$,
we set $\varphi_\epsilon^\delta(x):=\varphi_\epsilon(x)-\varphi_\delta(x)$ and we define
\begin{equation*}
(K\varphi_{\epsilon}^\delta*\mu)(x):=\int\varphi_{\epsilon}^\delta(x-y)K(x-y)\,d\mu(y),
\end{equation*}
thus
$(K\varphi_{\epsilon}^\delta*\mu)(x)=(K\varphi_{\epsilon}*\mu)(x)-(K\varphi_\delta*\mu)(x)$.

For $m\in\N$, $x\in\R^m$, and $R\geq r>0$, we denote by $B^m(x,r)$
the closed ball of $\R^m$ with center $x$ and radius $r$, and by
$A^m(x,r,R)$ the closed annulus of $\R^m$ centered at $x$ with inner
radius $r$ and outer radius $R$. We also use the notation $B(x,r)$
and $A(x,r,R)$ when there is no possible confusion about $m$.

Each function
$\varphi_\epsilon^\delta$ is non negative, and
$\supp\varphi_\epsilon^\delta\subset
A^n(0,2.1\epsilon\sqrt{n},3\delta\sqrt{n})\times\R^{d-n}\subset\R^d.$ Moreover,
$\sum_{j\in\Z}\varphi_{2^{-j-1}}^{\,2^{-j}}(x)=1$ for $\wit x\neq0$, and
there are at most two terms that do not vanish in the previous sum
for a given $x\in\R^d$.

\subsection{The $\alpha$ and $\beta$ coefficients. Special dyadic lattice}\label{ss alpha i beta}

Given $m\in\N$, $\lambda>0$, and a cube $Q\subset\R^m$  (i.e.
$Q:=[0,b)^m+a$ with $a\in\R^m$ and $b>0$), $\ell(Q)$ denotes the
side length of $Q$, $z_Q$ denotes the center of $Q$ and $\lambda Q$
denotes the cube with center $z_Q$ and side length $\lambda\ell(Q)$.
Throughout the paper, we will only use cubes with sides parallel to
the axes.

Let $\mu$ be a locally finite Borel measure on $\R^d$.
Given $1\leq p<\infty$ and a cube $Q\subset\R^d$, one sets (see \cite{DS2})
\begin{equation}\label{def beta_p}
\beta_{p,\mu}(Q) = \inf_L
\biggl\{ \frac1{\ell(Q)^n}\int_{2Q} \biggl(\frac{\dist(y,L)}{\ell(Q)}\biggr)^pd\mu(y)\biggr\}^{1/p},
\end{equation}
where the infimum is taken over all $n$-planes $L$ in $\R^d$.
For $p=\infty$ one replaces the $L^p$ norm by the supremum norm:
\begin{equation}\label{def beta_infty}
\beta_{\infty,\mu}(Q) = \inf_L \biggl\{ \sup_{y\in \supp\mu\cap 2Q}
\frac{\dist(y,L)}{\ell(Q)}\biggr\},
\end{equation}
where the infimum is taken over all $n$-planes $L$ in $\R^{d}$
again.  These coefficients were introduced by P. W. Jones in
\cite{Jones-Escorial} for $p=\infty$ and by G. David and S. Semmes
in \cite{DS1} for $1\leq p<\infty$.

Let $F\subset\R^d$ be the closure of an open set. Given two finite
Borel measures $\sigma$, $\nu$ on $\R^d$, one sets
\begin{equation}\label{wasserstein dist}
\dist_F(\sigma,\nu):= \sup\Bigl\{ \Bigl|{\textstyle \int f\,d\sigma -
\int f\,d\nu}\Bigr|:\,{\rm Lip}(f) \leq1,\,\supp f\subset
F\Bigr\}.
\end{equation}
It is easy to check that this is a distance in the space
of finite  Borel measures $\sigma$ such that $\supp\sigma\subset F$
and $\sigma(\partial F)=0$. Moreover, it turns out that this distance is a variant of the well known Wasserstein distance $W_1$ from optimal transportation (see \cite[Chapter 1]{Vi}).
 See \cite[Chapter 14]{Mattila-llibre}
for other properties of $\dist_F$.

Given a cube $Q$ which intersects $\supp\mu$,
consider the closed ball $B_Q:=B(z_Q,6\ell(Q))$. Then one defines (see \cite{To})
\begin{equation}\label{def alpha}
\alpha_\mu^n(Q):=\frac1{\ell(Q)^{n+1}}\,\inf_{c\geq0,L} \,\dist_{B_Q}(\mu,\,c\HH^n_{L}),
\end{equation}
where the infimum is taken over all constants $c\geq0$ and all $n$-planes $L$ in $\R^d$.
For convenience, if $Q$ does not intersect $\supp\mu$, we set $\alpha^n_\mu(Q)=0$.
To simplify notation, sometimes
we will write $\alpha_\mu(Q)$ or $\alpha(Q)$ instead of $\alpha_\mu^n(Q)$ (and analogously for the $\beta$'s).

The following result characterizes uniform rectifiability in terms of the $\alpha$ and $\beta$ coefficients.
\begin{teo}\label{teounif}
Let $\mu$ be an $n$-dimensional AD regular measure on $\R^d$, and consider any $p\in[1,2]$. Then, the following are equivalent:
\begin{itemize}
\item[(a)] $\mu$ is uniformly $n$-rectifiable.
\item[(b)] For any cube $R\subset\R^d$,
\begin{equation}\label{pack0}
\sum_{Q\in\DD_{\R^d}(R)}\beta_{p,\mu}(Q)^2\ell(Q)^n\leq C\ell(R)^n
\end{equation}
with $C$ independent of $R$, where $\DD_{\R^d}(R)$ stands for the
collection of cubes of $\R^d$ contained in $R$ which are obtained by
splitting $R$ dyadically.
\item[(c)] There exists $C>0$ such that, for any cube $R\subset\R^d$,
\begin{equation}\label{pack2}
\sum_{Q\in\DD_{\R^d}(R)}\alpha_\mu(Q)^2\ell(Q)^n\leq C\ell(R)^n.
\end{equation}
\end{itemize}
\end{teo}

The equivalence (a)$\Longleftrightarrow$(b) in Theorem \ref{teounif}
was proved by G. David and S. Semmes  in \cite{DS1}, and the
equivalence (a)$\Longleftrightarrow$(c) was proved by X. Tolsa in
\cite{To}.

In this paper we will use a slightly different definition of the
$\alpha$  and $\beta$ coefficients adapted to the $n$-uniformly
rectifiable measure $\mu=f\HH^n_\Gamma$, where
$\Gamma:=\{x\in \R^d\,:\,x=(\widetilde x,A(\widetilde x))\}$
is the $n$-dimensional graph of a given Lipschitz function
$A:\R^n\to\R^{d-n}$  and $f\in L^\infty(\HH^n_\Gamma)$ satisfies
$f(x)\approx1$ for almost all $x\in\Gamma$. To this end, we need to
introduce a special dyadic lattice of sets related to $\Gamma$.
Given a cube $\wit Q\subset\R^n$ (i.e. $\wit Q:=[0,b)^n+a$ with
$a\in\R^n$ and $b>0$), we define $Q:=\wit Q\times\R^{d-n}$. This
type of set will be called {\em v-cube} (``vertical'' cube). We
denote by $\ell(Q)$ and $\wit z_Q$ the side length and center of
$\wit Q$, respectively, and given $\lambda>0$ we set $\lambda
Q:=\lambda \wit Q\times\R^{d-n}$. Let $\wit\DD$ denote the standard
dyadic lattice of $\R^n$, and set $\DD:=\{Q:\wit Q\in\wit\DD\}.$ It
is easy to check that the v-cubes of $\DD$ intersected with $\Gamma$
provide a dyadic lattice associated to the graph $\Gamma$ in the
sense of \cite[Appendix 1]{David-LNM}. Finally, for $m\in\Z$, set
$\DD_m:=\{Q\in\DD:\,\ell(Q)=2^{-m}\}$.

Fix a constant $C_\Gamma>10\sqrt{n}(1+\Lip(A))$ (the
precise  value of $C_\Gamma$ will not be relevant in the proofs
given in the paper). Given $1\leq p\leq\infty$ and a v-cube
$Q\subset\R^d$, we define the coefficient $\beta_{p,\mu}(Q)$ as in
(\ref{def beta_p}) and (\ref{def beta_infty}) but replacing $2Q$ by
$C_\Gamma Q$. We also define $\alpha_\mu(Q)$ as in (\ref{def alpha})
but taking $B_Q:=B(\wit
z_Q,C_\Gamma\ell(Q))\times\R^{d-n}\subset\R^d$. This new definition
of the $\alpha$ and $\beta$ coefficients (adapted to the graph
$\Gamma$) is the one that we will use in the whole paper.

\begin{remark}\label{remark packing}
It is an exercise to check that, with this new definition of the
$\alpha$'s  and $\beta$'s, inequalities (\ref{pack0}) and
(\ref{pack2}) of Theorem \ref{teounif} still hold. Moreover, the
following is an easy consequence of (\ref{pack0}) and (\ref{pack2}):
{\em Let $\Gamma$ be an $n$-dimensional Lipschitz graph, $f\in
L^\infty(\HH^n_\Gamma)$ such that $f(x)\approx1$ for almost all
$x\in\Gamma$, and $\mu=f\HH^n_\Gamma$. Let $1\leq p\leq2$. Given
$C_1,C_2,C_3\geq1$, there exists a constant $C_4>0$ such that, for
any $R\in\DD$,}
\begin{equation*}
\begin{split}
\sum_{Q\in\DD:\,Q\subset C_1R}\big(\,\beta_{p,\mu}(C_2Q)^2+\alpha_\mu(C_3Q)^2\,\big)\,\mu(Q)\leq C_4\mu(R),
\end{split}
\end{equation*}
and the dependence of $C_4$ with respect to $\Gamma$ is only on $\Lip(A)$.
\end{remark}

\begin{remark}\label{remark alpha2}
It is shown in \cite[Lemma 3.2]{To}, that
$\beta_{1,\mu}(Q)\lesssim\alpha_\mu(Q)$  for all $Q\in\DD$. Given
$Q\in\DD$, let $L_Q$ be a minimizing $n$-plane for $\alpha_\mu(Q)$. In
general, $\beta_{\infty,\mu}(Q)$ can not be controlled by
$\beta_{1,\mu}(Q)$, so given $x\in\supp\mu\cap C_\Gamma Q$, we can
not control $\dist(x,L_Q)$ by means of $\alpha_\mu(Q)$. But it is
shown in \cite[Lemma 5.2]{To} that
$\dist(x,L_Q)\lesssim\sum_{R\in\DD:\,x\in R\subset Q}\alpha_\mu(R)\ell(R),$ and in particular, if $P\in\DD$ is such that $P\subset Q$ and
$x\in\supp\mu\cap C_\Gamma P$, and $L_P$ denotes a minimizing
$n$-plane for $\alpha_\mu(P)$, one has (see \cite[Remark 5.3]{To})
\begin{equation}\label{remark alpha2 eq1}
\begin{split}
\dist(x,L_Q)\lesssim\dist(x,L_P)+\sum_{R\in\DD:\,P\subset R\subset Q}\alpha_\mu(R)\ell(R).
\end{split}
\end{equation}
\end{remark}

\subsection{Martingales}\label{ss particular martingale}
First of all, let us recall a particular case of L\'{e}pingle's inequality (see \cite{JSW}, or \cite{Lepingle} and \cite[Theorem
6.4]{JKRW-ergodic} for martingales in a probability space):
\begin{teo}\label{osc-var martingala}
Let $(X,\Sigma,\lambda)$ be a $\sigma$-finite measure space and $\rho>2$. Then,
there exist constants $C_1,C_2>0$ such that, for every martingale
$\GG:=\{G_m\}_{m\in\Z}\in L^2(\lambda)$,
\begin{equation*}
\begin{split}
\|\VV_\rho(\GG)\|_{L^2(\lambda)}\leq
C_1\|\GG\|_{L^2(\lambda)}\quad\text{and}\quad
\|\OO(\GG)\|_{L^2(\lambda)}\leq C_2\|\GG\|_{L^2(\lambda)},
\end{split}
\end{equation*}
where $\|\GG\|_{L^2(\lambda)}:=\sup_{m\in\Z}\|G_m\|_{L^2(\lambda)}$.
The constants $C_1$ and $C_2$ do not depend on the measure
$\lambda$, and $C_2$ neither depends on the fixed sequence that
defines $\OO$.
\end{teo}

To prove Theorem \ref{main theorem}, we need to introduce a particular martingale, and to review some known results.

\begin{lema}\label{lema Tmu}
Fix a cube $\widetilde P\subset\R^n$ (not necessarily dyadic) and a
Lipschitz graph $\Gamma:=\{x\in \R^d\,:\,x=(\widetilde
x,A(\widetilde x))\}$ such that $\supp A\subset\widetilde P$. Consider
the measure $\mu:=f\HH^n_{\Gamma}$, where $f(x)=1$ for all
$\widetilde x\in\widetilde{P}^c$ and $C_0^{-1}\leq f(x)\leq C_0$  for all
$\widetilde x\in\widetilde P$, for some fixed constant $C_0>0$. Also
set $P:=\widetilde P\times \R^{d-n}$. Then, the following hold:
\begin{gather}
T_*\mu\in L^1_{loc}(\mu),\quad T_*(\chi_E\mu)\in L^1_{loc}(\mu)
\text{ for every compact set } E\subset\R^d,\text{ and}\label{condicio mu martingala2}\\
\|T\mu\|_{L^2(\mu)}\lesssim\mu(P)^{1/2}.\label{condicio mu
martingala3}
\end{gather}
\end{lema}

\begin{remark}\label{remark intergacio infty}
To avoid the problem of non-integrability near infinity, for this
type of measures $\mu$ we redefine
$T_\epsilon\mu(x):=\lim_{M\to\infty}\int\chi_{(\epsilon,M)}(|x-y|)K(x-y)\,d\mu(y)$,
which exists because $\mu$ is flat outside a compact set and $K$ is
odd. All the results in this paper remain valid with this new
definition and the adjustments that have to be done in the proofs
are minimal.

In this paper, we will deal with other integrals which concern the
kernel $K$ and the measure $\mu$ near infinity. The
non-integrability problem can be avoided in the same manner.
\end{remark}

\begin{proof}[Proof of Lemma \ref{lema Tmu}]
It is known that the operator $T_*^\mu$ is bounded in $L^2(\mu)$,
because $T_*^\mu$ is the maximal operator associated to a
Calder\'{o}n-Zygmund singular integral  and $\mu$ is a uniformly
rectifiable measure (see \cite{DS1}). Thus,
$T_*(\chi_E\mu)=T_*^\mu(\chi_E)\in L^1_{loc}(\mu)$ for every compact
set $E\subset\R^d$.

We are going to check that
$\|T_*\mu\|_{L^2(\mu)}\lesssim\mu(P)^{1/2}$.  This will imply that
$T_*\mu\in L^1_{loc}(\mu)$ and, since $T\mu$ exists (because $\mu$
is uniformly rectifiable) and $|T\mu|\leq T_*\mu$, we will also
obtain $\|T\mu\|_{L^2(\mu)}\lesssim\mu(P)^{1/2}$; so the lemma will
be proved.

Using that $T_*^\mu$ is bounded in $L^2(\mu)$, we have
\begin{equation}\label{martingala1}
\begin{split}
\|T_*\mu\|_{L^2(\mu)}&\leq\|T_*(\chi_{3P}\mu)\|_{L^2(\mu)}+\|T_*(\chi_{(3P)^c}\mu)\|_{L^2(\mu)}\\
&\lesssim\mu(P)^{1/2}+\|T_*(\chi_{(3P)^c}\mu)\|_{L^2(\mu)}.
\end{split}
\end{equation}

Set $L:=\R^n\times\{0\}^{d-n}\subset\R^d$; obviously $\chi_{P^c}\mu=\HH^n_{L\setminus P}$.
Since $L$ is an $n$-plane and $K$ is odd, $T_*\HH^n_L(x)=0$ for all $x\in L$. Thus,
\begin{equation}\label{martingala2}
\begin{split}
\|T_*\HH^n_{L\setminus3P}\|_{L^2(\HH^n_L)}\leq\|T_*\HH^n_{L}\|_{L^2(\HH^n_{L})}+\|T_*\HH^n_{L\cap3P}\|_{L^2(\HH^n_L)}
\lesssim\mu(P)^{1/2}.
\end{split}
\end{equation}

Set $z_P:=(\widetilde z_P,0,\ldots,0)\in L$ (recall that $\wit z_P$
denotes the center of $\wit P$) and $\chi_\epsilon(x):=\chi_{(\epsilon,\infty)}(|x|)$. It is obvious that
$\int\chi_{\epsilon}(z_P-y)K(z_P-y)\,d\HH^n_{L\setminus3P}(y)=0$ for
all $\epsilon>0$. Thus, given $x\in\supp\mu\cap P$,
\begin{equation*}
\begin{split}
|(K\chi_{\epsilon}*\HH^n_{L\setminus3P})(x)|&\leq\int\chi_{\epsilon}(x-y)|K(x-y)-K(z_P-y)|\,d\HH^n_{L\setminus3P}(y)\\
&\quad+\int|\chi_{\epsilon}(x-y)-\chi_{\epsilon}(z_P-y)||K(z_P-y)|\,d\HH^n_{L\setminus3P}(y).
\end{split}
\end{equation*}

Since $\Gamma$ is a Lipschitz graph, $|x-z_P|\lesssim\ell(P)$. So,
the first term on right hand side of the previous inequality is
easily bounded by an absolute constant independent of $\epsilon$, by
standard arguments. For the second term, notice that
$\supp(\chi_{\epsilon}(x-\cdot)-\chi_{\epsilon}(z_P-\cdot))\cap(L\setminus
3P)=\emptyset$ for all $\epsilon<\ell(P)$, and
$\HH^n_L(\{y\in\R^n\,:\,\chi_{\epsilon}(x-y)-\chi_{\epsilon}(z_P-y)\neq0\})\lesssim\ell(P)\epsilon^{n-1}$
for all $\epsilon\geq\ell(P)$. Therefore, since
$|z_P-y|\approx\epsilon$ for all
$y\in\supp(\chi_{\epsilon}(x-\cdot)-\chi_{\epsilon}(z_P-\cdot))\cap(L\setminus
3P)$, the second term can also be estimated by an absolute constant.
Thus, we conclude
$T_*\HH^n_{L\setminus3P}(x)=\sup_{\epsilon>0}|(K\chi_\epsilon*\HH^n_{L\setminus3P})(x)|\lesssim1$
for all $x\in\supp\mu\cap P$.

Using the previous observations and (\ref{martingala2}), we have
\begin{equation*}
\begin{split}
\|T_*(\chi_{(3P)^c}\mu)\|^2_{L^2(\mu)}&=\|T_*\HH^n_{L\setminus3P}\|^2_{L^2(\chi_P\mu)}+
\|T_*\HH^n_{L\setminus3P}\|^2_{L^2(\chi_{P^c}\mu)}\\
&\leq\|T_*\HH^n_{L\setminus3P}\|^2_{L^2(\chi_P\mu)}+
\|T_*\HH^n_{L\setminus3P}\|^2_{L^2(\HH^n_{L})}\lesssim\mu(P),
\end{split}
\end{equation*}
which, combined with (\ref{martingala1}), gives $\|T_*\mu\|_{L^2(\mu)}\lesssim\mu(P)^{1/2}$, as desired.
\end{proof}

We are ready to define the martingale. Let $P$ and $\mu$ be as  in
Lemma \ref{lema Tmu}. Given $m\in\Z$ and $a\in\R^{n}$, we set
$$\widetilde D^{\,a}_{m}:=a+[0,2^{-m})^{n}\subset\R^n\quad\text{and}\quad
D^{\,a}_{m}:=\widetilde D^{\,a}_{m}\times\R^{d-n}\subset\R^d.$$ Set
$\DD_m^{\,a}:=\{D^{a+2^{-m}k}_m\subset\R^d\,:\,k\in\Z^n\}$  (notice
that $\DD_m^{\,a}$ coincides with $\DD_m$ translated by a parameter
$a\in\R^n$ and, for a fixed $a$, $\bigcup_{m\in\Z}\DD^{\,a}_m$ is a
translation of the standard dyadic lattice). Notice that $\mu(D_m^{\,a})\approx2^{-mn}$ for all $m\in\Z$, $a\in\R^n$. For $D\in\DD_m^{\,a}$
and $x\in D$, we set
\begin{equation*}
E_D\mu(x):=\frac{1}{\mu(D)}\int_{D}\int_{D^{c}}K(z-y)\,d\mu(y)\,d\mu(z)
\end{equation*}
(take into account Remark \ref{remark intergacio infty} for the
meaning of $\int_{D^{c}}K(z-y)\,d\mu(y)$). Finally, for $x\in\R^d$,
we define the martingale
$E^{\,a}_{m}\mu(x):=\sum_{D\in\DD_m^{\,a}}\chi_{D}(x)E_D\mu(x)$, $m\in\Z$.

Let us make some comments to understand better the nature of
$E^{\,a}_m\mu$. First of all notice that, since $\mu(\partial D)=0$,
for any $D\in\DD_m^{\,a}$ and $\mu$-almost all $z\in D$ we have
\begin{equation}\label{existencia Em1}
\int_{D^c}K(z-y)\,d\mu(y)=\lim_{\epsilon\to0}\int_{D^c}\chi_{\epsilon}(z-y)K(z-y)\,d\mu(y),
\end{equation}
and for any $\epsilon>0$, we have
\begin{equation}\label{existencia Em2}
\int_{D}\int_{D}\chi_{\epsilon}(z-y)K(z-y)\,d\mu(y)\,d\mu(z)=0
\end{equation}
because of the antisymmetry of $K$. Therefore, by (\ref{existencia
Em1}), (\ref{existencia Em2}), (\ref{condicio mu martingala2}), and
the dominated convergence theorem,
$\int_{D}\left|\int_{D^c}K(z-y)\,d\mu(y)\right|d\mu(z)<\infty$ (in
particular, we have seen that $E^{\,a}_{m}\mu$ is well defined) and
$\int_{D}T(\chi_D\mu)\,d\mu=0$. Using this and (\ref{existencia Em1}), we finally have that
\begin{equation}\label{existencia Em4}
\begin{split}
E^{\,a}_{m}\mu(x)=\frac{1}{\mu(D)}\int_{D}T(\chi_{D^c}\mu)\,d\mu=\frac{1}{\mu(D)}\int_{D}T\mu\,d\mu
\end{split}
\end{equation}
for $x\in D\in\DD_m^{\,a}$, thus $E^{\,a}_{m}\mu(x)$ is the average
of the function $T\mu$ on the v-cube $D\in\DD_m^{\,a}$ which contains
$x$. So, it is completely clear that, for a fixed $a\in\R^n$,
$\{E^{\,a}_{m}\mu\}_{m\in\Z}$ is a martingale. In \cite{MV} it is
shown that $\{E^{\,a}_{m}\mu\}_{m\in\Z}$ is well defined and it is a
martingale without the assumption of the existence of $T\mu$ (i.e.,
for more general measures $\mu$).

Now, we can use (\ref{existencia Em4}), the $L^2$ boundedness of
the dyadic maximal operator and (\ref{condicio mu martingala3}) to
deduce that
\begin{equation}\label{existencia Em5}
\|E^{\,a}_{m}\mu\|_{L^2(\mu)}\lesssim \|T\mu\|_{L^2(\mu)}\lesssim\mu(P)^{1/2}
\end{equation}
for all $a\in\R^n$ and $m\in\Z$, where the constants that appear  in
the previous inequalities only depend on $C_0$, $n$, $d$ and
$\Lip(A)$.

Set $E^{\,a}\mu:=\{E^{\,a}_{m}\mu\}_{m\in\Z}$. Then, the martingale
$E^{\,a}\mu$ belongs to $L^2(\mu)$ by (\ref{existencia Em5}); thus
by Theorem \ref{osc-var martingala}, for all $a\in\R^n$,
\begin{equation}\label{existencia Em6}
\begin{split}
\|\VV_\rho(E^{\,a}\mu)\|_{L^2(\mu)}&\lesssim\|E^{\,a}\mu\|_{L^2(\mu)}\lesssim\mu(P)^{1/2} \qquad\text{for }\rho>2,\\
\|\OO (E^{\,a}\mu)\|_{L^2(\mu)}&\lesssim\|E^{\,a}\mu\|_{L^2(\mu)}\lesssim\mu(P)^{1/2},
\end{split}
\end{equation}
where the constants in the previous inequalities only depend  on
$C_0$, $n$, $d$, and $\Lip(A)$ (and on $\rho$, in the
case of $\VV_\rho$).

Finally, for $x\in\R^d$, we define
\begin{equation*}
E_{m}\mu(x):=2^{mn}\int_{\{a\,:\,x\in D^{\,a}_{m}\}}E_{m}^{\,a}\mu(x)\,da
\end{equation*}
(notice that $\LL^n(\{a\,:\,x\in D^{\,a}_{m}\})=2^{-mn}$). Thus,
$E_m\mu$ is an average (of the $m$'th term) of some martingales
depending on a parameter $a\in\R^n$.

Set $E\mu:=\{E_{m}\mu\}_{m\in\Z}$. We want to obtain estimates like
(\ref{existencia Em6})  for $\VV_\rho(E\mu)$ and $\OO(E\mu)$. We
will only show the details for $\VV_\rho(E\mu)$, because the case of
$\OO(E\mu)$ follows by similar arguments.

One can easily check that
$E_m\mu(x)=2^{Mn}\int_{[0,2^{-M}]^n}E_m^{\,a}\mu(x)\,da$ for all
$m,M\in\Z$ with $M\leq m$. Therefore, for all $M,r,s\in\Z$ with
$M\leq r\leq s$, we have
\begin{equation}\label{existencia Em7}
E_r\mu(x)-E_s\mu(x)=2^{Mn}\int_{[0,2^{-M}]^n}(E_r^{\,a}\mu(x)-E_s^{\,a}\mu(x))\,da.
\end{equation}

Given $M\in\Z$, we consider the auxiliary transformation
\begin{equation*}
\VV_{\rho,M}(E\mu)(x):=\sup_{\{r_{m}\}}\bigg(\sum_{m\in\Z}
|E_{r_{m+1}}\mu(x)-E_{r_{m}}\mu(x)|^{\rho}\bigg)^{1/\rho},
\end{equation*}
where the pointwise supremum is taken over all decreasing
sequences of integers $\{r_{m}\}_{m\in\Z}$ such that $r_m\geq M$ for
all $m\in\Z$. With this definition it is obvious that the sequence
$\{\VV_{\rho,M}(E\mu)(x)\}_{M\in\Z}$ is non increasing and
$\VV_{\rho}(E\mu)(x)=\lim_{M\to-\infty}\VV_{\rho,M}(E\mu)(x)$ for
all $x\in\R^d$. Minkowski's integral inequality and (\ref{existencia
Em7}) yield the pointwise estimate
\begin{equation*}
\begin{split}
\VV_{\rho,M}(E\mu)(x)&=\sup_{\{r_m\}\,:\,r_m\geq
M}\bigg(\sum_{m\in\Z}
|E_{r_{m+1}}\mu(x)-E_{r_{m}}\mu(x)|^{\rho}\bigg)^{1/\rho}\\
&\leq 2^{Mn}\int_{[0,2^{-M}]^n}\sup_{\{r_m\}}\bigg(\sum_{m\in\Z}
|E_{r_{m+1}}^{\,a}\mu(x)-E_{r_m}^{\,a}\mu(x))|^{\rho}\bigg)^{1/\rho}da\\
&=2^{Mn}\int_{[0,2^{-M}]^n}\VV_\rho(E^{\,a}\mu)(x)\,da.\\
\end{split}
\end{equation*}
Therefore, by the previous estimate, Minkowski's integral inequality and (\ref{existencia Em6}),
\begin{equation*}
\|\VV_{\rho,M}(E\mu)\|_{L^2(\mu)}\leq2^{Mn}\int_{[0,2^{-M}]^n}\|\VV_\rho(E^{\,a}\mu)\|_{L^2(\mu)}\,da\leq C\mu(P)^{1/2},
\end{equation*}
where $C>0$ only depends on $C_0$, $n$, $d$, $\Lip(A)$,  and $\rho$. By the monotone convergence theorem, we
conclude that $\|\VV_{\rho}(E\mu)\|_{L^2(\mu)}\lesssim\mu(P)^{1/2}.$ Thus we have proved the following theorem (which can be considered
the starting point to prove Theorem \ref{main theorem}):
\begin{teo}\label{osc-var martingala2}
Fix a cube $\widetilde P\subset\R^n$. Set $\Gamma:=\{x\in
\R^d\,:\,x=(\widetilde x,A(\widetilde x))\},$ where
$A:\R^n\to\R^{d-n}$ is a Lipschitz function supported in $\widetilde
P$, and set $P:=\widetilde P\times \R^{d-n}$. Set
$\mu:=f\HH^n_{\Gamma}$, where $f(x)=1$ for all $\widetilde x\in
\widetilde P^c$ and $C_0^{-1}\leq f(x)\leq C_0$ for all $\widetilde
x\in\widetilde P$, for some constant $C_0>0$.

Let $\rho>2$. Then,
there exist constants $C_1,C_2>0$ such that
$\|\VV_{\rho}(E\mu)\|_{L^2(\mu)}\leq C_1\mu(P)^{1/2}$ and
$\|\OO(E\mu)\|_{L^2(\mu)}\leq C_2\mu(P)^{1/2},$ where $C_1$ and
$C_2$ only depend on $C_0$, $n$, $d$, and $\Lip(A)$
(and on $\rho$ in the case of $C_1$).
\end{teo}

We need to introduce additional notation in order to express
$E_m\mu$ in a more convenient way for our purposes. Let
$\mu_1,\ldots,\mu_k$ be a finite collection of positive Borel
measures  such that $\mu_l(D^{\,a}_{m})>0$ for all $a\in\R^n$,
$m\in\Z$ and $l=1,\ldots,k$. Given $m\in\Z$ and
$\,x_1,\ldots,x_i,y_1,\ldots,y_j\in\R^{d}$, we define
\begin{equation*}
\Lambda^{\mu_1,\ldots,\mu_k}_m(x_1,\ldots,x_i\,;\,y_1,\ldots,y_j):=2^{nm}\int_{\left\{a\,:\,x_1,\ldots,x_i\in D^{\,a}_{m},\,y_1,\ldots,y_j\notin D^{\,a}_{m}\right\}}\frac{da}{\prod_{l=1}^k\mu_l(D^{\,a}_{m})}.
\end{equation*}

Then, by Fubini's theorem,
\begin{equation}\label{def1 Em}
\begin{split}
E_{m}\mu(x)&=\int_{\{a\,:\,x\in D^{\,a}_{m}\}}
\frac{2^{mn}}{\mu(D^{\,a}_{m})}\int_{D^{\,a}_{m}}\int_{(D^{\,a}_{m})^{c}}K(z-y)\,d\mu(y)\,d\mu(z)\,da\\
&=\iint\bigg(2^{mn}\int_{\left\{a\,:\,x,z\in D^{\,a}_{m},\,y\notin
D^{\,a}_{m}\right\}}
\frac{da}{\mu(D^{\,a}_{m})}\bigg)K(z-y)\,d\mu(z)\,d\mu(y)\\
&=\iint \Lambda^{\mu}_{m}(x,z\,;\,y)K(z-y)\,d\mu(z)\,d\mu(y).
\end{split}
\end{equation}

\section{Sketch of the proof of Theorem \ref{main theorem}}\label{seccio sketch
proof}

The proof relies on two basic facts: the known $L^2$ boundedness
of the $\rho$-variation and oscillation of martingales explained in the previous section and the
good geometric properties of Lipschitz graphs from a
measure-theoretic point of view.

As we said above, the starting point of the proof is Theorem \ref{osc-var
martingala2}, where the $L^2$ boundedness of the $\rho$-variation and
oscillation (of a convex combination) of some particular martingales is
stated. So, the first step consists in relating the results on
martingales in Theorem \ref{osc-var martingala2} with the $\rho$-variation and oscillation of 
singular integrals on Lipschitz graphs, and this is the aim
of the following two theorems:

\begin{teo}\label{teorema salts llarg martingala}
Let $\Gamma$ and $\mu$ be as in Theorem \ref{osc-var martingala2}. For each $x\in\Gamma$, define
\begin{equation}\label{definicio Wmu}
W\mu(x)^2:=\sum_{m\in\Z}|(K\varphi_{2^{-m}}*\mu)(x)-E_m\mu(x)|^{2}.
\end{equation}
Then, $\left\|W\mu\right\|^{2}_{L^{2}(\mu)}\leq C_1\sum_{Q\in\DD}\big(\,\alpha_\mu(C_2Q)^{2}+\beta_{2,\mu}(Q)^2\,\big)\mu(Q),$ where $C_1,C_2>0$ depend only on $C_0$, $n$, $d$, $K$, and $\Lip(A)$.
\end{teo}

\begin{teo}\label{teorema salts curt martingala}
Let $\Gamma$ and $\mu$ be as in Theorem \ref{osc-var martingala2}. For each $x\in\Gamma$, define
\begin{equation}\label{definicio Smu}
S\mu(x)^2:=\sup_{\{\epsilon_{m}\}}\sum_{j\in\Z}
\,\sum_{m\in\Z:\,\epsilon_{m},\epsilon_{m+1}\in I_j}
|(K\varphi_{\epsilon_{m+1}}^{\,\epsilon_{m}}*\mu)(x)|^{2},
\end{equation}
where $I_j=[2^{-j-1},2^{-j})$ and the supremum is taken over all decreasing
sequences of positive numbers $\{\epsilon_{m}\}_{m\in\Z}$.
Then, $\left\|S\mu\right\|^{2}_{L^{2}(\mu)}\leq
C\sum_{Q\in\DD}\big(\,\alpha_\mu(Q)^{2}+\beta_{2,\mu}(Q)^2\,\big)\mu(Q)$, where $C>0$ only
depends on $C_0$, $n$, $d$, $K$, and $\Lip(A)$.
\end{teo}

Two fundamental tools to study $W\mu$ and $S\mu$ are the $\alpha$
and $\beta$ coefficients, which will be used to measure the flatness
of $\Gamma$ at different scales, in order to estimate the terms
which appear  in the sums in (\ref{definicio Wmu}) and
(\ref{definicio Smu}). This will be done in sections \ref{s teorema
salts llarg martingala} and \ref{s teorema salts curt martingala}.
To use the $\alpha$ coefficients to relate the $\rho$-variation of
martingales with the $\rho$-variation of singular integrals, it is a
key fact that we are considering a ``smooth'' family like
$\varphi$, because the $\alpha$'s are
defined in terms of Lipschitz functions but $T_\epsilon$ is defined by means of a rough truncation. Moreover, we are taking a
truncation only on the first $n$-coordinates because the average of martingales that we are
using is taken over the parameter $a\in\R^n$, using the v-cubes
$D_M^{\,a}$ (see subsection \ref{ss particular martingale}).

Combining Theorem \ref{teorema salts llarg martingala} and Theorem
\ref{teorema salts curt martingala} with the $L^2$ estimates of the
$\rho$-variation and oscillation on the average of martingales
$E\mu$ in Theorem \ref{osc-var martingala2}, we are able to obtain
local $L^2$ estimates of
$\VV_\rho\circ\TT_{\varphi}^{\HH^n_\Gamma}$ and
$\OO\circ\TT_{\varphi}^{\HH^n_\Gamma}$ when $\Gamma$ is any
Lipschitz graph. More precisely, we separate the sum in the
definition of  $\VV_\rho\circ\TT_{\varphi}^{\HH^n_\Gamma}$
into two parts, which are classically called short and long
variation (and analogously for
$\OO\circ\TT_{\varphi}^{\HH^n_\Gamma}$). The short variation
corresponds to the sum $S\mu$ in Theorem \ref{teorema salts curt
martingala} (here $\mu$ is a suitable modification of
${\HH^n_\Gamma}$), where the indices run over $m\in\Z$ such that
both $\epsilon_m$ and $\epsilon_{m+1}$ lie in the same dyadic
interval, and can be handled using the $\alpha$'s and $\beta$'s. The
long variation corresponds to the sum over the indices $m\in\Z$ such
that $\epsilon_m$ and $\epsilon_{m+1}$ lie in different dyadic
intervals, so one may assume that the $\epsilon_m$'s are dyadic
numbers. It is handled by comparing $K\varphi_{2^{-m}}*\mu$ with
$E_m\mu$, and then using Theorem \ref{teorema salts llarg
martingala} and the fact the $\rho$-variation and oscillation of
$E\mu$ are bounded in $L^2(\mu)$, by Theorem \ref{osc-var
martingala2}. This will be done in section \ref{seccio localitzacio}
(see Theorem \ref{teorema localitzacio}).

Using the local $L^2$ estimates of Theorem \ref{teorema
localitzacio}, combined with rather standard techniques in
Calder\'{o}n-Zygmund theory, in section \ref{seccio Lp suau} we
obtain the $H^1({\HH^n_\Gamma})\to L^1({\HH^n_\Gamma})$ and
$L^\infty({\HH^n_\Gamma})\to BMO({\HH^n_\Gamma})$ boundedness of
$\VV_\rho\circ\TT_{\varphi}^{\HH^n_\Gamma}$ and
$\OO\circ\TT_{\varphi}^{\HH^n_\Gamma}$. Then, by
interpolation, we obtain the $L^p$ boundedness of these operators in
the whole range $1<p<\infty$, and in particular the $L^2$
boundedness (see Theorem \ref{teorema interpolacio}). Moreover, \cite[Theorem B]{CJRW-singular integrals} can be adapted to prove that the $L^2(\HH^n_\Gamma)$ boundedness of $\VV_\rho\circ\TT_{\varphi}^{\HH^n_\Gamma}$ and
$\OO\circ\TT_{\varphi}^{\HH^n_\Gamma}$ also yields the boundedness of these operators from $L^{1}(\HH^n_\Gamma)$ to $L^{1,\infty}(\HH^n_\Gamma)$.

Let us stress that almost all the estimates in the proof of
Theorem \ref{main theorem} (in particular, the constants involved in
the relationships $\lesssim$, $\gtrsim$ and $\approx$) depend
either on $n$, $d$, $K$ or $\Lip(A)$, and possibly on other
variables such as $\rho$ or $p$.

\section{Proof of Theorem \ref{teorema salts llarg martingala}}\label{s teorema salts llarg martingala}

In order to study the difference $(K
\varphi_{2^{-m}}*\mu)(x)-E_m\mu(x)$, we  are going to split
$E_m\mu(x)$ into two parts, the one we will compare with
$(K\varphi_{2^{-m}}*\mu)(x)$ (which corresponds to
integrate, in the definition of $E_m\mu(x)$, over the points
$y\in\R^d$ such that $2^{-m}\lesssim|\widetilde x-\widetilde y|$),
and the remaining part. Then, we will estimate each part of
$(K\varphi_{2^{-m}}*\mu)(x)-E_m\mu(x)$ separately, using
the cancelation properties of the kernel $K$ and the uniform
rectifiability of $\mu$.

Recall from (\ref{def1 Em}) that
$E_{m}\mu(x)=\iint \Lambda^\mu_{m}(x,z\,;\,y)K(z-y)\,d\mu(z)\,d\mu(y).$
Given $\epsilon>0$, we set $\gamma_\epsilon:=1-\varphi_\epsilon$. Then,
\begin{equation*}
\begin{split}
E_{m}\mu(x)&=\iint \varphi_{2^{-m}}(x-y)\Lambda^\mu_{m}(x,z\,;\,y)K(z-y)\,d\mu(z)\,d\mu(y)\\
&\quad+\iint \gamma_{2^{-m}}(x-y)\Lambda^\mu_{m}(x,z\,;\,y)K(z-y)\,d\mu(z)\,d\mu(y).
\end{split}
\end{equation*}

The first term in the previous sum is the one that we will compare
with $(K\varphi_{2^{-m}}*\mu)(x)$. For all $a\in\R^{n}$
such that $x\in D^{\,a}_{m}$, we have
$\supp\,\varphi_{2^{-m}}(x-\cdot)\cap
D^{\,a}_{m}=\emptyset$, and thus
$(K\varphi_{2^{-m}}*\mu)(x)=(K\varphi_{2^{-m}}*(\chi_{(D^{\,a}_{m})^{c}}\mu))(x)$.
Hence, using Fubini's theorem and the definition of
$\Lambda^\mu_{m}(x,z\,;\,y)$,
\begin{equation*}
\begin{split}
(K\varphi_{2^{-m}}*\mu)(x)&=2^{mn}\int_{\{a\,:\,x\in D^{\,a}_{m}\}}(K\varphi_{2^{-m}}*(\chi_{(D^{\,a}_{m})^{c}}\mu))(x)\,da\\
&=2^{mn}\int_{\{a\,:\,x\in D^{\,a}_{m}\}}\mu(D^{\,a}_{m})^{-1}\int_{D^{\,a}_{m}}(K\varphi_{2^{-m}}*(\chi_{(D^{\,a}_{m})^{c}}\mu))(x)\,d\mu(z)\,da\\
&=\iint \varphi_{2^{-m}}(x-y)\Lambda^\mu_{m}(x,z\,;\,y)K(x-y)\,d\mu(z)\,d\mu(y).
\end{split}
\end{equation*}

We can decompose $(K\varphi_{2^{-m}}*\mu)(x)-E_m\mu(x)$ as
\begin{equation}\label{R-E eq1}
\begin{split}
(K\varphi_{2^{-m}}&*\mu)(x)-E_m\mu(x)\\
&=\iint \varphi_{2^{-m}}(x-y)\Lambda^\mu_{m}(x,z\,;\,y)(K(x-y)-K(z-y))\,d\mu(z)\,d\mu(y)\\
&\quad-\iint \gamma_{2^{-m}}(x-y)\Lambda^\mu_{m}(x,z\,;\,y)K(z-y)\,d\mu(z)\,d\mu(y)\\
&=\sum_{j<m}F_{j}^{m}(x)-\sum_{j\in\Z}G_{j}^{m}(x),
\end{split}
\end{equation}
where
\begin{gather}
F_{j}^{m}(x):=\iint \varphi^{\,2^{-j}}_{2^{-j-1}}(x-y)\Lambda^\mu_{m}(x,z\,;\,y)
(K(x-y)-K(z-y))\,d\mu(z)\,d\mu(y),\\
G_{j}^{m}(x):=\iint\varphi^{\,2^{-j}}_{2^{-j-1}}(z-y)\gamma_{2^{-m}}(x-y)
\Lambda^\mu_{m}(x,z\,;\,y)K(z-y)\,d\mu(z)\,d\mu(y).\label{R-E eq5}
\end{gather}

Fix a v-cube $D\in\DD_m$, for some $m\in\Z$. In subsection \ref{F s1} (see (\ref{F estimacio final})) we will prove that
\begin{equation}\label{F estimacio final2}
\sum_{j<m}|F_{j}^{m}(x)|\lesssim\frac{\dist(x,L_D)}{\ell(D)}+\sum_{Q\in\DD\,:\,D\subset Q}\frac{\ell(D)}{\ell(Q)}\,\alpha(Q)
\end{equation}
for all $x\in D\cap\Gamma$, where $L_D$ denotes an $n$-plane that minimizes $\alpha(D)$, and in subsection \ref{G s1} (see (\ref{G estimacio final})) we will prove that there exists a constant $C_b>1$ such that
\begin{equation}\label{G estimacio final2}
\sum_{j\in\Z}|G_j^m(x)|\lesssim\alpha(C_bD)+
\sum_{Q\in\DD\,:\,Q\subset C_bD}\frac{\ell(Q)^{n+1}}{\ell(D)^{n+1}}\,\alpha(Q)
\end{equation}
for all $x\in D\cap\Gamma$. Assuming that these estimates hold, by (\ref{R-E eq1}),
\begin{equation}\label{R-E eq6}
\begin{split}
\|W\mu\|^{2}_{L^{2}(\mu)}&=\sum_{m\in\Z}\sum_{D\in\DD_m}\int_D|(K\varphi_{2^{-m}}*\mu)(x)-E_m\mu(x)|^{2}\,d\mu(x)\\
&\lesssim\sum_{D\in\DD}\int_D\bigg(\frac{\dist(x,L_D)}{\ell(D)}\bigg)^2\,d\mu(x)
+\sum_{D\in\DD}\bigg(\sum_{\begin{subarray}{c}Q\in\DD:\\D\subset Q\end{subarray}}\frac{\ell(D)}{\ell(Q)}\,\alpha(Q)\bigg)^2\mu(D)\\
&\quad+\sum_{D\in\DD}\alpha(C_bD)^2\mu(D)+\sum_{D\in\DD}\bigg(
\sum_{\begin{subarray}{c}Q\in\DD:\\Q\subset C_bD\end{subarray}}\frac{\ell(Q)^{n+1}}{\ell(D)^{n+1}}\,\alpha(Q)\bigg)^2\mu(D)\\
&=:W_1\mu+W_2\mu+W_3\mu+W_4\mu.
\end{split}
\end{equation}

If $L_D^1$ and $L_D^2$ denote a minimizing $n$-plane for $\beta_1(D)$ and $\beta_2(D)$, respectively, one can show that
$\dist_\HH(L_D\cap C_\Gamma D,L_D^1\cap C_\Gamma D)\lesssim\alpha(D)\ell(D)$ and $\dist_\HH(L_D^1\cap C_\Gamma D,L_D^2\cap C_\Gamma D)\lesssim\beta_2(D)\ell(D)$. This easily implies that, for $x\in D\cap\Gamma$, $\dist(x,L_D)\lesssim\dist(x,L^2_D)+\beta_2(D)\ell(D)+\alpha(D)\ell(D)$, so  $W_1\mu\lesssim\sum_{D\in\DD}(\alpha(D)^2+\beta_2(D)^2)\mu(D).$

By Cauchy-Schwarz inequality,
\begin{equation*}
\begin{split}
W_2\mu&\leq\sum_{D\in\DD}\mu(D)
\bigg(\sum_{Q\in\DD:\,D\subset Q}\frac{\ell(D)}{\ell(Q)}\,\alpha(Q)^2\bigg)
\bigg(\sum_{Q\in\DD:\,D\subset Q}\frac{\ell(D)}{\ell(Q)}\bigg)\\
&\approx\sum_{D\in\DD}\,\sum_{Q\in\DD:\,D\subset Q}\frac{\ell(D)^{n+1}}{\ell(Q)}\,\alpha(Q)^{2}\approx
\sum_{Q\in\DD}\alpha(Q)^{2}\mu(Q),
\end{split}
\end{equation*}
and also
\begin{equation*}
\begin{split}
W_4\mu&\leq\sum_{D\in\DD}\mu(D)
\bigg(\sum_{Q\in\DD:\,Q\subset C_bD}\frac{\ell(Q)^{n+1}}{\ell(D)^{n+1}}\,\alpha(Q)^2\bigg)
\bigg(\sum_{Q\in\DD:\,Q\subset C_bD}\frac{\ell(Q)^{n+1}}{\ell(D)^{n+1}}\bigg)\\
&\approx\sum_{D\in\DD}\,\sum_{Q\in\DD:\,Q\subset C_bD}\ell(Q)^n\frac{\ell(Q)}{\ell(D)}\,\alpha(Q)^2
\lesssim\sum_{Q\in\DD}\alpha(Q)^2\mu(Q).\\
\end{split}
\end{equation*}
Therefore, using (\ref{R-E eq6}) and that $\alpha(Q)\lesssim\alpha(C_b Q)$, we conclude that
$$\|W\mu\|^{2}_{L^{2}(\mu)}\lesssim\sum_{Q\in\DD}(\alpha(C_bQ)^2+\beta_2(Q)^2)\mu(Q),$$
and the theorem follows. It only remains to prove (\ref{F estimacio final2}) and (\ref{G estimacio final2}).

\subsection{Estimate of $\sum_{j<m}F_{j}^{m}(x)$ when $x\in D\cap\Gamma$ for some $D\in\DD_m$}\label{F s1}
Assume that $x\in D\cap\Gamma$ for some $D\in\DD_{m}$ and $j<m$.
Let $L_{D}$ be an $n$-plane that minimizes $\alpha(D)$ and let $\sigma_{D}:=c_{D}\HH^{n}_{L_{D}}$ be a minimizing measure of $\alpha(D)$. Let $L_D^x$ be the $n$-plane parallel to $L_D$ that contains $x$ and set $\sigma^x_{D}:=c_{D}\HH^{n}_{L^x_{D}}$.

Notice that, because of $x\in L_D^x$, the antisymmetry of $\varphi^{\,2^{-j}}_{2^{-j-1}}K$, and since $j<m$ (so, if $x\in D^{\,a}_m$ and $y\in\supp\varphi^{\,2^{-j}}_{2^{-j-1}}(x-\cdot)$, then $y\notin D^{\,a}_m$), we have
\begin{equation}\label{eq1}
\begin{split}
0&=\int\varphi^{\,2^{-j}}_{2^{-j-1}}(x-y)K(x-y)\,d\sigma_{D}^x(y)\\
&=\int_{\{a\,:\,x\in D^{\,a}_{m}\}}\frac{2^{mn}}{\sigma_D^x(D^{\,a}_{m})}\int_{D^{\,a}_{m}}\int_{(D^{\,a}_{m})^c}
\varphi^{\,2^{-j}}_{2^{-j-1}}(x-y)K(x-y)\,d\sigma_{D}^x(y)\,d\sigma_{D}^x(z)\,da\\
&=\iint\varphi^{\,2^{-j}}_{2^{-j-1}}(x-y)\Lambda_m^{\sigma_D^x}(x,z\,;\,y)K(x-y)\,d\sigma_{D}^x(z)\,d\sigma_{D}^x(y).
\end{split}
\end{equation}

Given $a\in\R^n$, let $b:=a+\{2^{-m-1}\}^{n}\in\R^n$ be the center of $\wit D_m^{\,a}$. For $u\in\R^n$ we denote $\|u\|_\infty:=\max_{i=1,\ldots,n}|u^i|$. Then, given $t\in\R^d$, it is clear that $t\in \overline{D_m^{\,a}}$ if and only if $\|\tilde t-b\|_\infty\leq2^{-m}$. Using that $\sigma_D^x$ is a Hausdorff measure on an $n$-plane, that $K$ is antisymmetric and that $\varphi^{\,2^{-j}}_{2^{-j-1}}$ is symmetric, one can show that
\begin{equation*}
\begin{split}
0&=\int_{\|\wit x-b\|_\infty\leq2^{-m}}
\int_{\|\wit z-b\|_\infty\leq2^{-m}}\int_{\|\wit y-b\|_\infty>2^{-m}}
\varphi^{\,2^{-j}}_{2^{-j-1}}(x-y)K(z-y)
\,d\sigma_{D}^x(y)\,d\sigma_{D}^x(z)\,db.
\end{split}
\end{equation*}
By the change of variable $b=a+\{2^{-m-1}\}^{n}$, it is easy to see
that this triple integral is equal to  $\int_{\{a\,:\,x\in
D^{\,a}_{m}\}}\int_{D^{\,a}_{m}}\int_{(D^{\,a}_{m})^c}\varphi^{\,2^{-j}}_{2^{-j-1}}(x-y)K(z-y)
\,d\sigma_{D}^x(y)\,d\sigma_{D}^x(z)\,da$. Thus, since
$\sigma_D^x(D^{\,a}_{m})$ does not depend on $a\in\R^n$ because
$\sigma_D^x$ is flat,
\begin{equation}\label{eq2}
\begin{split}
0&=\int_{\{a\,:\,x\in D^{\,a}_{m}\}}\frac{2^{mn}}{\sigma_D^x(D^{\,a}_{m})}\int_{D^{\,a}_{m}}\int_{(D^{\,a}_{m})^c}\varphi^{\,2^{-j}}_{2^{-j-1}}(x-y)K(z-y)
\,d\sigma_{D}^x(y)\,d\sigma_{D}^x(z)\,da\\
&=\iint\varphi^{\,2^{-j}}_{2^{-j-1}}(x-y)\Lambda_m^{\sigma_D^x}(x,z\,;\,y)K(z-y)\,d\sigma_{D}^x(z)\,d\sigma_{D}^x(y).
\end{split}
\end{equation}

By (\ref{eq1}) and (\ref{eq2}), we conclude that
\begin{equation}\label{eq3}
0=\iint\varphi^{\,2^{-j}}_{2^{-j-1}}(x-y)\Lambda_m^{\sigma_D^x}(x,z\,;\,y)(K(x-y)-K(z-y))
\,d\sigma_{D}^x(z)\,d\sigma_{D}^x(y).
\end{equation}

By definition, it is clear that $\Lambda_m^{\sigma_D^x}(x,z\,;\,y)=\Lambda_m^{\sigma_D}(x,z\,;\,y)$. Therefore, using (\ref{eq3}), we can decompose
\begin{equation}\label{F eq9}
F_j^m(x)=F1_j^m(x)+F2_j^m(x)+F3_j^m(x)+F4_j^m(x),
\end{equation}
where
\begin{multline}\label{F eq5}
F1_j^m(x):=\iint\varphi^{\,2^{-j}}_{2^{-j-1}}(x-y)\Lambda_m^{\mu}(x,z\,;\,y)\\
(K(x-y)-K(z-y))\,d(\mu-\sigma_D)(z)\,d\mu(y),
\end{multline}
\vskip-20pt
\begin{multline}\label{F eq6}
F2_j^m(x):=\iint\varphi^{\,2^{-j}}_{2^{-j-1}}(x-y)\Lambda_m^{\mu}(x,z\,;\,y)\\
(K(x-y)-K(z-y))\,d\sigma_D(z)\,d(\mu-\sigma_D)(y),
\end{multline}
\vskip-20pt
\begin{multline}
F3_j^m(x):=\iint\varphi^{\,2^{-j}}_{2^{-j-1}}(x-y)
(\Lambda_m^{\mu}(x,z\,;\,y)-\Lambda_m^{\sigma_D}(x,z\,;\,y))\\
(K(x-y)-K(z-y))\,d\sigma_D(z)\,d\sigma_D(y),
\end{multline}
\vskip-20pt
\begin{multline}\label{F eq8}
F4_j^m(x):=\iint\varphi^{\,2^{-j}}_{2^{-j-1}}(x-y)\Lambda_m^{\sigma_D^x}(x,z\,;\,y)\\
(K(x-y)-K(z-y))\,d(\sigma_D\times\sigma_D-\sigma_D^x\times\sigma_D^x)(z,y).
\end{multline}

In the next subsections we will prove the following estimates:
\begin{gather}
|F1^m_j(x)|+|F3^m_j(x)|\lesssim2^{j-m}\alpha(D), \label{F eq10}\\
|F2^m_j(x)|\lesssim2^{j-m}\sum_{Q\in\DD\,:\,D\subset Q,\,\ell(Q)\leq2^{-j}}\alpha(Q), \label{F eq11}\\
|F4^m_j(x)|\lesssim2^{j-m}\,\frac{\dist(x,L_D)}{\ell(D)}. \label{F eq13}
\end{gather}
Then, using (\ref{F eq9}), we will finally get that, for all $D\in\DD_m$ and $x\in D\cap\Gamma$,
\begin{equation}\label{F estimacio final}
\begin{split}
\sum_{j<m}|F_{j}^{m}(x)|&\lesssim\frac{\dist(x,L_D)}{\ell(D)}+\sum_{j\leq m}2^{j-m}\sum_{Q\in\DD:\,
D\subset Q,\,\ell(Q)\leq2^{-j}}\alpha(Q)\\
&\lesssim\frac{\dist(x,L_D)}{\ell(D)}+\sum_{Q\in\DD\,:\,D\subset Q}\frac{\ell(D)}{\ell(Q)}\,\alpha(Q),
\end{split}
\end{equation}
which gives (\ref{F estimacio final2}).

\subsubsection{{\bf Estimate of $F1_{j}^{m}(x)$}}\label{F ss1}
Notice that, if $|\widetilde x-\widetilde z|>2^{-m}\sqrt{n}$, there is no $a\in\R^n$ such that $x,z\in D_m^a$, and this means that $\Lambda_m^{\mu}(x,z\,;\,y)=0$. Thus, we can assume that $|\widetilde x-\widetilde z|\leq2^{-m}\sqrt{n}$. Therefore, if the constant $C_\Gamma$ (see the definition of the $\alpha$'s in subsection \ref{ss alpha i beta}) is big enough, $\supp\Lambda_m^{\mu}(x,\cdot\,;\,y)\subset B_D$.

For $y,z\in\Gamma$ such that $y\in\supp\varphi^{\,2^{-j}}_{2^{-j-1}}(x-\cdot)$, $j<m$ and $|\widetilde x-\widetilde z|\leq2^{-m}\sqrt{n}$ (so, in particular, $|x-z|\lesssim|x-y|$), we have the following estimates:
\begin{gather*}
|K(x-y)-K(z-y)|\lesssim|x-z||x-y|^{-n-1}\lesssim2^{j(n+1)-m},\\
|\nabla_{z}(K(x-y)-K(z-y))|=|\nabla_{z}K(z-y)|\lesssim2^{j(n+1)}.
\end{gather*}
\begin{claim}\label{F1 eq3}
We have $|\Lambda_m^{\mu}(x,z\,;\,y)|\lesssim2^{mn}$ and $|\nabla_z\Lambda_m^{\mu}(x,z\,;\,y)|\lesssim2^{m(n+1)}$ for all $x,y,z\in\R^d$.
\end{claim}
Claim \ref{F1 eq3} and the subsequent ones \ref{F2 eq2},\ldots,\ref{G5 eq4} will be proved in subsection \ref{subseccio claims} below.

Putting all these estimates together we obtain that
\begin{equation*}
\Big|\nabla_{z}\Big(\Lambda_m^{\mu}(x,z\,;\,y)
(K(x-y)-K(z-y))\Big)\Big|\lesssim2^{j(n+1)+mn},
\end{equation*}
and, since $\supp\Lambda_m^{\mu}(x,\cdot\,;\,y)\subset B_D$, recalling the definition of $\dist_{B_D}$ in (\ref{wasserstein dist}),
\begin{equation*}
\left|\int\Lambda_m^{\mu}(x,z\,;\,y)(K(x-y)-K(z-y))\,d(\mu-\sigma_D)(z)\right|
\lesssim2^{j(n+1)+mn}\dist_{B_{D}}(\mu,\sigma_{D}).
\end{equation*}
We can use this last estimate in (\ref{F eq5}) to obtain
\begin{equation*}
\begin{split}
|F1_j^m(x)|&\lesssim2^{j(n+1)+mn}\dist_{B_{D}}(\mu,\sigma_{D})\int\varphi^{\,2^{-j}}_{2^{-j-1}}(x-y)\,d\mu(y)\\
&\lesssim2^{j+mn}\dist_{B_{D}}(\mu,\sigma_{D})
\approx2^{j-m}\,\ell(D)^{-n-1}\dist_{B_{D}}(\mu,\sigma_{D})\lesssim2^{j-m}\alpha(D),
\end{split}
\end{equation*}
which, together with the estimate of $|F3_j^m(x)|$ in subsection \ref{F ss3}, gives (\ref{F eq10}).

\subsubsection{{\bf Estimate of $F2_{j}^{m}(x)$}}\label{F ss2}
Arguing as in subsection \ref{F ss1}, we can obtain the following estimates for $x,y,z$ as above:
\begin{gather}
|\varphi^{\,2^{-j}}_{2^{-j-1}}(x-y)|\leq1\quad\text{and}\quad
|\nabla_y\varphi^{\,2^{-j}}_{2^{-j-1}}(x-y)|\lesssim2^j,\label{estimacio varphi}\\
|K(x-y)-K(z-y)|\lesssim|x-z||x-y|^{-n-1}\lesssim2^{j(n+1)-m},\\
|\nabla_{y}(K(x-y)-K(z-y))|\lesssim|x-z||x-y|^{-n-2}\lesssim2^{j(n+2)-m}.
\end{gather}

\begin{claim}\label{F2 eq2}
For $j<m$,
$y\in\supp\varphi^{\,2^{-j}}_{2^{-j-1}}(x-\cdot)$, and
$|\widetilde x-\widetilde z|\leq2^{-m}\sqrt{n}$, the following hold:
$|\Lambda_m^{\mu}(x,z\,;\,y)|\lesssim2^{mn}$ and
$\nabla_y\Lambda_m^{\mu}(x,z\,;\,y)=0.$
\end{claim}
Notice that the first estimate in Claim \ref{F2 eq2} is the same as
the  first one in Claim \ref{F1 eq3}.

Let $D_{j}\in \DD_{j}$ be the unique dyadic v-cube with
$\ell(D_j)=2^{-j}$  which contains $D$. Then,
$\supp\varphi^{\,2^{-j}}_{2^{-j-1}}(x-\cdot)\subset
B_{D_j}$ for $C_\Gamma$ big enough. Therefore, we can use the
previous estimates to see that the gradient of the term inside the integral with respect to $y$ in (\ref{F eq6}) is bounded by $2^{j(n+2)+m(n-1)}$ and is supported in $B_{D_j}$, and then by (\ref{wasserstein dist}) we derive that
\begin{equation}\label{F2 eq3}
\begin{split}
|F2_j^m(x)|&\leq\int\left|\int\varphi^{\,2^{-j}}_{2^{-j-1}}(x-y)\Lambda_m^{\mu}(x,z\,;\,y)\right.\\ &\qquad\qquad\qquad\qquad\qquad(K(x-y)-K(z-y))\,d(\mu-\sigma_D)(y)\bigg|\,d\sigma_D(z)\\
&\lesssim\int_{|\widetilde x-\widetilde z|\leq2^{-m}\sqrt{n}}
2^{j(n+2)+m(n-1)}\dist_{B_{D_j}}(\mu,\sigma_{D})\,d\sigma_D(z)\\
&\lesssim2^{j(n+2)-m}\dist_{B_{D_j}}(\mu,\sigma_{D}).
\end{split}
\end{equation}

We shall estimate $\dist_{B_{D_j}}(\mu,\sigma_{D})$ in terms of the
$\alpha$ coeficients. Consider the unique sequence of dyadic v-cubes
$D=:D_{m}\subset\ldots\subset D_{i+1}\subset
D_{i}\subset\ldots\subset D_{j}$ such that each $D_{i}$ belongs to
$\DD_{i}$, for $i=j,\ldots,m$. Let $L_{D_i}$ be an $n$-plane that
gives the minimum in the definition of $\alpha(D_{i})$ and let
$\sigma_{D_i}:=c_{D_i}d\HH^{n}_{L_{D_i}}$ be a minimizing measure.
We will prove that
\begin{equation}\label{F2 eq4}
\dist_{B_{D_{j}}}(\mu,\sigma_{D})\lesssim2^{-j(n+1)}\sum_{i=j}^{m-1}\alpha(D_{i}).
\end{equation}
Combining (\ref{F2 eq4}) with (\ref{F2 eq3}), we will finally obtain that
$|F2_{j}^{m}(x)|\lesssim2^{j-m}\sum_{i=j}^{m-1}\alpha(D_{i}),$
which gives (\ref{F eq11}).

Let us prove (\ref{F2 eq4}). By the triangle inequality,
\begin{eqnarray*}
\dist_{B_{D_{j}}}(\mu,\sigma_{D})\leq&\dist_{B_{D_{j}}}(\mu,\sigma_{D_j})+\sum_{i=j}^{m-1}\dist_{B_{D_{j}}}(\sigma_{D_i},\sigma_{D_{i+1}})\\
\lesssim&2^{-j(n+1)}\alpha(D_{j})+\sum_{i=j}^{m-1}\dist_{B_{D_{j}}}(\sigma_{D_i},\sigma_{D_{i+1}}),
\end{eqnarray*}
so we are reduced to prove that, for all $i=j,\ldots,m-1$,
\begin{equation}\label{F2 eq6}
\dist_{B_{D_{j}}}(\sigma_{D_i},\sigma_{D_{i+1}})\lesssim2^{-j(n+1)}\alpha(D_{i}).
\end{equation}

By definition, $\dist_{B_{D_{j}}}(\sigma_{D_i},\sigma_{D_{i+1}})=\sup\big|\int g\,d(c_{D_i}\HH^{n}_{L_{D_i}}-c_{D_{i+1}}\HH^{n}_{L_{D_{i+1}}})\big|,$ where the supremum is taken over all Lipschitz functions $g$ supported in $B_{D_{j}}$ such that $\Lip(g)\leq1$. Fix one of such Lipschitz functions $g$. Then,
\begin{equation}\label{F2 eq7}
\begin{split}
\int g\,d(c_{D_i}\HH^{n}_{L_{D_i}}-c_{D_{i+1}}\HH^{n}_{L_{D_{i+1}}})
&=(c_{D_i}-c_{D_{i+1}})\int g\,d\HH^{n}_{L_{D_i}}\\
&\quad+c_{D_{i+1}}\int g\,d(\HH^{n}_{L_{D_i}}-\HH^{n}_{L_{D_{i+1}}}).
\end{split}
\end{equation}
It is shown in \cite[Lemma 3.4]{To} that
$|c_{D_i}-c_{D_{i+1}}|\lesssim\alpha(D_{i})$,  so the first term on
the right hand side of (\ref{F2 eq7}) is bounded in absolute value
by $C2^{-j(n+1)}\alpha(D_{i})$.

In order to estimate the second term of the right hand side of
(\ref{F2 eq7}), set $L_{D_{i+1}}=\{(\widetilde t, a(\widetilde
t))\in\R^d\,:\,\widetilde t\in\R^n\}$ (where $a:\R^n\to\R^{d-n}$ is
an appropriate affine map), and let $p:L_{D_i}\to L_{D_{i+1}}$ be
the projection defined by $p(t):=(\widetilde t,a(\widetilde t))$.
Since $\Gamma$ is a Lipschitz graph, $a$ is well defined and $p$ is
a homeomorphism. Let $p_{\sharp}\HH^{n}_{L_{D_i}}$ be the image
measure of $\HH^{n}_{L_{D_i}}$ by $p$. It is easy to check that
$\HH^{n}_{L_{D_{i+1}}}=\tau p_{\sharp}\HH^{n}_{L_{D_i}}$, where
$\tau$ is some positive constant such that
$|\tau-1|\lesssim\alpha(D_{i})$ and $\tau\lesssim1$. Therefore,
\begin{equation}\label{F2 eq8}
\begin{split}
\bigg|\int g\,d(\HH^{n}_{L_{D_i}}-&\HH^{n}_{L_{D_{i+1}}})\bigg|=\left|\int (g(t)-\tau g(p(t))\,)\,d\HH^{n}_{L_{D_i}}(t)\right|\\
&\leq\left|\int (1-\tau)g(t)\,d\HH^{n}_{L_{D_i}}(t)\right|+\left|\int\tau(g(t)-g(p(t))\,)\,d\HH^{n}_{L_{D_i}}(t)\right|\\
&\lesssim2^{-j(n+1)}\alpha(D_{i})+\int|(g(t)-g(p(t))|\,d\HH^{n}_{L_{D_i}}(t).\\
\end{split}
\end{equation}

Since $g$ and $g\circ p$ are supported in $B_{D_j}$ and $g$ is
1-Lipschitz, by \cite[Lemma 3.4]{To},
\begin{equation*}
\begin{split}
\int|(g-g\circ p)|\,d\HH^{n}_{L_{D_i}}
&\lesssim\int_{B_{D_{j}}}\dist_{\HH}(L_{D_i}\cap B_{D_{j}},L_{D_{i+1}}\cap B_{D_{j}})\,d\HH^{n}_{L_{D_i}}\\
&\lesssim\,\,2^{-jn}\dist_{\HH}(L_{D_i}\cap B_{D_{j}},L_{D_{i+1}}\cap B_{D_{j}})\\
&\lesssim\,\,2^{-jn}2^{i-j}\dist_{\HH}(L_{D_i}\cap B_{D_{i}},L_{D_{i+1}}\cap B_{D_{i}})\lesssim\,\,2^{-j(n+1)}\alpha(D_{i}).
\end{split}
\end{equation*}
This last estimate together with (\ref{F2 eq8}) and the fact  that
$|c_{D_{i+1}}|\lesssim1$ implies that the second term on the right
hand side of (\ref{F2 eq7}) is also bounded in absolute value by
$C2^{-j(n+1)}\alpha(D_{i})$. Therefore, to obtain (\ref{F2 eq6}) we
only have to take the supremum in (\ref{F2 eq7}) over all admissible
functions $g$.

\subsubsection{{\bf Estimate of $F3_{j}^{m}(x)$}}\label{F ss3}
Notice that, by Fubini's theorem,
\begin{equation*}
\begin{split}
\Lambda_m^{\mu}(x,z\,;\,y)&-\Lambda_m^{\sigma_D}(x,z\,;\,y)=2^{mn}\int_{\left\{a\,:\,x,z\in D^{\,a}_{m},\,y\notin D^{\,a}_{m}\right\}}\left(\frac{1}{\mu(D^{\,a}_{m})}-\frac{1}{\sigma_{D}(D^{\,a}_{m})}\right)\,da\\
&=2^{mn}\int_{\left\{a\,:\,x,z\in D^{\,a}_{m},\,y\notin D^{\,a}_{m}\right\}}\frac{\sigma_{D}(D^{\,a}_{m})-\mu(D^{\,a}_{m})}{\mu(D^{\,a}_{m})\sigma_{D}(D^{\,a}_{m})}\,da\\
&=2^{mn}\int_{\left\{a\,:\,x,z\in D^{\,a}_{m},\,y\notin D^{\,a}_{m}\right\}}
\bigg(\int_{t\in D^{\,a}_{m}}\,d(\sigma_{D}-\mu)(t)\bigg)\mu(D^{\,a}_{m})^{-1}\sigma_{D}(D^{\,a}_{m})^{-1}\,da\\
&=\int\Lambda^{\mu,\sigma_D}_m(x,z,t\,;\,y)\,d(\sigma_{D}-\mu)(t).
\end{split}
\end{equation*}
Since $\Lambda^{\mu,\sigma_D}_m(x,z,t\,;\,y)=0$ if $|\widetilde x-\widetilde t|>2^{-m}\sqrt{n}$, we may assume that
$\supp\Lambda^{\mu,\sigma_D}_m(x,z,\cdot\,;\,y)\subset B_D$ (by taking $C_\Gamma$ big enough).
\begin{claim}\label{F3 eq2}
We have $|\Lambda_m^{\mu,\sigma_D}(x,z,t\,;\,y)|\lesssim2^{2mn}$ and $|\nabla_t\Lambda_m^{\mu,\sigma_D}(x,z,t\,;\,y)|\lesssim2^{m(2n+1)}$ for all $x,y,z,t\in\R^d$.
\end{claim}
Using Claim \ref{F3 eq2}, we deduce that
$|\Lambda_m^{\mu}(x,z\,;\,y)-\Lambda_m^{\sigma_D}(x,z\,;\,y)|\lesssim
2^{m(2n+1)}\,\dist_{B_{D}}(\mu,\sigma_{D}),$
and then,
\begin{equation*}
\begin{split}
|F3_{j}^{m}(x)|&\lesssim\int\varphi^{\,2^{-j}}_{2^{-j-1}}(x-y)\int_{|\widetilde x-\widetilde z|\leq2^{-m}\sqrt{n}}
2^{m(2n+1)}\,\dist_{B_{D}}(\mu,\sigma_{D})\\
&\quad\qquad\qquad\qquad\qquad\qquad\qquad\qquad|x-z||x-y|^{-n-1}\,d\sigma_D(z)\,d\sigma_D(y)\\
&\lesssim2^{2mn+j(n+1)}\,\dist_{B_{D}}(\mu,\sigma_{D})\iint_{\begin{subarray}{l}|\widetilde x-\widetilde y|\leq2^{-j}3\sqrt{n}\\|\widetilde x-\widetilde z|\leq2^{-m}\sqrt{n}\end{subarray}}d\sigma_D(z)\,d\sigma_D(y)\\
&\lesssim2^{mn+j}\,\dist_{B_{D}}(\mu,\sigma_{D})\lesssim2^{j-m}\alpha(D),
\end{split}
\end{equation*}
which, together with the estimate of $|F1_j^m(x)|$ (see subsection \ref{F ss1}), gives (\ref{F eq10}).

\subsubsection{{\bf Estimate of $F4_{j}^{m}(x)$}}
Set $L_{D}=\{(\widetilde y, a(\widetilde y))\in\R^d\,:\,\widetilde
y\in\R^n\}$,  where $a:\R^n\to\R^{d-n}$ is an appropriate affine
map, and let $p:L_{D}^x\to L_{D}$ be the projection defined by
$p(y):=(\widetilde y,a(\widetilde y))$. Since $\Gamma$ is a
Lipschitz graph, $a$ is well defined and $p$ is a homeomorphism. If
$p_\sharp\sigma^x_{D}$ is the image measure of $\sigma^x_{D}$ under
$p$, we obviously have $\sigma_D=p_\sharp\sigma^x_{D}$ because $L_D$
and $L_D^x$ differ by a translation. Therefore, since
$\wit{p(y)}=\wit y$, (\ref{F eq8}) becomes
\begin{multline*}
F4_j^m(x)=\iint\big(K(x-p(y))-K(p(z)-p(y))-(K(x-y)-K(z-y))\big)\\
\varphi^{\,2^{-j}}_{2^{-j-1}}(x-y)\Lambda_m^{\sigma_D^x}(x,z\,;\,y)\,d\sigma_D^x(z)\,d\sigma_D^x(y).
\end{multline*}
For $y,z\in L_D^x$ such that $\varphi^{\,2^{-j}}_{2^{-j-1}}(x-y)\Lambda_m^{\sigma_D^x}(x,z\,;\,y)\neq0$, we have $K(p(z)-p(y))-K(z-y)=0$, so we can estimate
\begin{equation*}
\begin{split}
|K(x-p(y))-K(p(z)-p(y))&-(K(x-y)-K(z-y))|=|K(x-p(y))-K(x-y)|\\
&\lesssim\frac{|y-p(y)|}{|x-y|^{n+1}}
\lesssim2^{j(n+1)}|y-p(y)|\approx2^{j(n+1)}\dist(x,L_D).
\end{split}
\end{equation*}

By the same arguments as in the proof of Claim \ref{F1 eq3}, one  can
easily see that $|\Lambda_m^{\sigma_D^x}(x,z\,;\,y)|\lesssim2^{mn}$.
Therefore,
\begin{equation*}
\begin{split}
|F4_j^m(x)|&\lesssim2^{j(n+1)}\dist(x,L_D)2^{mn}
\iint_{\begin{subarray}{l}|\widetilde x-\widetilde y|\leq2^{-j}3\sqrt{n}\\|\widetilde x-\widetilde z|\leq2^{-m}\sqrt{n}\end{subarray}}d\sigma_D^x(z)\,d\sigma_D^x(y)\\
&\lesssim2^{j}\,\dist(x,L_D)\approx2^{j-m}\,\dist(x,L_D)/\ell(D),
\end{split}
\end{equation*}
which gives (\ref{F eq13}).

\subsection{Estimate of $\sum_{j\in\Z}G_{j}^{m}(x)$ when $x\in D\cap\Gamma$ for some $D\in\DD_m$}\label{G s1}
Assume that $x\in D$ for some $D\in\DD_{m}$. Recall from (\ref{R-E eq5}) that
\begin{equation*}
\begin{split}
G_{j}^{m}(x)=
\iint\varphi^{\,2^{-j}}_{2^{-j-1}}(z-y)&\gamma_{2^{-m}}(x-y)
\Lambda^\mu_{m}(x,z\,;\,y)K(z-y)\,d\mu(z)\,d\mu(y),
\end{split}
\end{equation*}
where $0\leq\gamma_{2^{-m}}(x-y)\leq1$, $|\nabla_y\gamma_{2^{-m}}(x-y)|\lesssim2^m$ for all $x,y\in\R^d$, and $\gamma_{2^{-m}}(x-y)=0$ whenever $|\widetilde x-\widetilde y|>2^{-m}3\sqrt{n}$. Notice that $\Lambda^\mu_{m}(x,z\,;\,y)=0$ if $|\widetilde x-\widetilde z|>2^{-m}\sqrt{n}$, thus we can assume that $|\widetilde x-\widetilde z|\leq 2^{-m}\sqrt{n}$ and $|\widetilde z-\widetilde y|\leq2^{-m+2}\sqrt{n}$ in the integral that defines $G_j^m(x)$. Hence, if $j\leq m-2$, $\varphi^{\,2^{-j}}_{2^{-j-1}}(z-y)\Lambda^\mu_{m}(x,z\,;\,y)=0$ for all $z,y\in\R^d$, because
$\varphi^{\,2^{-j}}_{2^{-j-1}}(z-y)=0$ if $|\widetilde z-\widetilde y|\leq2^{-j-1}2.1\sqrt{n}$, and $2^{-j-1}2.1\sqrt{n}\geq2^{-m+2}\sqrt{n}$ when $j\leq m-2$. Therefore, $G_j^m(x)=0$ for $j\leq m-2$, and then
\begin{equation}\label{G eq14}
\sum_{j\in\Z}G_{j}^{m}(x)=\sum_{j\geq m-1}G_{j}^{m}(x);
\end{equation}
so, from now on, we assume that $j\geq m-1$.

Let $L_{D}$ be an $n$-plane that minimizes $\alpha(D)$ and let
$\sigma_{D}:=c_{D}\HH^{n}_{L_{D}}$ be a minimizing measure of
$\alpha(D)$. As we did in the beginning of subsection \ref{F s1},
given $a\in\R^n$, let $b:=a+\{2^{-m-1}\}^{n}\in\R^n$ be the center
of $\wit D_m^{\,a}$. Recall that, for $t\in\R^d$, $t\in
\overline{D_m^{\,a}}$ if and only if $\|\tilde
t-b\|_\infty\leq2^{-m}$. Using that $\sigma_D$ is a Hausdorff
measure on an $n$-plane, that $K$ is antisymmetric and that
$\varphi^{\,2^{-j}}_{2^{-j-1}}$ and $\gamma_{2^{-m}}$ are
symmetric, one can show that
\begin{equation*}
\begin{split}
0=\int_{\|\wit x-b\|_\infty\leq2^{-m}}
\int_{\|\wit z-b\|_\infty\leq2^{-m}}\int_{\|\wit y-b\|_\infty>2^{-m}}
&\varphi^{\,2^{-j}}_{2^{-j-1}}(z-y)\\
&\gamma_{2^{-m}}(x-y)K(z-y)
\,d\sigma_{D}(y)\,d\sigma_{D}(z)\,db.
\end{split}
\end{equation*}
By the change of variable $b=a+\{2^{-m-1}\}^{n}$, it is easy to see
that this triple integral is equal to  $\int_{\{a\,:\,x\in
D^{\,a}_{m}\}}\int_{D^{\,a}_{m}}\int_{(D^{\,a}_{m})^c}\varphi^{\,2^{-j}}_{2^{-j-1}}(z-y)\gamma_{2^{-m}}(x-y)
K(z-y)\,d\sigma_{D}(y)\,d\sigma_{D}(z)\,da$. Thus, since
$\sigma_D(D^{\,a}_{m})$ does not depend on $a\in\R^n$ because
$\sigma_D$ is flat,
\begin{equation}\label{G eq2}
\begin{split}
0&=\int_{\{a\,:\,x\in D^{\,a}_{m}\}}\frac{2^{mn}}{\sigma_D(D^{\,a}_{m})}\int_{D^{\,a}_{m}}\int_{(D^{\,a}_{m})^c}
\varphi^{\,2^{-j}}_{2^{-j-1}}(z-y)\\
&\qquad\qquad\qquad\qquad\qquad\qquad\qquad\qquad\gamma_{2^{-m}}(x-y)K(z-y)\,d\sigma_{D}(y)\,d\sigma_{D}(z)\,da\\
&=\iint\varphi^{\,2^{-j}}_{2^{-j-1}}(z-y)\gamma_{2^{-m}}(x-y)
\Lambda_m^{\sigma_D}(x,z\,;\,y)K(z-y)\,d\sigma_{D}(z)\,d\sigma_{D}(y).
\end{split}
\end{equation}

Let $\{\eta_{Q}\}_{Q\in\DD_{j}}$ be a partition of the unity with
respect to the v-cubes $Q\in\DD_{j}$, i.e. $\eta_{Q}:\R^d\to\R$ are
$\CC^{\infty}$ functions such that:
$\chi_{0.9Q}\leq\eta_{Q}\leq\chi_{1.1Q}$,
$|\nabla\eta_{Q}|\lesssim\ell(Q)^{-1}=2^j$,
$\sum_{Q\in\DD_{j}}\eta_{Q}=1$ and $\eta_{Q}(y)=\eta_{Q}(\widetilde
y,0)$ for all $y\in\R^{d}$. It is easy to check that, if $j\geq
m-1$, $Q\in\DD_j$, and
$\supp\eta_{Q}\cap\supp\gamma_{2^{-m}}(x-\cdot)\neq\emptyset$, then
$Q\subset C_eD$ for some absolute constant $C_e>1$.

Given $Q\in\DD_j$, let $L_{Q}$ and
$\sigma_{Q}:=c_{Q}\HH^{n}_{L_{Q}}$ be  a minimizing $n$-plane and
measure for $\alpha(Q)$, respectively, and consider the measure
$$\lambda:=\sum_{Q\in\DD_{j}\,:\,Q\subset C_eD}\eta_{Q}\sigma_{Q}.$$
By (\ref{G eq2}) and the properties of the partition of the unity
$\{\eta_{Q}\}_{Q\in\DD_{j}}$, for $j\geq m-1$ we can decompose
$G_j^m(x)$ as
\begin{equation}\label{G eq3}
G_j^m(x)=G1_{j}^{m}(x)+G2_{j}^{m}(x)+G3_{j}^{m}(x)+G4_{j}^{m}(x)+G5_{j}^{m}(x),
\end{equation}
where
\vskip-10pt
\begin{eqnarray}
\qquad G1_{j}^{m}(x)&:=&\sum_{Q\in\DD_{j}:\,Q\subset C_eD}\iint\ldots\;d(\mu-\sigma_{Q})(z)\,d\mu(y),\\
G2_{j}^{m}(x)&:=&\sum_{Q\in\DD_{j}:\,Q\subset C_eD}\iint\ldots\;d\sigma_{Q}(z)\,d(\mu-\sigma_{Q})(y),\\
G3_{j}^{m}(x)&:=&\sum_{Q\in\DD_{j}:\,Q\subset C_eD}\iint\ldots\;
d(\sigma_{Q}\times\sigma_{Q}-\sigma_{D}\times\sigma_{D})(z,y),\label{G eq6}
\end{eqnarray}
where ``$\ldots$'' stands for $\varphi^{\,2^{-j}}_{2^{-j-1}}(z-y)\gamma_{2^{-m}}(x-y)\eta_Q(y)K(z-y)\Lambda^\mu_{m}(x,z\,;\,y)$, and
\begin{multline}\label{G eq7}
G4_{j}^{m}(x):=
\iint\varphi^{\,2^{-j}}_{2^{-j-1}}(z-y)\gamma_{2^{-m}}(x-y)K(z-y)\\
(\Lambda^\mu_{m}(x,z\,;\,y)-\Lambda^{\lambda}_{m}(x,z\,;\,y))\,d\sigma_{D}(z)\,d\sigma_{D}(y),
\end{multline}
\vskip-15pt
\begin{multline}\label{G eq8}
G5_{j}^{m}(x):=
\iint\varphi^{\,2^{-j}}_{2^{-j-1}}(z-y)\gamma_{2^{-m}}(x-y)K(z-y)\\
(\Lambda^\lambda_{m}(x,z\,;\,y)-\Lambda^{\sigma_D}_{m}(x,z\,;\,y))\,d\sigma_{D}(z)\,d\sigma_{D}(y).
\end{multline}

In the next subsections we will prove the following estimates:
\begin{equation}\label{G eq9}
|G1^m_j(x)|+|G2^m_j(x)|+|G4^m_j(x)|
\lesssim\sum_{Q\in\DD_{j}\,:\,Q\subset C_eD}2^{(m-j)(n+1)}\alpha(Q),
\end{equation}
\vskip-15pt
\begin{equation} \label{G eq11}
|G3^m_j(x)|+|G5^m_j(x)|\lesssim\sum_{Q\in\DD_{j}\,:\,Q\subset C_eD}
2^{(m-j)(n+1)}\bigg(\alpha(C_bD)+\sum_{R\in\DD\,:\,Q\subset R\subset C_bD}\alpha(R)\bigg),
\end{equation}
where $C_b$ is some absolute constant bigger than $C_e$.
Then, using (\ref{G eq14}) and (\ref{G eq3}), we will finally obtain that, for all $D\in\DD_m$ and $x\in D\cap\Gamma$,
\begin{equation}\label{G estimacio final}
\begin{split}
\sum_{j\in\Z}|G_j^m(x)|&\lesssim\sum_{j\geq m-1}\,\sum_{Q\in\DD_{j}\,:\,Q\subset C_eD}
2^{(m-j)(n+1)}\bigg(\alpha(C_bD)+\sum_{R\in\DD\,:\,Q\subset R\subset C_bD}\alpha(R)\bigg)\\
&\lesssim\sum_{Q\in\DD\,:\,Q\subset C_eD}
\frac{\ell(Q)^{n+1}}{\ell(D)^{n+1}}\bigg(\alpha(C_bD)+\sum_{R\in\DD\,:\,Q\subset R\subset C_bD}\alpha(R)\bigg)\\
&\lesssim\alpha(C_bD)+\sum_{R\in\DD\,:\,R\subset C_bD}\frac{\ell(R)^{n+1}}{\ell(D)^{n+1}}\,\alpha(R),
\end{split}
\end{equation}
which gives (\ref{G estimacio final2}).

\subsubsection{{\bf Estimate of $G1_{j}^{m}(x)$}}\label{G ss1}
If $\varphi^{\,2^{-j}}_{2^{-j-1}}(z-y)\neq0$ then $2^{-j-1}2.1\sqrt{n}\leq|\widetilde z-\widetilde y|\leq2^{-j}3\sqrt{n}$, so if we also have that $y\in\supp\eta_Q$, then $z\in8\sqrt{n}Q$ because $Q\in\DD_j$. Therefore, we can assume that $\supp\varphi^{\,2^{-j}}_{2^{-j-1}}(\cdot-y)\eta_Q(y)\subset B_Q$ if $C_\Gamma$ is big enough.
\begin{claim}\label{G1 eq1}
For $z\in\supp\varphi^{\,2^{-j}}_{2^{-j-1}}(\cdot-y)$, the
following hold: $|\Lambda_m^{\mu}(x,z\,;\,y)|\lesssim2^{m(n+1)-j}$,
 $|\nabla_z\Lambda_m^{\mu}(x,z\,;\,y)|\lesssim2^{m(n+1)},$ and $|\nabla_y\Lambda_m^{\mu}(x,z\,;\,y)|\lesssim2^{m(n+1)}.$
\end{claim}
We have that $|K(z-y)|\lesssim2^{jn}$ and $|\nabla_zK(z-y)|\lesssim2^{j(n+1)}$ for all $z\in\supp\varphi^{\,2^{-j}}_{2^{-j-1}}(\cdot-y)$. Using (\ref{estimacio varphi}) and the first two estimates in Claim \ref{G1 eq1}, we get
$$|\nabla_z\big(\varphi^{\,2^{-j}}_{2^{-j-1}}(z-y)K(z-y)\Lambda^\mu_{m}(x,z\,;\,y)\big)|\lesssim2^{m(n+1)+jn}.$$
Therefore, for $y\in\supp\eta_Q$,
\begin{equation*}
\begin{split}
\left|\int\varphi^{\,2^{-j}}_{2^{-j-1}}(z-y)K(z-y)\Lambda^\mu_{m}(x,z\,;\,y)\,d(\mu-\sigma_{Q})(z)\right|
&\lesssim2^{m(n+1)+jn}\dist_{B_Q}(\mu,\sigma_Q)\\
&\lesssim2^{m(n+1)-j}\alpha(Q),
\end{split}
\end{equation*}
and then,
\begin{equation*}
\begin{split}
|G1_{j}^{m}(x)|&\lesssim\sum_{Q\in\DD_{j}\,:\,Q\subset C_eD}\int_{\supp\eta_{Q}}2^{m(n+1)-j}\alpha(Q)\,d\mu(y)
\lesssim\sum_{Q\in\DD_{j}\,:\,Q\subset C_eD}2^{(m-j)(n+1)}\alpha(Q).
\end{split}
\end{equation*}

\subsubsection{{\bf Estimate of $G2_{j}^{m}(x)$}}\label{G ss2}
It can be estimated using the arguments of subsection \ref{G ss1},
but  now we also have to take into account that
$|\nabla_{y}\gamma_{2^{-m}}(x-y)|\lesssim2^{m}\lesssim2^{j},$
because we are assuming $j\geq m-1$, and we have to use the last estimate in Claim \ref{G1 eq1}.

\subsubsection{{\bf Estimate of $G3_{j}^{m}(x)$}}\label{G ss3}
Given $x\in D\cap\Gamma$ and $Q\in\DD_j$, denote
\begin{equation*}
H_{Q}(y,z):=\varphi^{\,2^{-j}}_{2^{-j-1}}(z-y)\gamma_{2^{-m}}(x-y)\eta_Q(y)K(z-y)
\Lambda^\mu_{m}(x,z\,;\,y).
\end{equation*}
Then, (\ref{G eq6}) becomes
\begin{equation}\label{G3 eq1}
\begin{split}
G3_{j}^{m}(x)&=\sum_{Q\in\DD_{j}:\,Q\subset C_eD}
\iint H_Q(y,z)\,d(\sigma_{Q}\times\sigma_{Q}-\sigma_{D}\times\sigma_{D})(z,y)\\
&=\sum_{Q\in\DD_{j}:\,Q\subset C_eD}
\iint H_{Q}(y,z)\,d(c_{Q}^{2}\,\HH^{n}_{L_{Q}}\times\HH^{n}_{L_{Q}}-c_{D}^{2}\,\HH^{n}_{L_{D}}\times\HH^{n}_{L_{D}})(z,y)\\
&=\sum_{Q\in\DD_{j}:\,Q\subset C_eD}
\iint H_{Q}(y,z)(c_{Q}^{2}-c_{D}^{2})\,d\HH^{n}_{L_{Q}}(z)\,d\HH^{n}_{L_{Q}}(y)\\
&\quad+\sum_{Q\in\DD_{j}:\,Q\subset C_eD}
c_{D}^{2}\iint H_{Q}(y,z)\,d(\HH^{n}_{L_{Q}}\times\HH^{n}_{L_{Q}}-\HH^{n}_{L_{D}}\times\HH^{n}_{L_{D}})(z,y)\\
&=:G3A_j^m(x)+G3B_j^m(x).
\end{split}
\end{equation}

We are going to estimate the terms $G3A_j^m(x)$ and $G3B_j^m(x)$ separately.
Recall that $\ell(D)=2^{-m}$. Given a v-cube $Q\in\DD_{j}$ such that $Q\subset C_eD$, let $Q=:Q_{j}\subset\ldots\subset Q_{i+1}\subset Q_{i}\subset\ldots\subset Q_{m-1}$ be the sequence of v-cubes such that $Q_{i}$ belongs to $\DD_{i}$ for $i=m-1,\ldots,j$. Evidently, $Q_{m-1}\subset C_bD$ for some constant $C_b$ big enough, because $\ell(Q_{m-1})=2\ell(D)$ and $Q\subset Q_{m-1}\cap C_eD$. Let $L_{Q_{i}}$ be an $n$-plane that minimizes $\alpha(Q_{i})$ and let $\sigma_{Q_{i}}:=c_{Q_{i}}\HH^{n}_{L_{Q_{i}}}$ be a minimizing measure of $\alpha(Q_{i})$. Also, let $L_{C_bD}$ and $\sigma_{C_bD}:=c_{C_bD}\HH^{n}_{L_{C_bD}}$ be a minimizing $n$-plane and measure of $\alpha(C_bD)$, respectively.

In order to estimate $G3A_j^m(x)$, notice that, by \cite[Lemma
3.4]{To}  and the triangle inequality, $|c_{Q_{i}}|\lesssim1$ for
all $i=m-1,\ldots,j$, and
\begin{equation}\label{G eq15}
\begin{split}
|c_{Q}^{2}-c_{D}^{2}|&=|c_{Q}+c_{D}|\,|c_{Q}-c_{D}|\lesssim|c_{Q_j}-c_{D}|\\
&\lesssim|c_{Q_{m-1}}-c_{C_bD}|+|c_{C_bD}-c_{D}|+\sum_{i=m-1}^{j-1}|c_{Q_{i+1}}-c_{Q_{i}}|\\
&\lesssim\alpha(C_bD)+\sum_{i=m-1}^{j-1}\alpha(Q_{i})
\lesssim\alpha(C_bD)+\sum_{R\in\DD:\,Q\subset R\subset C_bD}\alpha(R)
\end{split}
\end{equation}
(in the case that $j=m-1$, there are no intermediate scales between
$j$  and $m-1$, so one just obtains
$|c_{Q}^{2}-c_{D}^{2}|\lesssim\alpha(C_bD)$).
\begin{claim}\label{G3 eq2}
For $z\in\supp\varphi^{\,2^{-j}}_{2^{-j-1}}(\cdot-y)$,  we
have $|\Lambda_m^{\mu}(x,z\,;\,y)|\lesssim2^{m(n+1)-j}.$
\end{claim}
Notice that this last estimate is the same as the first one in Claim
\ref{G1 eq1}.  Using Claim \ref{G3 eq2} and that
$|K(z-y)|\lesssim2^{jn}$ for all
$z\in\supp\varphi^{\,2^{-j}}_{2^{-j-1}}(\cdot-y)$, we easily
obtain $|H_{Q}(y,z)|\lesssim2^{m(n+1)+j(n-1)}$.
Therefore, using (\ref{G eq15}),
\begin{equation}\label{G3 eq12}
\begin{split}
|G3A_j^m(x)|&\lesssim\sum_{Q\in\DD_{j}\,:\,Q\subset C_eD}
|c_{Q}^{2}-c_{D}^{2}|\iint |H_{Q}(y,z)|\,d\HH^{n}_{L_{Q}}(z)\,d\HH^{n}_{L_{Q}}(y)\\
&\lesssim\sum_{Q\in\DD_{j}\,:\,Q\subset C_eD}
2^{(m-j)(n+1)}\bigg(\alpha(C_bD)+\sum_{R\in\DD\,:\,Q\subset R\subset C_bD}\alpha(R)\bigg).
\end{split}
\end{equation}

To estimate $G3B_j^m(x)$ in (\ref{G3 eq1}), we set
\begin{equation}\label{G3 eq10}
G3B_j^m(x)=\sum_{Q\in\DD_{j}\,:\,Q\subset C_eD}c_{D}^{2}\,G3B(Q)_j^m(x),
\end{equation}
where $G3B(Q)_j^m(x):=\iint H_{Q}\,d(\HH^{n}_{L_{Q}}\times\HH^{n}_{L_{Q}}-\HH^{n}_{L_{D}}\times\HH^{n}_{L_{D}})$. Given $Q\in\DD_j$ such that $Q\subset C_eD$, we split $G3B(Q)_j^m(x)$ as follows:
\begin{equation}\label{G3 eq6}
\begin{split}
G3B(Q)_j^m(x)&=\sum_{i=m-1}^{j-1}\iint H_{Q}\,d(\HH^{n}_{L_{Q_{i+1}}}\times\HH^{n}_{L_{Q_{i+1}}}-\HH^{n}_{L_{Q_{i}}}\times\HH^{n}_{L_{Q_{i}}})\\
&\quad+\iint H_{Q}\,d(\HH^{n}_{L_{Q_{m-1}}}\times\HH^{n}_{L_{Q_{m-1}}}-\HH^{n}_{L_{C_bD}}\times\HH^{n}_{L_{C_bD}})\\
&\quad+\iint H_{Q}\,d(\HH^{n}_{L_{C_bD}}\times\HH^{n}_{L_{C_bD}}-\HH^{n}_{L_{D}}\times\HH^{n}_{L_{D}})
\end{split}
\end{equation}
(as before, in the case $j=m-1$ the first term on the right hand side of (\ref{G3 eq6}) does not exist).

Fix $i\in\Z$ such that $m-1\leq i<j$. Set $L_{Q_{i+1}}=\{(\widetilde y, a(\widetilde y))\in\R^d\,:\,\widetilde y\in\R^n\}$, where $a:\R^n\to\R^{d-n}$ is an appropriate affine map, and let $p:L_{Q_{i}}\to L_{Q_{i+1}}$ be the map defined by $p(y):=(\widetilde y,a(\widetilde y))$. Let $p_{\sharp}\HH^{n}_{L_{Q_{i}}}$ be the image measure of $\HH^{n}_{L_{Q_{i}}}$ by $p$. It is easy to check that $\HH^{n}_{L_{Q_{i+1}}}=\tau p_{\sharp}\HH^{n}_{L_{Q_{i}}}$, where $\tau$ is some positive constant such that $|\tau-1|\lesssim\alpha(Q_{i})$ and $\tau\lesssim1$. Therefore,
\begin{equation}\label{G3 eq3}
\begin{split}
\iint H_{Q}(y,z)\,&d(\HH^{n}_{L_{Q_{i+1}}}\times\HH^{n}_{L_{Q_{i+1}}}
-\HH^{n}_{L_{Q_{i}}}\times\HH^{n}_{L_{Q_{i}}})(z,y)\\
&=\iint \Big(\tau^{2}H_{Q}(p(y),p(z))-H_{Q}(y,z)\Big)\,d\HH^{n}_{L_{Q_{i}}}(z)\,d\HH^{n}_{L_{Q_{i}}}(y)\\
&=\iint \tau^{2}\Big(H_{Q}(p(y),p(z))-H_{Q}(y,z)\Big)\,d\HH^{n}_{L_{Q_{i}}}(z)\,d\HH^{n}_{L_{Q_{i}}}(y)\\
&\quad+\iint (\tau^{2}-1)H_{Q}(y,z)\,d\HH^{n}_{L_{Q_{i}}}(z)\,d\HH^{n}_{L_{Q_{i}}}(y).
\end{split}
\end{equation}
Since $|\tau^{2}-1|\lesssim\alpha(Q_{j})$ and we have seen that
$|H_{Q}(y,z)|\lesssim2^{m(n+1)+j(n-1)}$ after Claim \ref{G3 eq2},
the second term on the right side of the last equality is bounded by $C2^{(m-j)(n+1)}\alpha(Q_{i})$.

In order to estimate the first term on the right hand side of (\ref{G3 eq3}), notice that
$\varphi^{\,2^{-j}}_{2^{-j-1}}(z-y),\gamma_{2^{-m}}(x-y),\eta_Q(y)$
and $\Lambda^\mu_{m}(x,z\,;\,y)$ only depend on the first $n$
coordinates of $y$ and $z$ (i.e., on $\widetilde y$ and $\widetilde
z$), thus their values coincide on $(y,z)$ and $(p(y),p(z))$. Then,
\begin{equation*}
\begin{split}
\iint\tau^{2}&\Big(H_{Q}(p(y),p(z))-H_{Q}(y,z)\Big)\,d\HH^{n}_{L_{Q_{i}}}(z)\,d\HH^{n}_{L_{Q_{i}}}(y)\\
&=\tau^2\iint\varphi^{\,2^{-j}}_{2^{-j-1}}(z-y)\gamma_{2^{-m}}(x-y)\eta_Q(y)\Lambda^\mu_{m}(x,z\,;\,y)\\
&\qquad\qquad\qquad\qquad\qquad\qquad\Big(K(p(z)-p(y))-K(z-y)\Big)\,d\HH^{n}_{L_{Q_{i}}}(z)\,d\HH^{n}_{L_{Q_{i}}}(y).
\end{split}
\end{equation*}
Let $\theta_i$ be the angle between $L_{Q_i}$ and $L_{Q_{i+1}}$.
One can easily see that, for $y,z\in L_{Q_{i}}$,
$|(p(z)-p(y))-(z-y)|\lesssim\sin(\theta_i)|z-y|\lesssim\alpha(Q_{i})|z-y|$.
Thus, if also $z\in\supp\varphi^{\,2^{-j}}_{2^{-j-1}}(\cdot-y)$,
\begin{equation*}
\begin{split}
|K(p(z)-p(y))-K(z-y)|&\lesssim2^{j(n+1)}|(p(z)-p(y))-(z-y)|\lesssim2^{jn}\alpha(Q_{i}).
\end{split}
\end{equation*}
Together with Claim \ref{G3 eq2} and the fact that $\tau^{2}\lesssim1$, this gives
\begin{equation*}
\begin{split}
\bigg|\iint\tau^{2}\Big(H_{Q}&(p(y),p(z))-H_{Q}(y,z)\Big)\,d\HH^{n}_{L_{Q_{i}}}(z)\,d\HH^{n}_{L_{Q_{i}}}(y)\bigg|\\
&\lesssim\iint\varphi^{\,2^{-j}}_{2^{-j-1}}(z-y)\eta_Q(y)2^{m(n+1)+j(n-1)}\alpha(Q_{i})\,d\HH^{n}_{L_{Q_{i}}}(z)\,d\HH^{n}_{L_{Q_{i}}}(y)\\
&\lesssim2^{(m-j)(n+1)}\alpha(Q_{i}).
\end{split}
\end{equation*}

From the last estimates and (\ref{G3 eq3}), we get
\begin{equation*}
\Big|\iint H_{Q}\,d(\HH^{n}_{L_{Q_{i+1}}}\times\HH^{n}_{L_{Q_{i+1}}}
-\HH^{n}_{L_{Q_{i}}}\times\HH^{n}_{L_{Q_{i}}})\Big|\lesssim2^{(m-j)(n+1)}\alpha(Q_{i})
\end{equation*}
for $i=m-1,\ldots,j-1$. With similar arguments, one also obtains
\begin{equation*}
\begin{split}
\Big|\iint H_{Q}\,d(\HH^{n}_{L_{Q_{m-1}}}\times\HH^{n}_{L_{Q_{m-1}}}-\HH^{n}_{L_{C_bD}}\times\HH^{n}_{L_{C_bD}})\Big|
&\lesssim2^{(m-j)(n+1)}\alpha(C_bD),\\
\Big|\iint H_{Q}\,d(\HH^{n}_{L_{C_bD}}\times\HH^{n}_{L_{C_bD}}-\HH^{n}_{L_{D}}\times\HH^{n}_{L_{D}})\Big|
&\lesssim2^{(m-j)(n+1)}\alpha(C_bD).
\end{split}
\end{equation*}
These last three inequalities together with (\ref{G3 eq6}), (\ref{G3 eq10}) and the fact that $|c_D|\lesssim1$ yield
\begin{equation}\label{G3 eq11}
\begin{split}
|G3B_j^m(x)|&\lesssim\sum_{Q\in\DD_{j}\,:\,Q\subset C_eD}2^{(m-j)(n+1)}\bigg(\alpha(C_bD)+\sum_{i=m-1}^{j-1}\alpha(Q_{i})\bigg)\\
&\leq\sum_{Q\in\DD_{j}\,:\,Q\subset C_eD}2^{(m-j)(n+1)}\bigg(\alpha(C_bD)+\sum_{R\in\DD\,:\,Q\subset R\subset C_bD}\alpha(R)\bigg).
\end{split}
\end{equation}
Finally, (\ref{G3 eq12}) and (\ref{G3 eq11}) applied to (\ref{G3 eq1}) give half of (\ref{G eq11}).

\subsubsection{{\bf Estimate of $G4_{j}^{m}(x)$}}\label{G ss4}
By Fubini's theorem and the definitions of $\lambda$, $\Lambda_m^{\mu}$ and $\Lambda_m^{\lambda}$,
\begin{equation*}
\begin{split}
\Lambda_m^{\mu}&(x,z\,;\,y)-\Lambda_m^{\lambda}(x,z\,;\,y)
=2^{mn}\int_{\{a:\,x,z\in D^{\,a}_{m},\,y\notin D^{\,a}_{m}\}}\frac{\lambda(D^{\,a}_{m})-\mu(D^{\,a}_{m})}{\mu(D^{\,a}_{m})\lambda(D^{\,a}_{m})}\,da\\
&=2^{mn}\int_{\{a:\,x,z\in D^{\,a}_{m},\,y\notin D^{\,a}_{m}\}}
\bigg(\int_{t\in D^{\,a}_{m}}\sum_{Q\in\DD_{j}\,:\,Q\subset C_eD}\eta_{Q}(t)\,d(\sigma_{Q}-\mu)(t)\bigg)\frac{da}{\mu(D^{\,a}_{m})\lambda(D^{\,a}_{m})}\\
&=\sum_{Q\in\DD_{j}\,:\,Q\subset C_eD}\int\eta_{Q}(t)\Lambda^{\mu,\lambda}_m(x,z,t\,;\,y)\,d(\sigma_{Q}-\mu)(t).
\end{split}
\end{equation*}
We also used in the second equality that $1=\sum_{Q\in\DD_{j}}\eta_{Q}(t)=\sum_{Q\in\DD_{j}:\,Q\subset C_eD}\eta_{Q}(t)$ for all $t\in D_m^{\,a}$ if $C_e$ is big enough, and this is because $j\geq m-1$ and $|\wit x-\wit t|\lesssim2^{-m}$ for all $t\in D_m^{\,a}$.
\begin{claim}\label{G4 eq1}
For $x\in D$, $j\geq m-1$, $|x-y|\lesssim2^{-m}$, and $z\in\supp\varphi^{\,2^{-j}}_{2^{-j-1}}(\cdot-y)$, the
following  hold:
$|\Lambda_m^{\mu,\lambda}(x,z,t\,;\,y)|\lesssim2^{m(2n+1)-j}$ and
$|\nabla_t\Lambda_m^{\mu,\lambda}(x,z,t\,;\,y)|\lesssim2^{m(2n+1)}.$
\end{claim}
Using Claim \ref{G4 eq1} and the properties of $\eta_Q$, we obtain
\begin{equation*}
\begin{split}
|\Lambda_m^{\mu}(x,z\,;\,y)-\Lambda_m^{\lambda}(x,z\,;\,y)|
&\lesssim\sum_{Q\in\DD_{j}\,:\,Q\subset C_eD}2^{m(2n+1)}\dist_{B_Q}(\mu,\sigma_Q)\\
&\lesssim\sum_{Q\in\DD_{j}\,:\,Q\subset C_eD}2^{m(2n+1)-j(n+1)}\alpha(Q).
\end{split}
\end{equation*}
Plugging this estimate into the definition of $G4_{j}^{m}(x)$ in (\ref{G eq7}), we get
\begin{equation*}
\begin{split}
|G4_{j}^{m}(x)|&\lesssim\iint\varphi^{\,2^{-j}}_{2^{-j-1}}(z-y)\gamma_{2^{-m}}(x-y)|K(z-y)|\\
&\qquad\qquad\qquad\qquad\sum_{Q\in\DD_{j}\,:\,Q\subset C_eD}2^{m(2n+1)-j(n+1)}\alpha(Q)\,d\sigma_{D}(z)\,d\sigma_{D}(y)\\
&\lesssim\sum_{Q\in\DD_{j}\,:\,Q\subset C_eD}2^{(m-j)(n+1)}\alpha(Q),\\
\end{split}
\end{equation*}
which, together with the estimates of $|G1_{j}^{m}(x)|$ and
$|G2_{j}^{m}(x)|$  in subsections \ref{G ss1} and \ref{G ss2}, gives
(\ref{G eq9}).

\subsubsection{{\bf Estimate of $G5_{j}^{m}(x)$}}\label{G ss5}
Arguing as in subsection \ref{G ss4}, we have
\begin{equation}\label{G5 eq1}
\begin{split}
\Lambda_m^{\lambda}(x,z\,;\,y)-\Lambda_m^{\sigma_D}(x,z\,;\,y)
&=\sum_{Q\in\DD_{j}\,:\,Q\subset C_eD}\int\eta_{Q}(t)\Lambda^{\lambda,\sigma_D}_m(x,z,t\,;\,y)\,d(\sigma_{D}-\sigma_Q)(t)\\
&=\sum_{Q\in\DD_{j}\,:\,Q\subset C_eD}\int H_{Q}(t)\,d(\sigma_{D}-\sigma_Q)(t),
\end{split}
\end{equation}
where we have set $H_Q(t):=\eta_{Q}(t)\Lambda^{\lambda,\sigma_D}_m(x,z,t\,;\,y)$.

We are going to estimate the right hand side of (\ref{G5 eq1}) using the techniques of subsection \ref{G
ss3}. We have
\begin{equation}\label{G5 eq3}
\int H_Q\,d(\sigma_{D}-\sigma_{Q})=(c_{D}-c_{Q})\int H_Q\,d\HH^{n}_{L_{D}}+c_{Q}\int H_{Q}\,d(\HH^{n}_{L_{D}}-\HH^{n}_{L_{Q}}).
\end{equation}
We introduce the intermediate v-cubes between $Q\in\DD_{j}$ and $D\in\DD_{m}$ to obtain
\begin{equation}\label{G5 eq5}
|c_{D}-c_{Q}|\lesssim\alpha(C_bD)+\sum_{R\in\DD\,:\,Q\subset R\subset C_bD}\alpha(R).
\end{equation}
\begin{claim}\label{G5 eq4}
For $x\in D$, $j\geq m-1$, $|x-y|\lesssim2^{-m}$, and $z\in\supp\varphi^{\,2^{-j}}_{2^{-j-1}}(\cdot-y)$, the
following holds:
$|\Lambda^{\lambda,\sigma_D}_m(x,z,t\,;\,y)|\lesssim2^{m(2n+1)-j}.$
\end{claim}
Combining Claim \ref{G5 eq4} with (\ref{G5 eq5}), we derive that
\begin{equation}\label{G5 eq6}
|c_{D}-c_{Q}|\int|H_Q|\,d\HH^{n}_{L_{D}}
\lesssim2^{m(2n+1)-j(n+1)}\bigg(\alpha(C_bD)+\sum_{R\in\DD\,:\,Q\subset R\subset C_bD}\alpha(R)\bigg).
\end{equation}

For the second term on the right side of (\ref{G5 eq3}),  one
can also use the arguments in subsection \ref{G ss3} (see (\ref{G3
eq6}) and following) to show that
\begin{multline}\label{G5 eq7}
\left|\int H_Q\,d(\HH^{n}_{L_{D}}-\HH^{n}_{L_{Q}})\right|
\lesssim2^{m(2n+1)-j(n+1)}\bigg(\alpha(C_bD)+\sum_{R\in\DD\,:\,Q\subset R\subset C_bD}\alpha(R)\bigg)
\end{multline}
(now it is easier because the function $H_{Q}(t)$ only depends on the first $n$ coordinates of the points involved, i.e., it depends only on $\widetilde x$, $\widetilde y$, $\widetilde z$ and $\widetilde t$, so when we project vertically to deal with the image measure, the function $H_Q$ is not affected).
Therefore, by (\ref{G5 eq6}), (\ref{G5 eq7}), (\ref{G5 eq3}), and (\ref{G5 eq1}), we obtain
\begin{multline*}
|\Lambda_m^{\lambda}(x,z\,;\,y)-\Lambda_m^{\sigma_D}(x,z\,;\,y)|
\lesssim\sum_{\begin{subarray}{c}Q\in\DD_{j}:\\Q\subset C_eD\end{subarray}}2^{m(2n+1)-j(n+1)}\bigg(\alpha(C_bD)+
\sum_{\begin{subarray}{c}R\in\DD:\\Q\subset R\subset C_bD\end{subarray}}\alpha(R)\bigg).
\end{multline*}

From the definition of $G5_{j}^{m}(x)$ in (\ref{G eq8}), we conclude that
\begin{equation*}
\begin{split}
|G5_{j}^{m}(x)|&\lesssim\sum_{Q\in\DD_{j}\,:\,Q\subset C_eD}2^{m(2n+1)-j(n+1)}
\bigg(\alpha(C_bD)+\sum_{R\in\DD\,:\,Q\subset R\subset C_bD}\alpha(R)\bigg)\\
&\qquad\qquad\qquad\iint\varphi^{\,2^{-j}}_{2^{-j-1}}(z-y)\gamma_{2^{-m}}(x-y)|K(z-y)|\,d\sigma_{D}(z)\,d\sigma_{D}(y)\\
&\lesssim\sum_{Q\in\DD_{j}\,:\,Q\subset C_eD}
2^{(m-j)(n+1)}\bigg(\alpha(C_bD)+
\sum_{R\in\DD\,:\,Q\subset R\subset C_bD}\alpha(R)\bigg),\\
\end{split}
\end{equation*}
which, together with the estimate of $|G3_j^m(x)|$ in subsection \ref{G ss3}, gives (\ref{G eq11}).

\subsection{Proof of Claims \ref{F1 eq3}$,\ldots,$ \ref{G5 eq4} }\label{subseccio claims}
We have to prove:
\begin{itemize}
\item{\bf Claim \ref{F1 eq3}:} We have $|\Lambda_m^{\mu}(x,z\,;\,y)|\lesssim2^{mn}$ and $|\nabla_z\Lambda_m^{\mu}(x,z\,;\,y)|\lesssim2^{m(n+1)}$ for all $x,y,z\in\R^d$.
\item{\bf Claim \ref{F2 eq2}:} For $j<m$, $y\in\supp\varphi^{\,2^{-j}}_{2^{-j-1}}(x-\cdot)$, and $|\widetilde x-\widetilde z|\leq2^{-m}\sqrt{n}$, the following hold:
    $|\Lambda_m^{\mu}(x,z\,;\,y)|\lesssim2^{mn}$ and $\nabla_y\Lambda_m^{\mu}(x,z\,;\,y)=0.$
\item{\bf Claim \ref{F3 eq2}:} We have $|\Lambda_m^{\mu,\sigma_D}(x,z,t\,;\,y)|\lesssim2^{2mn}$ and $|\nabla_t\Lambda_m^{\mu,\sigma_D}(x,z,t\,;\,y)|\lesssim2^{m(2n+1)}$ for all $x,y,z,t\in\R^d$.
\item{\bf Claim \ref{G1 eq1}:} For $z\in\supp\varphi^{\,2^{-j}}_{2^{-j-1}}(\cdot-y)$, the following hold: $|\Lambda_m^{\mu}(x,z\,;\,y)|\lesssim2^{m(n+1)-j}$, $|\nabla_z\Lambda_m^{\mu}(x,z\,;\,y)|\lesssim2^{m(n+1)},$ and $|\nabla_y\Lambda_m^{\mu}(x,z\,;\,y)|\lesssim2^{m(n+1)}.$
\item{\bf Claim \ref{G3 eq2}:} For $z\in\supp\varphi^{\,2^{-j}}_{2^{-j-1}}(\cdot-y)$,  $|\Lambda_m^{\mu}(x,z\,;\,y)|\lesssim2^{m(n+1)-j}.$
\item{\bf Claim \ref{G4 eq1}:} For $x\in D$, $j\geq m-1$, $|x-y|\lesssim2^{-m}$, and $z\in\supp\varphi^{\,2^{-j}}_{2^{-j-1}}(\cdot-y)$, the following hold: $|\Lambda_m^{\mu,\lambda}(x,z,t\,;\,y)|\lesssim2^{m(2n+1)-j}$ and $|\nabla_t\Lambda_m^{\mu,\lambda}(x,z,t\,;\,y)|\lesssim2^{m(2n+1)}.$
\item{\bf Claim \ref{G5 eq4}:}  For $x\in D$, $j\geq m-1$, $|x-y|\lesssim2^{-m}$, and $z\in\supp\varphi^{\,2^{-j}}_{2^{-j-1}}(\cdot-y)$, the following holds: $|\Lambda^{\lambda,\sigma_D}_m(x,z,t\,;\,y)|\lesssim2^{m(2n+1)-j}.$
\end{itemize}

To prove the claims, we need to express the function $\Lambda$ at
the end  of subsection \ref{ss particular martingale} in a more
convenient way. Notice that we can
replace $D^{\,a}_{m}$ by $\overline{D^{\,a}_{m}}$ in the definition
of $\Lambda$ because $\mu$ and the $n$-dimensional
Hausdorff measure vanish on $\partial D^{\,a}_{m}$.

For $u\in\R^n$ and $r>0$, we denote $|u|_\infty:=\max_{i=1,\ldots,n}|u^i|$, $B_\infty(u,r):=\{v\in\R^n\,:\,|u-v|_\infty\leq r\}$, and $B_\infty^m(u):=B_\infty(u,2^{-m-1})$.
Given $a\in\R^n$, let $b:=a+\{2^{-m-1}\}^{n}\in\R^n$ be the center of $\wit D_m^{\,a}$. Then, given $q\in\R^d$,
$$q\in \overline{D_m^{\,a}}\;\Longleftrightarrow\;|\tilde q-b|_\infty\leq2^{-m}\;\Longleftrightarrow\;b\in B_\infty^m(\wit q).$$

Let $\mu_1,\ldots,\mu_k$ be positive Borel measures such that
$\mu_l(D^{\,a}_{m})>0$ and $\mu_l(\partial D^{\,a}_{m})=0$  for all
$a\in\R^n$, $m\in\Z$ and $l=1,\ldots,k$. Given $m\in\Z$ and
$\,x_1,\ldots,x_i,y_1,\ldots,y_j\in\R^{d}$ we have, by the change of
variable $b=a+\{2^{-m-1}\}^{n}\in\R^n$,
\begin{equation}\label{claims formula Lambda}
\begin{split}
\Lambda^{\mu_1,\ldots,\mu_k}_m(x_1,\ldots,x_i\,;\,y_1,\ldots,&y_j)
=\int_{\left\{a\in\R^n\,:\,x_1,\ldots,x_i\in \overline{D^{\,a}_{m}},\,y_1,\ldots,y_j\notin \overline{D^{\,a}_{m}}\right\}}\frac{2^{nm}\,da}{\prod_{l=1}^k\mu_l(D^{\,a}_{m})}\\
&=2^{nm}\int\frac{\chi_{B_\infty^m(\wit x_1)\cap\ldots\cap B_\infty^m(\wit{x_i})\cap B_\infty^m(\wit{y_1})^c\cap\ldots\cap B_\infty^m(\wit{y_j})^c}(b)}{\prod_{l=1}^k\mu_l\big(D^{b-\{2^{-m-1}\}^{n}}_{m}\big)}\,db.
\end{split}
\end{equation}

\begin{proof}[Proof of Claim \ref{F1 eq3}]
By (\ref{claims formula Lambda}), we have 
\begin{equation}\label{claim 4.1 formula1}
\Lambda_m^{\mu}(x,z\,;\,y)=2^{nm}\int\mu\big(D^{b-\{2^{-m-1}\}^{n}}_{m}\big)^{-1}
\chi_{B_\infty^m(\wit{x})\cap B_\infty^m(\wit{z})\cap B_\infty^m(\wit{y})^c}(b)\,db.
\end{equation}
Since $\mu(D^{\,b}_{m})\gtrsim2^{-mn}$ for all $b\in\R^n$,
$$|\Lambda_m^{\mu}(x,z\,;\,y)|\lesssim2^{2mn}\LL^n(B_\infty^m(\wit{x})\cap B_\infty^m(\wit{z})\cap B_\infty^m(\wit{y})^c)\leq2^{2mn}\LL^n(B_\infty^m(\wit{x}))\leq2^{mn}.$$

To deal with the second inequality in Claim \ref{F1 eq3}, we will
estimate
$$|\Lambda_m^{\mu}(x,z_1\,;\,y)-\Lambda_m^{\mu}(x,z_2\,;\,y)|/|z_1-z_2|$$
for $z_1$ close enough to $z_2$. Recall that, given two sets
$F_1,F_2\subset\R^n$, $F_1\Delta F_2:=(F_1\setminus
F_2)\cup(F_2\setminus F_1)$ denotes their symmetric difference. Using (\ref{claim 4.1 formula1}), we get
\begin{equation}\label{claim 4.1 formula2}
\begin{split}
|\Lambda_m^{\mu}(x,&z_1\,;\,y)-\Lambda_m^{\mu}(x,z_2\,;\,y)|\\
&\lesssim2^{2nm}\int
|\chi_{B_\infty^m(\wit{x})\cap B_\infty^m(\wit{z_1})\cap B_\infty^m(\wit{y})^c}(b)
-\chi_{B_\infty^m(\wit{x})\cap B_\infty^m(\wit{z_2})\cap B_\infty^m(\wit{y})^c}(b)|\,db\\
&=2^{2nm}\LL^n\Big(\big(B_\infty^m(\wit{x})\cap B_\infty^m(\wit{z_1})\cap B_\infty^m(\wit{y})^c\big)\Delta
\big(B_\infty^m(\wit{x})\cap B_\infty^m(\wit{z_2})\cap B_\infty^m(\wit{y})^c\big) \Big)\\
&\leq2^{2nm}\LL^n\big(B_\infty^m(\wit{z_1})\Delta B_\infty^m(\wit{z_2})\big)
\lesssim2^{2nm}|\wit{z_1}-\wit{z_2}|2^{-m(n-1)}\leq2^{m(n+1)}|z_1-z_2|,
\end{split}
\end{equation}
and the claim follows.
\end{proof}
\begin{proof}[Proof of Claim \ref{F2 eq2}]
The first estimate has been already proved in Claim \ref{F1 eq3}.
Let us deal with the second one. Notice that if
$y\in\supp\varphi^{\,2^{-j}}_{2^{-j-1}}(x-\cdot)$ then
$|\wit x-\wit y|\geq2^{-j-1}2.1\sqrt{n}$. Thus, if also $j<m$ and
$|\widetilde x-\widetilde z|\leq2^{-m}\sqrt{n}$, then $|\wit x-\wit
y|>2^{-m}\sqrt{n}$ and $|\wit z-\wit y|>2^{-m}\sqrt{n}$. Therefore,
$B_\infty^m(\wit{x})\cap B_\infty^m(\wit{z})\cap
B_\infty^m(\wit{y})^c=B_\infty^m(\wit{x})\cap B_\infty^m(\wit{z})$
for all $y\in\supp\varphi^{\,2^{-j}}_{2^{-j-1}}(x-\cdot)$,
if $|\widetilde x-\widetilde z|\leq2^{-m}\sqrt{n}$. This means,
using (\ref{claim 4.1 formula1}), that $\Lambda_m^{\mu}(x,z\,;\,y)$
does not depend on $y$, so $\nabla_y\Lambda_m^{\mu}(x,z\,;\,y)=0$
for all $y\in\supp\varphi^{\,2^{-j}}_{2^{-j-1}}(x-\cdot)$,
and the claim is proved.
\end{proof}
\begin{proof}[Proof of Claim \ref{F3 eq2}]
This claim follows by arguments very similar to the ones in the proof of Claim \ref{F1 eq3}. Just notice that
$\mu(D^{\,b}_{m})\sigma_D(D^{\,b}_{m})\gtrsim2^{-2mn}$ for all $b\in\R^n$.
\end{proof}
\begin{proof}[Proof of Claim \ref{G1 eq1}]
Using (\ref{claim 4.1 formula1}), we have that
$$|\Lambda_m^{\mu}(x,z\,;\,y)|\lesssim2^{2mn}\LL^n(B_\infty^m(\wit{x})\cap B_\infty^m(\wit{z})\cap B_\infty^m(\wit{y})^c)\leq2^{2mn}\LL^n(B_\infty^m(\wit{z})\cap B_\infty^m(\wit{y})^c).$$
Notice that $\LL^n(B_\infty^m(\wit{z})\cap B_\infty^m(\wit{y})^c)\lesssim2^{-m(n-1)}|\wit y-\wit z|$.
Since $z\in\supp\varphi^{\,2^{-j}}_{2^{-j-1}}(\cdot-y)$, $|\wit y-\wit z|\leq2^{-j}3\sqrt{n}$. Thus,
$\LL^n(B_\infty^m(\wit{z})\cap B_\infty^m(\wit{y})^c)\lesssim2^{-m(n-1)-j}$, and then
$|\Lambda_m^{\mu}(x,z\,;\,y)|\lesssim2^{m(n+1)-j}$.

In Claim \ref{F1 eq3} we already proved that $|\nabla_z\Lambda_m^{\mu}(x,z\,;\,y)|\lesssim2^{m(n+1)}.$ Finally, to prove that $|\nabla_y\Lambda_m^{\mu}(x,z\,;\,y)|\lesssim2^{m(n+1)},$ one can repeat the computations done in (\ref{claim 4.1 formula2}) but applied to the $y$ coordinate  and use that $B_\infty^m(\wit{y_1})^c\Delta B_\infty^m(\wit{y_2})^c=B_\infty^m(\wit{y_1})\Delta B_\infty^m(\wit{y_2}).$
\end{proof}

\begin{proof}[Proof of Claim \ref{G3 eq2}]
This claim is included in the previous one.
\end{proof}
\begin{proof}[Proof of Claim \ref{G4 eq1}]
Recall that $\lambda=\sum_{Q\in\DD_{j}\,:\,Q\subset C_eD}\eta_{Q}\sigma_{Q},$ where $C_e$ is some fixed constant big enough (see the beginning of subsection \ref{G s1}). Using the properties of $\eta_Q$ and that $C_e$ is big enough, it is not difficult to show that $\lambda\big(D^{b-\{2^{-m-1}\}^{n}}_{m}\big)\gtrsim2^{-mn}$ for all $b\in\R^n$ such that $b\in B_\infty^m(\wit x)$ (recall that $x\in D$ and $j\geq m-1$). Therefore, by (\ref{claims formula Lambda}),
\begin{equation*}
\begin{split}
|\Lambda_m^{\mu,\lambda}(x,z,t\,;\,y)|&\lesssim2^{3nm}\LL^n\big(B_\infty^m(\wit{x})\cap B_\infty^m(\wit{z})
\cap B_\infty^m(\wit{t})\cap B_\infty^m(\wit{y})^c\big)\\
&\leq2^{3nm}\LL^n\big(B_\infty^m(\wit{z})\cap B_\infty^m(\wit{y})^c\big).
\end{split}
\end{equation*}
As in the proof of Claim \ref{G1 eq1}, we have
$\LL^n\big(B_\infty^m(\wit{z})\cap
B_\infty^m(\wit{y})^c\big)\lesssim2^{-m(n-1)-j}$ for all
$z\in\supp\varphi^ {\,2^{-j}}_{2^{-j-1}}(\cdot-y)$. Thus,
$|\Lambda_m^{\mu,\lambda}(x,z,t\,;\,y)|\lesssim2^{m(2n+1)-j}$, as
wished.

For the second estimate in Claim \ref{G4 eq1}, we argue as in (\ref{claim 4.1 formula2}). For $t_1$ and $t_2$ close enough,
\begin{equation*}
\begin{split}
|\Lambda_m^{\mu,\lambda}(x,&z,t_1\,;\,y)-\Lambda_m^{\mu,\lambda}(x,z,t_2\,;\,y)|\\
&\lesssim2^{3nm}\int
|\chi_{B_\infty^m(\wit{x})\cap B_\infty^m(\wit{z})\cap B_\infty^m(\wit{t_1})\cap B_\infty^m(\wit{y})^c}(b)
-\chi_{B_\infty^m(\wit{x})\cap B_\infty^m(\wit{z})\cap B_\infty^m(\wit{t_2})\cap B_\infty^m(\wit{y})^c}(b)|\,db\\
&\leq2^{3nm}\LL^n\big(B_\infty^m(\wit{t_1})\Delta B_\infty^m(\wit{t_2})\big)
\lesssim2^{3nm}|\wit{t_1}-\wit{t_2}|2^{-m(n-1)}\leq2^{m(2n+1)}|t_1-t_2|,
\end{split}
\end{equation*}
and the claim follows by letting $t_1\to t_2$.
\end{proof}
\begin{proof}[Proof of Claim \ref{G5 eq4}]
This claim is proved as the first estimate in Claim \ref{G4 eq1},
replacing $\mu$ by $\sigma_D$ (we only used that
$\mu(D^{\,b}_{m})\gtrsim2^{-mn}$ for all $b\in\R^n$, which also
holds for $\sigma_D$).
\end{proof}

\section{Proof of Theorem \ref{teorema salts curt martingala}}\label{s teorema salts curt martingala}
Given $x\in\Gamma$, let $\{\epsilon_{m}\}_{m\in\Z}$ be a
decreasing sequence of positive numbers such that
\begin{equation}\label{Short eq7}
S\mu(x)^2\leq2
\sum_{j\in\Z}\,\sum_{m\in\Z:\,\epsilon_{m},\epsilon_{m+1}\in I_j}
|(K\varphi_{\epsilon_{m+1}}^{\,\epsilon_{m}}*\mu)(x)|^{2},
\end{equation}
so $\{\epsilon_{m}\}_{m\in\Z}$ depends on $x$.

Fix $j\in\Z$ and assume that $x\in D$, for some $D\in\DD_j$. Let
$L_{D}$ be an $n$-plane that minimizes $\alpha(D)$ and let
$\sigma_{D}:=c_{D}\HH^{n}_{L_{D}}$ be a minimizing measure for
$\alpha(D)$. Let $L_D^x$ be the $n$-plane parallel to $L_D$
which contains $x$, and set $\sigma^x_{D}:=c_{D}\HH^{n}_{L^x_{D}}$.

By the antisymmetry of the function
$\varphi_{\epsilon_{m+1}}^{\,\epsilon_{m}}K$, and since $\sigma_D^x$ is a
Hausdorff measure on the $n$-plane $L_D^x$ and $x\in L_{D}^x$, we have
\begin{equation*}
(K\varphi_{\epsilon_{m+1}}^{\,\epsilon_{m}}*\sigma_D^x)(x)=\int\varphi_{\epsilon_{m+1}}^{\,\epsilon_{m}}(x-y)K(x-y)\,d\sigma_{D}^x(y)=0
\end{equation*}
for all $m\in\Z$. Therefore, we can decompose
\begin{equation}\label{Short eq1a}
(K\varphi_{\epsilon_{m+1}}^{\,\epsilon_{m}}*\mu)(x)=(K\varphi_{\epsilon_{m+1}}^{\,\epsilon_{m}}*(\mu-\sigma_D))(x)
+(K\varphi_{\epsilon_{m+1}}^{\,\epsilon_{m}}*(\sigma_D-\sigma_D^x))(x).
\end{equation}

For every $m\in\Z$ such that $\epsilon_m,\epsilon_{m+1}\in I_j$ we will prove the following inequalities:
\begin{gather}
|(K\varphi_{\epsilon_{m+1}}^{\,\epsilon_{m}}*(\mu-\sigma_D))(x)|\lesssim2^j|\epsilon_m-\epsilon_{m+1}|\alpha(D),
\label{Short eq1b}\\
|(K\varphi_{\epsilon_{m+1}}^{\,\epsilon_{m}}*(\sigma_D-\sigma_D^x))(x)|\lesssim2^{2j}|\epsilon_m-\epsilon_{m+1}|\dist(x,L_D).
\label{Short eq1c}
\end{gather}
Assume for a moment that these estimates hold. Then, by (\ref{Short eq1a}),
$$|(K\varphi_{\epsilon_{m+1}}^{\,\epsilon_{m}}*\mu)(x)|\lesssim2^j|\epsilon_m-\epsilon_{m+1}|\left(\alpha(D)+2^j\dist(x,L_D)\right).$$
Then, using (\ref{Short eq7}), we conclude that
\begin{equation*}
\begin{split}
\|S\mu\|^{2}_{L^{2}(\mu)}&\leq2
\sum_{j\in\Z}\sum_{D\in\DD_j}\int_D\sum_{m\in\Z:\,\epsilon_{m},\epsilon_{m+1}\in
I_j}|(K\varphi_{\epsilon_{m+1}}^{\,\epsilon_{m}}*\mu)(x)|^{2}\,d\mu(x)\\
&\lesssim\sum_{j\in\Z}\sum_{D\in\DD_j}\int_D\bigg(\alpha(D)+\frac{\dist(x,L_D)}{2^{-j}}\bigg)^{2}
\sum_{\begin{subarray}{c}m\in\Z:\\\epsilon_{m},\epsilon_{m+1}\in I_j\end{subarray}}\bigg(\frac{|\epsilon_m-\epsilon_{m+1}|}{2^{-j}}\bigg)^2\,d\mu(x)\\
&\lesssim\sum_{D\in\DD}\alpha(D)^{2}\mu(D)+
\sum_{D\in\DD}\int_{D}\bigg(\frac{\dist(x,L_{D})}{\ell(D)}\bigg)^{2}\,d\mu(x).
\end{split}
\end{equation*}
Notice that, under the integral sign on the right hand side of the first inequality above, $\epsilon_m$ and $\epsilon_{m+1}$ depend on $x$. It is not obvious that the resulting function is measurable. To deal with this issue more carefully, we might first ask $\{\epsilon_m\}_{m\in\Z}$ to lie in some finite set, prove the variational norm inequality with constants  independent of the set, and then enlarge the set. Then, by monotone convergence we would obtain the result with $\{\epsilon_m\}_{m\in\Z}$ restricted to a countable set dense in $(0,\infty)$. The final theorem would follow then from the continuity properties of the operators involved. This applies to other similar situations in the paper. However, for the
sake of conciseness, we will skip further details.

The second term on the right hand side of the last inequality coincides with $W_1\mu$ (see (\ref{R-E eq6})),
thus it is bounded (modulo constants) by $\sum_{D\in\DD}\big(\alpha(D)^2+\beta_2(D)^2\big)\mu(D)$, and Theorem \ref{teorema salts curt martingala} is proved.

It only remains to verify (\ref{Short eq1b}) and (\ref{Short eq1c}) for $x\in D\in\DD_j$ and $m\in\Z$ such that $\epsilon_m,\epsilon_{m+1}\in I_j$. First of all, notice that $\varphi_{\epsilon_{m+1}}^{\,\epsilon_{m}}$ satisfies
\begin{equation}\label{Short eq1}
\begin{split}
|\varphi_{\epsilon_{m+1}}^{\,\epsilon_{m}}(x-y)|&=\left|\varphi_\R\bigg(\frac{|\widetilde x-\widetilde y|}{\epsilon_{m+1}}\bigg)-\varphi_\R\bigg(\frac{|\widetilde x-\widetilde y|}{\epsilon_{m}}\bigg)\right|
\leq\|\varphi'_\R\|_{L^\infty(\R)}\left|\frac{|\widetilde x-\widetilde y|}{\epsilon_{m+1}}-\frac{|\widetilde x-\widetilde y|}{\epsilon_{m}}\right|\\
&=\|\varphi'_\R\|_{\infty}|\widetilde x-\widetilde y|\,\frac{\epsilon_{m}-\epsilon_{m+1}}{\epsilon_{m}\epsilon_{m+1}}\lesssim2^{j}|\epsilon_{m}-\epsilon_{m+1}|
\end{split}
\end{equation}
for all $y\in\supp\,\varphi_{\epsilon_{m+1}}^{\,\epsilon_m}(x-\cdot)$. For $i=1,\ldots,d$,
$$\partial_{y^i}\left(\varphi_{\epsilon_m}(x-y)\right)=
\varphi'_\R\left(\frac{|\widetilde x-\widetilde y|}{\epsilon_{m}}\right)\frac{y^i-x^i}{\epsilon_m|\widetilde x-\widetilde y|}\,\chi_{[1,n]}(i).$$
Hence,
\begin{equation*}
\begin{split}
\big|\partial_{y^i}\big(\varphi_{\epsilon_{m+1}}^{\,\epsilon_m}&(x-y)\big)\big|\leq\left|\varphi'_\R\left(\frac{|\widetilde x-\widetilde y|}{\epsilon_{m}}\right)\frac{1}{\epsilon_m}-
\varphi'_\R\left(\frac{|\widetilde x-\widetilde y|}{\epsilon_{m+1}}\right)\frac{1}{\epsilon_{m+1}}\right|\\
&\leq\left|\varphi'_\R\left(\frac{|\widetilde x-\widetilde y|}{\epsilon_{m}}\right)\right|\left|\frac{1}{\epsilon_{m}}-\frac{1}{\epsilon_{m+1}}\right|
+\left|\varphi'_\R\left(\frac{|\widetilde x-\widetilde y|}{\epsilon_{m}}\right)-\varphi'_\R\left(\frac{|\widetilde x-\widetilde y|}{\epsilon_{m+1}}\right)\right|\frac{1}{\epsilon_{m+1}}\\
&\leq\left(\|\varphi_\R'\|_{\infty}+\|\varphi''_\R\|_{\infty}\frac{|\widetilde x-\widetilde y|}{\epsilon_{m+1}}\right)\frac{\epsilon_{m}-\epsilon_{m+1}}{\epsilon_{m}\epsilon_{m+1}}.
\end{split}
\end{equation*}
Since $\epsilon_m,\epsilon_{m+1}\in I_j$, we deduce from the previous estimate that, for $x-y\in\supp\,\varphi_{\epsilon_{m+1}}^{\,\epsilon_m}$,
\begin{equation}\label{Short eq2}
|\nabla_{y}\left(\varphi_{\epsilon_{m+1}}^{\,\epsilon_m}(x-y)\right)|\lesssim\frac{\epsilon_{m}-\epsilon_{m+1}}{\epsilon_{m}\epsilon_{m+1}}\approx2^{2j}|\epsilon_{m}-\epsilon_{m+1}|.
\end{equation}

We are going to use (\ref{Short eq1}) and (\ref{Short eq2}) to prove (\ref{Short eq1b}) and (\ref{Short eq1c}). Let us start with (\ref{Short eq1b}). Since $\epsilon_m,\epsilon_{m+1}\in I_j$, we can assume that $\supp\,\varphi_{\epsilon_{m+1}}^{\,\epsilon_m}(x-\cdot)\subset B_D$, by taking $C_\Gamma$ big enough.

By (\ref{Short eq1}) and (\ref{Short eq2}), for all $y\in\supp\,\varphi_{\epsilon_{m+1}}^{\,\epsilon_m}(x-\cdot)$,
$$\left|\nabla_y\left(\varphi_{\epsilon_{m+1}}^{\,\epsilon_{m}}(x-y)K(x-y)\right)\right|\lesssim2^{j(n+2)}|\epsilon_{m}-\epsilon_{m+1}|,$$
hence
\begin{equation*}
\begin{split}
|(K\varphi_{\epsilon_{m+1}}^{\,\epsilon_{m}}*(\mu-\sigma_D))(x)|&\lesssim2^{j(n+2)}|\epsilon_{m}-\epsilon_{m+1}|\dist_{B_D}(\mu,\sigma_D)
\lesssim2^{j}|\epsilon_{m}-\epsilon_{m+1}|\alpha(D),
\end{split}
\end{equation*}
which gives (\ref{Short eq1b}).

In order to prove (\ref{Short eq1c}), set $L^x_{D}=\{(\widetilde t, a(\widetilde t))\in\R^d\,:\,\widetilde t\in\R^n\}$, where $a:\R^n\to\R^{d-n}$ is an appropriate affine map, and let $p:L_{D}\to L_{D}^x$ be the map defined by $p(t):=(\widetilde t,a(\widetilde t))$. Since $\Gamma$ is a Lipschitz graph, $a$ is well defined and $p$ is a homeomorphism. Let $p_{\sharp}\HH^{n}_{L_{D}}$ be the image measure of $\HH^{n}_{L_D}$ by $p$. It is easy to see that,
$|y-p(y)|\approx\dist(x,L_{D})$ for all $y\in L_{D}$. Notice also that the image measure $p_{\sharp}\HH^{n}_{L_{D}}$ coincides with $\HH_{L_{D}^x}^{n}$. Therefore, since $\varphi_{\epsilon_{m+1}}^{\,\epsilon_{m}}(x-y)$ only depends on $\widetilde x-\widetilde y$,
\begin{equation}\label{Short eq1d}
\begin{split}
(K\varphi_{\epsilon_{m+1}}^{\,\epsilon_{m}}*(\sigma_D-\sigma_D^x))&(x)
=c_{D}\int\varphi_{\epsilon_{m+1}}^{\,\epsilon_{m}}(x-y)K(x-y)\,d(\HH^{n}_{L_{D}}-p_{\sharp}\HH^{n}_{L_{D}})(y)\\
&=c_{D}\int\varphi_{\epsilon_{m+1}}^{\,\epsilon_{m}}(x-y)(K(x-y)-K(x-p(y)))\,d\HH^{n}_{L_{D}}(y).
\end{split}
\end{equation}
For $y\in\supp\,\varphi_{\epsilon_{m+1}}^{\,\epsilon_m}(x-\cdot)\cap L_{D},$ we have
\begin{equation*}
|K(x-y)-K(x-p(y))|\lesssim2^{j(n+1)}|y-p(y)|\approx2^{j(n+1)}\dist(x,L_{D}).
\end{equation*}
Plugging this estimate and (\ref{Short eq1}) into (\ref{Short eq1d}), we conclude that
\begin{equation*}
|(K\varphi_{\epsilon_{m+1}}^{\,\epsilon_{m}}*(\sigma_D-\sigma_D^x))(x)|\lesssim2^{2j}|\epsilon_{m}-\epsilon_{m+1}|\dist(x,L_{D}),
\end{equation*}
which gives (\ref{Short eq1c}); and the theorem follows.

\section{$L^2$ localization of $\VV_\rho\circ\TT_{\varphi}^{\HH^{n}_{\Gamma}}$ and $\OO\circ\TT_{\varphi}^{\HH^{n}_{\Gamma}}$}\label{seccio localitzacio}
From here till the end of the paper,
$\Gamma:=\{x\in \R^d\,:\,x=(\widetilde x,A(\widetilde x))\}$ will be the graph of a Lipschitz function $A:\R^n\to\R^{d-n}$, without any assumption on the support of $A$.

\begin{teo}\label{teorema localitzacio}
Let $\rho>2$. There exist $C_1,C_2>0$ such that, for every $f\in L^\infty(\HH^{n}_{\Gamma})$ supported in $\Gamma\cap D$ (where $D:=\widetilde D\times\R^{d-n}$ and $\widetilde D$ is a cube of $\R^n$),
\begin{eqnarray}
\int_{D}\big((\VV_\rho\circ\TT_{\varphi}^{\HH^{n}_{\Gamma}})f\big)^2\,d\HH^{n}_{\Gamma}&\leq&
C_1\|f\|^2_{L^\infty(\HH^{n}_{\Gamma})}\HH^{n}_{\Gamma}(D)\quad\text{and}\label{teorema localitzacio1}\\
\int_{D}\big((\OO\circ\TT_{\varphi}^{\HH^{n}_{\Gamma}})f\big)^2\,d\HH^{n}_{\Gamma}&\leq&C_2
\|f\|^2_{L^\infty(\HH^{n}_{\Gamma})}\HH^{n}_{\Gamma}(D).\label{teorema localitzacio2}
\end{eqnarray}
The constant $C_2$ does not depend on the fixed sequence that defines $\OO$.
\end{teo}
We will only give the proof of (\ref{teorema localitzacio1}), because the proof of (\ref{teorema localitzacio2}) follows by very similar arguments and computations.

We claim that it is enough to prove (\ref{teorema localitzacio1}) for all functions $f$ such that $f(x)\approx1$ for all $x\in\Gamma\cap D$. Otherwise, we consider $g(x):=\|f\|_{L^\infty(\HH^{n}_{\Gamma})}^{-1}f(x)+2\chi_{D}(x),$ which clearly satisfies $g(x)\approx1$ for all $x\in\Gamma\cap D$. Since $f=\|f\|_{L^\infty(\HH^{n}_{\Gamma})}\left(g-2\chi_{D}\right)$,
$$(\VV_\rho\circ\TT_{\varphi}^{\HH^{n}_{\Gamma}})f(x)\leq\|f\|_{L^\infty(\HH^{n}_{\Gamma})}\big((\VV_\rho\circ\TT_{\varphi}^{\HH^{n}_{\Gamma}})g(x)
+2(\VV_\rho\circ\TT_{\varphi}^{\HH^{n}_{\Gamma}})\chi_D(x)\big).$$
Applying (\ref{teorema localitzacio1}) to the functions $g$ and $\chi_D$, we finally get
$$\int_{D}\big((\VV_\rho\circ\TT_{\varphi}^{\HH^{n}_{\Gamma}})f\big)^2\,d\HH^{n}_{\Gamma}\lesssim\|f\|_{L^\infty(\HH^{n}_{\Gamma})}^2\HH^{n}_{\Gamma}(D).$$

Given $f$ and $D$ as in Theorem \ref{teorema localitzacio}, from now on, we assume that $f\approx1$ in $\Gamma\cap D$. Let $\widetilde z_D$ be the center of $\widetilde D$ and set $z_D:=(\widetilde z_D,A(\widetilde z_D))$. One can easily construct a Lipschitz function $A_D:\R^n\to\R^{d-n}$ such that $\Lip(A_D)\lesssim\Lip(A)$, $A_D(\widetilde x)=A(\widetilde z_D)$ for all $\widetilde x\in(3\widetilde D)^c$, and $A_D(\widetilde x)=A(\widetilde x)$ for all $\widetilde x\in\widetilde D$. Let $\Gamma_D$ be the graph of $A_D$ and define the measure $\mu:=\HH^{n}_{\Gamma_D\setminus D}+f\HH^{n}_{\Gamma_D\cap D}$. Notice that $\chi_{(3D)^c}\mu$ is supported in the $n$-plane $L:=\R^n\times\{A(\widetilde z_D)\}$ and $\chi_D\mu=f\HH^{n}_{\Gamma\cap D}$.

Since $f\approx1$ in $\Gamma\cap D$ and $\chi_D\mu=(1-\chi_{(3D)^c}-\chi_{3D\setminus D})\mu$, we can decompose
\begin{equation*}
\begin{split}
\int_{D}\big((\VV_\rho\circ\TT_{\varphi}^{\HH^{n}_{\Gamma}})&f\big)^2\,d\HH^{n}_{\Gamma}\approx\int_{D}\VV_\rho(K\varphi*(\chi_D\mu))^2\,d\mu\\
&\lesssim\int_{D}\left(\VV_\rho(K\varphi*\mu)+\VV_\rho(K\varphi*(\chi_{(3D)^c}\mu))
+\VV_\rho(K\varphi*(\chi_{3D\setminus D}\mu))\right)^2\,d\mu.
\end{split}
\end{equation*}

In the next subsections, we will see that $\int_{D}\VV_\rho(K\varphi*\mu)^2\,d\mu$,
$\int_{D}\VV_\rho(K\varphi*(\chi_{(3D)^c}\mu))^2\,d\mu$, and $\int_{D}\VV_\rho(K\varphi*(\chi_{3D\setminus D}\mu))^2\,d\mu$ are bounded by $C\mu(D)$, and (\ref{teorema localitzacio1}) will follow.

\subsection{Proof of $\int_D\VV_\rho(K\varphi*\mu)^2\,d\mu\lesssim\mu(D)$}
Fix $x\in\supp\mu$, and let $\{\epsilon_m\}_{m\in\Z}$ be a decreasing sequence of positive numbers (which depends on $x$) such that
\begin{equation}\label{teorema localitzacio8}
\left(\VV_\rho(K\varphi*\mu)(x)\right)^{\rho}\leq2
\sum_{m\in\Z}|(K\varphi_{\epsilon_{m+1}}^{\,\epsilon_m}*\mu)(x)|^{\rho}.
\end{equation}
For $j\in\Z$ we set $I_j:=[2^{-j-1},2^{-j})$. We decompose $\Z=\SSS\cup\LL$, where
\begin{equation}\label{teorema interpolacio H1-5}
\begin{split}
&\SSS:=\bigcup_{j\in\Z}\SSS_j,\quad\SSS_j:=\{m\in\Z\,:\,\epsilon_m,\epsilon_{m+1}\in I_j\},\\
&\LL:=\{m\in\Z\,:\,\epsilon_m\in I_i,\,\epsilon_{m+1}\in I_j\text{ for }i<j\}.
\end{split}
\end{equation}
Then, $\sum_{m\in\Z}|(K\varphi_{\epsilon_{m+1}}^{\,\epsilon_m}*\mu)(x)|^{\rho}=
\sum_{m\in\SSS}|(K\varphi_{\epsilon_{m+1}}^{\,\epsilon_m}*\mu)(x)|^{\rho}+
\sum_{m\in\LL}|(K\varphi_{\epsilon_{m+1}}^{\,\epsilon_m}*\mu)(x)|^{\rho}$.

Notice that, since the $\ell^\rho(\Z)$-norm is smaller than the $\ell^2(\Z)$-norm for $\rho>2$,
\begin{equation}\label{teorema localitzacio5}
\sum_{m\in\SSS}|(K\varphi_{\epsilon_{m+1}}^{\,\epsilon_m}*\mu)(x)|^{\rho}\leq S\mu(x)^\rho,
\end{equation}
where $S\mu(x)$ has been defined in Theorem \ref{teorema salts curt martingala}.

Let us now estimate the sum over the indices $m\in\LL$. For $m\in\Z$ we define $j(\epsilon_m)$ as the integer such that $\epsilon_m\in I_{j(\epsilon_m)}$. Since $\{\epsilon_m\}_{m\in\Z}$ is decreasing, given $j\in\Z$, there is at most one index $m\in\LL$ such that $\epsilon_m\in I_j$. Thus, if $k,m\in\LL$ and $k<m$, one has $j(\epsilon_k)<j(\epsilon_m)$.

With this notation, we have
\begin{equation}\label{teorema localitzacio6}
\begin{split}
\sum_{m\in\LL}|(K\varphi&_{\epsilon_{m+1}}^{\,\epsilon_m}*\mu)(x)|^{\rho}
=\sum_{m\in\LL}|(K\varphi_{\epsilon_{m+1}}*\mu)(x)-(K\varphi_{\epsilon_m}*\mu)(x)|^{\rho}\\
&\lesssim\sum_{m\in\LL}|(K\varphi_{\epsilon_{m+1}}*\mu)(x)-(K\varphi_{2^{-j(\epsilon_{m+1})-1}}*\mu)(x)|^{\rho}\\
&\quad+\sum_{m\in\LL}|(K\varphi_{2^{-j(\epsilon_{m+1})-1}}*\mu)(x)-E_{j(\epsilon_{m+1})+1}\mu(x)|^{\rho}\\
&\quad+\sum_{m\in\LL}|E_{j(\epsilon_{m+1})+1}\mu(x)-E_{j(\epsilon_m)+1}\mu(x)|^{\rho}\\
&\quad+\sum_{m\in\LL}|E_{j(\epsilon_m)+1}\mu(x)-(K\varphi_{2^{-j(\epsilon_m)-1}}*\mu)(x)|^{\rho}\\
&\quad+\sum_{m\in\LL}|(K\varphi_{2^{-j(\epsilon_m)-1}}*\mu)(x)-(K\varphi_{\epsilon_m}*\mu)(x)|^{\rho}\\
&\lesssim S\mu(x)^\rho+W\mu(x)^\rho+\VV_\rho(E\mu)(x)^\rho,
\end{split}
\end{equation}
where $S\mu(x)$ and $W\mu(x)$ have been defined in Theorems
\ref{teorema salts curt martingala} and \ref{teorema salts llarg
martingala}, respectively, and $\VV_\rho(E\mu)$ is the
$\rho$-variation of the average of martingales $\{E_m\mu\}_{m\in\Z}$
from subsection \ref{ss particular martingale}. Therefore, by
(\ref{teorema localitzacio5}), (\ref{teorema localitzacio6}), and
(\ref{teorema localitzacio8}), we deduce that
\begin{equation*}
\VV_\rho(K\varphi*\mu)(x)\lesssim S\mu(x)+W\mu(x)+\VV_\rho(E\mu)(x)
\end{equation*}
for all $x\in\supp\mu$, and so
\begin{equation}\label{teorema localitzacio7}
\int_D\VV_\rho(K\varphi*\mu)^2\,d\mu\lesssim \|S\mu\|^2_{L^2(\mu)}+\|W\mu\|^2_{L^2(\mu)}+\|\VV_\rho(E\mu)\|^2_{L^2(\mu)}.
\end{equation}

Clearly, Theorem \ref{osc-var martingala2}, Theorem \ref{teorema salts llarg martingala}, and Theorem \ref{teorema salts curt martingala} can be applied to the measure $\mu$, because $\supp\mu$ is a translation of the graph of a Lipschitz function with compact support. These theorems in combination with (\ref{teorema localitzacio7}) yield
\begin{equation}\label{teorema localitzacio9}
\int_D\VV_\rho(K\varphi*\mu)^2\,d\mu\leq C_1\bigg(\mu(3D)+\sum_{Q\in\DD}\big(\alpha_{\mu}(C_2Q)^{2}+\beta_{2,\mu}(Q)^2\big)\mu(Q)\bigg),\\
\end{equation}
where $C_1,C_2>0$ only depend on $n$, $d$, $K$, $\Lip(A)$, and $\rho$ (the condition $\rho>2$ is used to ensure the $L^2$ boundedness of $\VV_\rho(E\mu)$). Obviously, $\mu(3D)\approx\mu(D)$ and, since $\chi_{(3D)^c}\mu$ coincides with the $n$-dimensional Hausdorff measure on an $n$-plane, using Remark \ref{remark packing} it is easy to check that $\sum_{Q\in\DD}\big(\alpha_{\mu}(C_2Q)^{2}+\beta_{2,\mu}(Q)^2\big)\mu(Q)\lesssim\mu(3D)$. Hence, we conclude that $\int_D\VV_\rho(K\varphi*\mu)^2\,d\mu\lesssim\mu(D)$ by (\ref{teorema localitzacio9}).

\subsection{Proof of $\int_D\VV_\rho(K\varphi*(\chi_{(3D)^c}\mu))^2\,d\mu\lesssim\mu(D)$}\label{ss localitzacio1}
Fix $x\in\supp\mu\cap D$, and let $\{\epsilon_m\}_{m\in\Z}$ be a decreasing sequence of positive numbers (which depends on $x$) such that
\begin{equation}\label{teorema localitzacio10}
\left(\VV_\rho(K\varphi*(\chi_{(3D)^c}\mu))(x)\right)^\rho\leq2
\sum_{m\in\Z}|(K\varphi_{\epsilon_{m+1}}^{\,\epsilon_m}*(\chi_{(3D)^c}\mu))(x)|^{\rho}.
\end{equation}

Recall that $\widetilde z_D$ is the center of $\widetilde D$, $z_D:=(\widetilde z_D,A(\widetilde z_D))$ and
$L:=\R^n\times\{A(\widetilde z_D)\}$.
Since $\chi_{(3D)^c}\mu=\HH^n_{L\setminus3D}$ and $z_D$ is the center of $L\cap D$, $(K\varphi_{\epsilon}^{\delta}*(\chi_{(3D)^c}\mu))(z_D)=0$ for all $0<\epsilon\leq\delta$. Thus,
$|(K\varphi_{\epsilon_{m+1}}^{\,\epsilon_m}*(\chi_{(3D)^c}\mu))(x)|=
|(K\varphi_{\epsilon_{m+1}}^{\,\epsilon_m}*(\chi_{(3D)^c}\mu))(x)
-(K\varphi_{\epsilon_{m+1}}^{\,\epsilon_m}*(\chi_{(3D)^c}\mu))(z_D)|\leq\Theta1_m+\Theta2_m,$ where
\begin{equation}\label{teorema localitzacio11}
\begin{split}
&\Theta1_m:=\int_{(3D)^c}\varphi_{\epsilon_{m+1}}^{\,\epsilon_m}(x-y)|K(x-y)-K(z_D-y)|\,d\mu(y),\\
&\Theta2_m:=\int_{(3D)^c}|\varphi_{\epsilon_{m+1}}^{\,\epsilon_m}(x-y)-
\varphi_{\epsilon_{m+1}}^{\,\epsilon_m}(z_D-y)|K(z_D-y)\,d\mu(y).
\end{split}
\end{equation}

Since $x\in\supp\mu\cap D$ and $A$ is a Lipschitz function, we have $|x-z_D|\lesssim\ell(D)$, and then
$|K(x-y)-K(z_D-y)|\lesssim|x-z_D||z_D-y|^{-n-1}\lesssim\ell(D)|z_D-y|^{-n-1}$ for all $y\in(3D)^c$. Therefore, using that $\sum_{m\in\Z}\varphi_{\epsilon_{m+1}}^{\,\epsilon_m}\leq1$ and that $\rho>1$,
\begin{equation}\label{teorema localitzacio12}
\begin{split}
\bigg(\sum_{m\in\Z}\Theta1_m^\rho\bigg)^{1/\rho}\leq\sum_{m\in\Z}\Theta1_m\lesssim
\int_{(3D)^c}\ell(D)|z_D-y|^{-n-1}\,d\mu(y)\lesssim1.
\end{split}
\end{equation}

To deal with $\Theta2_m$, we decompose $\Z=\SSS\cup\LL$ as in (\ref{teorema interpolacio H1-5}). As before, given $m\in\Z$, let $j(\epsilon_m)$ be the integer such that $\epsilon_m\in I_{j(\epsilon_m)}$.
Observe that
$$\supp\,\varphi_{\epsilon_{m+1}}^{\,\epsilon_m}(x-\cdot)\subset A(\wit x,2.1\sqrt{n}2^{-j(\epsilon_{m+1})-1},3\sqrt{n}2^{-j(\epsilon_m)})\times\R^{d-n}=:A_m(x).$$
Notice also that the sets $A_m(x)$ have finite overlap for $m\in\LL$, and the same is true for the sets $A'_j(x):=A(\wit x,2.1\sqrt{n}2^{-j-1},3\sqrt{n}2^{-j})\times\R^{d-n}$ for $j\in\Z$. The same observations hold if we replace $x$ by $z_D$ (and $\wit x$ by $\wit z_D$). Obviously, $A_m(x)\subset A'_j(x)$ (and $A_m(z_D)\subset A'_j(z_D)$) for $m\in\SSS_j$.

Assume that $m\in\SSS$. With the same computations as those carried out in (\ref{Short eq2}), one can easily prove that, for all $z-y\in\supp\varphi_{\epsilon_{m+1}}^{\,\epsilon_m}$,
\begin{gather*}
|\nabla_z\big(\varphi_{\epsilon_{m+1}}^{\,\epsilon_m}(z-y)\big)|\lesssim
\left(\|\varphi_\R'\|_{L^\infty(\R)}
+\|\varphi''_\R\|_{L^\infty(\R)}\frac{|\widetilde z-\widetilde y|}{\epsilon_{m+1}}\right)
\frac{\epsilon_m-\epsilon_{m+1}}{\epsilon_m\epsilon_{m+1}}
\lesssim2^{j(\epsilon_m)}\frac{|\epsilon_m-\epsilon_{m+1}|}{|\widetilde z-\widetilde y|},
\end{gather*}
because $|\widetilde z-\widetilde y|\approx\epsilon_m\approx\epsilon_{m+1}\approx2^{j(\epsilon_m)}$ for all $z-y\in\supp\varphi_{\epsilon_{m+1}}^{\,\epsilon_m}$ and $m\in\SSS$. In particular, if $z\in D$ and $y\in(3D)^c$, $|\nabla_z\big(\varphi_{\epsilon_{m+1}}^{\,\epsilon_m}(z-y)\big)|\lesssim
2^{j(\epsilon_m)}|\epsilon_m-\epsilon_{m+1}||\widetilde z_D-\widetilde y|^{-1}$.
Hence,
\begin{equation*}
\Theta2_m\lesssim\int_{(A_m(x)\,\cup\,A_m(z_D))\setminus3D}
\ell(D)2^{j(\epsilon_m)}\frac{|\epsilon_m-\epsilon_{m+1}|}{|\wit z_D-\wit y|^{n+1}}\,d\mu(y),
\end{equation*}
and then,
\begin{equation}\label{teorema localitzacio13}
\begin{split}
\bigg(\sum_{m\in\SSS}\Theta2_m^\rho\bigg)^{1/\rho}
&\lesssim\sum_{m\in\SSS}\int_{(A_m(x)\,\cup\,A_m(z_D))\setminus3D}
\ell(D)2^{j(\epsilon_m)}\frac{|\epsilon_m-\epsilon_{m+1}|}{|\wit z_D-\wit y|^{n+1}}\,d\mu(y)\\
&\leq\sum_{j\in\Z}\int_{(A'_j(x)\,\cup\,A'_j(z_D))\setminus3D}
\frac{\ell(D)}{|\wit z_D-\wit y|^{n+1}}\sum_{m\in\SSS_j}
\frac{|\epsilon_m-\epsilon_{m+1}|}{2^{-j}}\,d\mu(y)\\
&\lesssim\int_{(3D)^c}\frac{\ell(D)}{|\wit z_D-\wit y|^{n+1}}\,d\mu(y)\lesssim1.
\end{split}
\end{equation}

Assume now that $m\in\LL$. It is easy to check that $|\nabla_z\big(\varphi_{\epsilon_{m+1}}^{\,\epsilon_m}(z-y)\big)|\lesssim|\wit z-\wit y|^{-1}$ for all $z,y\in\R^d$. So, if also $z\in D$ and $y\in(3D)^c$, $|\nabla_z\big(\varphi_{\epsilon_{m+1}}^{\,\epsilon_m}(z-y)\big)|\lesssim|\wit z_D-\wit y|^{-1}$. Therefore,
\begin{equation}\label{teorema localitzacio14}
\begin{split}
\bigg(\sum_{m\in\LL}\Theta2_m^\rho\bigg)^{1/\rho}
&\lesssim\sum_{m\in\LL}\int_{(A_m(x)\,\cup\,A_m(z_D))\setminus3D}\frac{\ell(D)}{|z_D-y|^{n+1}}\,d\mu(y)\\
&\lesssim\int_{(3D)^c}\frac{\ell(D)}{|\wit z_D-\wit y|^{n+1}}\,d\mu(y)\lesssim1.
\end{split}
\end{equation}

Finally combining (\ref{teorema localitzacio12}), (\ref{teorema localitzacio13}), and (\ref{teorema localitzacio14}), with (\ref{teorema localitzacio10}) and the fact that 
$(K\varphi_{\epsilon_{m+1}}^{\,\epsilon_m}*(\chi_{(3D)^c}\mu))(x)|
\leq\Theta1_m+\Theta2_m$, we conclude that
\begin{equation*}
\VV_\rho(K\varphi*(\chi_{(3D)^c}\mu))(x)
\lesssim\bigg(\sum_{m\in\Z}\Theta1_m^\rho\bigg)^{1/\rho}
+\bigg(\sum_{m\in\SSS}\Theta2_m^\rho\bigg)^{1/\rho}
+\bigg(\sum_{m\in\LL}\Theta2_m^\rho\bigg)^{1/\rho}\lesssim1
\end{equation*}
for all $x\in\supp\mu\cap D$. Therefore, $\int_D\VV_\rho(K\varphi*(\chi_{(3D)^c}\mu))^2\,d\mu\lesssim\mu(D)$.

\subsection{Proof of $\int_D\VV_\rho(K\varphi*(\chi_{3D\setminus D}\mu))^2\,d\mu\lesssim\mu(D)$}
Fix $x\in\supp\mu\cap D$. Since $\rho>1$,
\begin{equation*}
\begin{split}
\VV_\rho(K\varphi*&(\chi_{3D\setminus D}\mu))(x)=
\sup_{\{\epsilon_m\}}\bigg(\sum_{m\in\Z}
\bigg|\int_{3D\setminus D}\varphi_{\epsilon_{m+1}}^{\,\epsilon_m}(x-y)K(x-y)\,d\mu(y)\bigg|^{\rho}\bigg)^{1/\rho}\\
&\leq\sup_{\{\epsilon_m\}}\sum_{m\in\Z}
\int_{3D\setminus D}\varphi_{\epsilon_{m+1}}^{\,\epsilon_m}(x-y)|K(x-y)|\,d\mu(y)
\leq\int_{3D\setminus D}|K(x-y)|\,d\mu(y).
\end{split}
\end{equation*}
By a standard computation, one can show that $$\int_D\bigg(\int_{3D\setminus D}|K(x-y)|\,d\mu(y)\bigg)^{2}d\mu(x)\lesssim\mu(D),$$
hence we conclude that $\int_D\VV_\rho(K\varphi*(\chi_{3D\setminus D}\mu))^2d\mu\lesssim\mu(D).$

\section{$L^p$ and endpoint estimates for $\VV_\rho\circ\TT_{\varphi}^{\HH^{n}_{\Gamma}}$ and
$\OO\circ\TT_{\varphi}^{\HH^{n}_{\Gamma}}$}\label{seccio Lp suau}

We denote by $H^1(\HH^{n}_{\Gamma})$ and $BMO(\HH^{n}_{\Gamma})$ the (atomic) Hardy space and the space of functions with bounded mean oscillation, respectively, with respect to the measure $\HH^{n}_{\Gamma}$. These spaces are defined as the classical $H^1(\R^d)$ and $BMO(\R^d)$ (see \cite[Chapter 6]{Duoandi}, for example), but by replacing the true cubes of $\R^d$ by our special v-cubes.

\begin{teo}\label{teorema interpolacio}
Let $\rho>2$. The operators $\VV_\rho\circ\TT_{\varphi}^{\HH^{n}_{\Gamma}}$ and $\OO\circ\TT_{\varphi}^{\HH^{n}_{\Gamma}}$ are bounded
\begin{itemize}
\item in $L^p(\HH^{n}_{\Gamma})$ for $1<p<\infty$,
\item from $H^1(\HH^{n}_{\Gamma})$ to $L^1(\HH^{n}_{\Gamma})$,
\item from $L^1(\HH^{n}_{\Gamma})$ to $L^{1,\infty}(\HH^{n}_{\Gamma})$, and
\item from $L^\infty(\HH^{n}_{\Gamma})$ to $BMO(\HH^{n}_{\Gamma})$,
\end{itemize}
and the norm of $\OO\circ\TT_{\varphi}^{\HH^{n}_{\Gamma}}$ in the cases above is bounded independently of the sequence that defines $\OO$.
\end{teo}

We will only give the proof of Theorem \ref{teorema interpolacio} in the case of the $\rho$-variation, because the proof for the oscillation follows by analogous arguments.

\subsection{The operator $\VV_\rho\circ\TT_{\varphi}^{\HH^{n}_{\Gamma}}\,:\,H^1(\HH^{n}_{\Gamma})\to L^1(\HH^{n}_{\Gamma})$ is bounded}\label{teorema interpolacio s1}
Fix a cube $\widetilde D\subset\R^n$ and set $D:=\widetilde D\times\R^{d-n}$. Let $f$ be an atom, i.e., a function defined on $\Gamma$ and such that
\begin{equation}\label{teorema interpolacio H1}
\supp f\subset D,\quad\|f\|_{L^\infty(\HH^n_\Gamma)}\leq\frac{1}{\HH^{n}_{\Gamma}(D)},\quad\text{and}\quad\int f\,d\HH^{n}_{\Gamma}=0.
\end{equation}
We have to prove that $\int(\VV_\rho\circ\TT_{\varphi}^{\HH^{n}_{\Gamma}})f\,d\HH^{n}_{\Gamma}\leq C$, for some constant $C>0$ which does not depend on $f$ or $D$. Since $(\VV_\rho\circ\TT_{\varphi}^{\HH^{n}_{\Gamma}})f(x)$ is well defined and non negative for $f\in L^1(\HH^{n}_{\Gamma})$, the uniform boundedness of $\VV_\rho\circ\TT_{\varphi}^{\HH^{n}_{\Gamma}}$ on atoms yields its boundedness from $H^1(\HH^{n}_{\Gamma})$ to $L^1(\HH^{n}_{\Gamma})$ by standard arguments. We omit the details.

First of all, by H\"{o}lder's inequality, Theorem \ref{teorema localitzacio}, and (\ref{teorema interpolacio H1}),
\begin{equation*}
\begin{split}
\int_{3D}(\VV_\rho\circ\TT_{\varphi}^{\HH^{n}_{\Gamma}})f\,d\HH^{n}_{\Gamma}&\leq\HH^{n}_{\Gamma}(3D)^{1/2}
\bigg(\int_{3D}\big((\VV_\rho\circ\TT_{\varphi}^{\HH^{n}_{\Gamma}})f\big)^2\,d\HH^{n}_{\Gamma}\bigg)^{1/2}\\
&\lesssim\HH^{n}_{\Gamma}(3D)^{1/2}\left(\|f\|^2_{L^\infty(\HH^n_\Gamma)}\HH^{n}_{\Gamma}(3D)\right)^{1/2}\lesssim1.
\end{split}
\end{equation*}
Thus, it remains to prove that $\int_{(3D)^c}(\VV_\rho\circ\TT_{\varphi}^{\HH^{n}_{\Gamma}})f\,d\HH^{n}_{\Gamma}\leq C$.

Given $x\in\Gamma\setminus3D$, let $\{\epsilon_m\}_{m\in\Z}$ be a decreasing sequence of positive numbers (which depends on $x$) such that
\begin{equation}\label{teorema interpolacio H1-1}
\big((\VV_\rho\circ\TT_{\varphi}^{\HH^{n}_{\Gamma}})f(x)\big)^\rho\leq2
\sum_{m\in J}|(K\varphi_{\epsilon_{m+1}}^{\,\epsilon_m}*(f\HH^n_\Gamma))(x)|^{\rho},
\end{equation}
where $J:=\{m\in\Z\,:\,\supp\,\varphi_{\epsilon_{m+1}}^{\,\epsilon_m}(x-\cdot)\,\cap\,\supp f\neq\emptyset\}$.

Set $z_D:=(\widetilde z_D,A(\widetilde z_D))\in D\cap\Gamma$, where $\widetilde z_D$ is the center of $\widetilde D$. By (\ref{teorema interpolacio H1}), we have $\int\varphi_{\epsilon}^{\delta}(x-z_D)K(x-z_D)f(y)\,d\HH^n_\Gamma(y)=0$
for all $0<\epsilon\le\delta$. Thus, given $m\in J$, we can decompose
\begin{equation*}
\begin{split}
(K\varphi_{\epsilon_{m+1}}^{\,\epsilon_m}*(f\HH^n_\Gamma))(x)
&=\int\varphi_{\epsilon_{m+1}}^{\,\epsilon_m}(x-y)\left(K(x-y)-K(x-z_D)\right)f(y)\,d\HH^n_\Gamma(y)\\
&\quad+\int\big(\varphi_{\epsilon_{m+1}}^{\,\epsilon_m}(x-y)-\varphi_{\epsilon_{m+1}}^{\,\epsilon_m}(x-z_D)\big)K(x-z_D)f(y)\,d\HH^n_\Gamma(y),
\end{split}
\end{equation*}
and we obtain $|(K\varphi_{\epsilon_{m+1}}^{\,\epsilon_m}*(f\HH^n_\Gamma))(x)|\leq\|f\|_{L^\infty(\HH^n_\Gamma)}(\Theta1_m+\Theta2_m)$, where
\begin{equation*}
\begin{split}
\Theta1_m:&=\int_D\varphi_{\epsilon_{m+1}}^{\,\epsilon_m}(x-y)\left|K(x-y)-K(x-z_D)\right|\,d\HH^n_\Gamma(y),\\
\Theta2_m:&=\int_D\big|\varphi_{\epsilon_{m+1}}^{\,\epsilon_m}(x-y)-\varphi_{\epsilon_{m+1}}^{\,\epsilon_m}(x-z_D)\big||K(x-z_D)|\,d\HH^n_\Gamma(y).
\end{split}
\end{equation*}

The term $\Theta1_m$ can be easily handled. For $x\in\Gamma\setminus3D$, we have
\begin{equation}\label{teorema interpolacio H1-2}
\Theta1_m\lesssim\ell(D)\,\dist(x,D)^{-n-1}\int_D\varphi_{\epsilon_{m+1}}^{\,\epsilon_m}(x-y)\,d\HH^n_\Gamma(y).
\end{equation}

Let us estimate $\Theta2_m$. Decompose $J=\SSS\cup\LL$, where $\SSS$ and $\LL$ are as in (\ref{teorema interpolacio H1-5}) but replacing $m\in\Z$ by $m\in J$, and as before, let $j(\epsilon_m)$ be the integer such that $\epsilon_m\in I_{j(\epsilon_m)}$.
Using that $x\in\Gamma\setminus3D$ and $\supp f\subset D$, one can easily check that $\LL$ contains a finite number of elements, and this number only depends on $n$ and $d$. Similarly, $\SSS_j=\emptyset$ for all $j\in\Z$ except on a finite number which only depends on $n$ and $d$.

Assume that $m\in\SSS$. With the same computations as those in (\ref{Short eq2}), one can prove that, for all $y\in\supp\varphi_{\epsilon_{m+1}}^{\,\epsilon_m}(x-\cdot)$,
$|\nabla_y\big(\varphi_{\epsilon_{m+1}}^{\,\epsilon_m}(x-y)\big)|
\lesssim2^{j(\epsilon_m)}|\epsilon_m-\epsilon_{m+1}||\widetilde x-\widetilde y|^{-1},$
because $|\widetilde x-\widetilde y|\approx\epsilon_m\approx\epsilon_{m+1}\approx2^{-j(\epsilon_m)}$ for all $y\in\supp\varphi_{\epsilon_{m+1}}^{\,\epsilon_m}(x-\cdot)$. Thus,
\begin{equation}\label{teorema interpolacio H1-3}
\Theta2_m\lesssim\ell(D)^{n+1}\,\dist(x,D)^{-n-1}2^{j(\epsilon_m)}|\epsilon_m-\epsilon_{m+1}|.
\end{equation}

Assume now that $m\in\LL$. It is easy to verify that $|\nabla_y\big(\varphi_{\epsilon_{m+1}}^{\,\epsilon_m}(x-y)\big)|
\lesssim|\widetilde x-\widetilde y|^{-1},$ so $\Theta2_m\lesssim\ell(D)^{n+1}\,\dist(x,D)^{-n-1}$.

Combining this last estimate with (\ref{teorema interpolacio H1-2}), (\ref{teorema interpolacio H1-3}), the fact that $|(K\varphi_{\epsilon_{m+1}}^{\,\epsilon_m}*(f\HH^n_\Gamma))(x)|\leq\|f\|_{L^\infty(\HH^n_\Gamma)}(\Theta1_m+\Theta2_m)$, the remark on $\SSS$ and $\LL$ made just after (\ref{teorema interpolacio H1-2}), (\ref{teorema interpolacio H1-1}), and that $\rho>1$, we have that, for all $x\in\Gamma\setminus3D$,
\begin{equation*}
\begin{split}
(\VV_\rho&\circ\TT_{\varphi}^{\HH^{n}_{\Gamma}})f(x)\lesssim\|f\|_{L^\infty(\HH^n_\Gamma)}\bigg(\sum_{m\in J}\Theta1_m+\sum_{m\in\SSS}\Theta2_m+\sum_{m\in\LL}\Theta2_m\bigg)\\
&\lesssim\frac{\|f\|_{L^\infty(\HH^n_\Gamma)}\ell(D)^{n+1}}{\dist(x,D)^{n+1}}\,
\bigg(\sum_{m\in J}\int_D\frac{\varphi_{\epsilon_{m+1}}^{\,\epsilon_m}(x-y)}{\ell(D)^n}\,d\HH^n_\Gamma(y)+\sum_{m\in \SSS}\frac{|\epsilon_m-\epsilon_{m+1}|}{2^{-j(\epsilon_m)}}+\sum_{m\in \LL}1\bigg)\\
&\lesssim\frac{\|f\|_{L^\infty(\HH^n_\Gamma)}\ell(D)^{n+1}}{\dist(x,D)^{n+1}}.
\end{split}
\end{equation*}
Then, using (\ref{teorema interpolacio H1}) and standard estimates, we conclude that
\begin{equation*}
\begin{split}
\int_{(3D)^c}(\VV_\rho\circ\TT_{\varphi}^{\HH^{n}_{\Gamma}})f(x)\,d\HH^n_\Gamma(x)
\lesssim\int_{(3D)^c}\frac{\|f\|_{L^\infty(\HH^n_\Gamma)}\ell(D)^{n+1}}{\dist(x,D)^{n+1}}\,d\HH^n_\Gamma(x)\lesssim1.
\end{split}
\end{equation*}

\subsection{The operator $\VV_\rho\circ\TT_{\varphi}^{\HH^{n}_{\Gamma}}\,:\,L^\infty(\HH^{n}_{\Gamma})\to BMO(\HH^{n}_{\Gamma})$ is bounded}\label{teorema interpolacio s2}
We have to prove that there exists a constant $C>0$ such that, for any $f\in L^\infty(\HH^{n}_{\Gamma})$ and any cube $\widetilde D\subset\R^n$, there exists some constant $c$ depending on $f$ and $\wit D$ such that $$\int_D\big|(\VV_\rho\circ\TT_{\varphi}^{\HH^{n}_{\Gamma}})f-c\big|\,d\HH^{n}_{\Gamma}\leq C\|f\|_{L^\infty(\HH^n_\Gamma)}\HH^n_\Gamma(D).$$

Let $f$ and $D$ be as above, and set $f_1:=f\chi_{3D}$ and $f_2:=f-f_1$.
First of all, by H\"{o}lder's inequality and Theorem \ref{teorema localitzacio}, we have
\begin{equation*}
\begin{split}
\int_{D}(\VV_\rho\circ\TT_{\varphi}^{\HH^{n}_{\Gamma}})f_1\,d\HH^{n}_{\Gamma}&\leq\HH^{n}_{\Gamma}(D)^{1/2}
\bigg(\int_{3D}\big((\VV_\rho\circ\TT_{\varphi}^{\HH^{n}_{\Gamma}})f_1\big)^2\,d\HH^{n}_{\Gamma}\bigg)^{1/2}\\
&\lesssim\HH^{n}_{\Gamma}(D)^{1/2}\left(\|f_1\|^2_{L^\infty(\HH^n_\Gamma)}\HH^{n}_{\Gamma}(3D)\right)^{1/2}\lesssim\|f\|_{L^\infty(\HH^n_\Gamma)}\HH^{n}_{\Gamma}(D).
\end{split}
\end{equation*}

Notice that $|(\VV_\rho\circ\TT_{\varphi}^{\HH^{n}_{\Gamma}})(f_1+f_2)-(\VV_\rho\circ\TT_{\varphi}^{\HH^{n}_{\Gamma}})f_2|\leq
(\VV_\rho\circ\TT_{\varphi}^{\HH^{n}_{\Gamma}})f_1$, because $\VV_\rho\circ\TT_{\varphi}^{\HH^{n}_{\Gamma}}$ is sublinear and positive. Then, for any $c\in\R$,
\begin{equation*}
\begin{split}
|(\VV_\rho\circ\TT_{\varphi}^{\HH^{n}_{\Gamma}})(f_1+f_2)-c|
&\leq|(\VV_\rho\circ\TT_{\varphi}^{\HH^{n}_{\Gamma}})(f_1+f_2)-(\VV_\rho\circ\TT_{\varphi}^{\HH^{n}_{\Gamma}})f_2|
+|(\VV_\rho\circ\TT_{\varphi}^{\HH^{n}_{\Gamma}})f_2-c|\\
&\leq(\VV_\rho\circ\TT_{\varphi}^{\HH^{n}_{\Gamma}})f_1
+|(\VV_\rho\circ\TT_{\varphi}^{\HH^{n}_{\Gamma}})f_2-c|,
\end{split}
\end{equation*}
hence we are reduced to prove that, for some constant $c\in\R$,
\begin{equation}\label{teorema interpolacio Linf1}
\int_D\big|(\VV_\rho\circ\TT_{\varphi}^{\HH^{n}_{\Gamma}})f_2-c\big|\,d\HH^{n}_{\Gamma}\leq C\|f\|_{L^\infty(\HH^n_\Gamma)}\HH^n_\Gamma(D).
\end{equation}

Set $z_D:=(\widetilde z_D,A(\widetilde z_D))$, where $z_D$ is the center of $\wit D$, and take
$c:=(\VV_\rho\circ\TT_{\varphi}^{\HH^{n}_{\Gamma}})f_2(z_D)$. We may assume that $c<\infty$ (this is the case if, for example, $f$ has compact support). 

Given a family $a:=\{a_\epsilon\}_{\epsilon>0}\subset\C$, define its $\rho$-variation semi-norm to be 
$$\|a\|_{\VV_\rho}=\sup_{\{\epsilon_{m}\}}\bigg(\sum_{m\in\Z}
|a_{\epsilon_{m+1}}-a_{\epsilon_{m}}|^{\rho}\bigg)^{1/\rho},$$
where the supremum is taken over all decreasing
sequences $\{\epsilon_{m}\}_{m\in\Z}\subset (0,\infty)$.
Since $\|\cdot\|_{\VV_\rho}$ satisfies the triangle inequality,
we have $|\|a\|_{\VV_\rho}-\|b\|_{\VV_\rho}|\leq\|a-b\|_{\VV_\rho}$ for all $a:=\{a_\epsilon\}_{\epsilon>0}$ and $b:=\{b_\epsilon\}_{\epsilon>0}$, where $a-b:=\{a_\epsilon-b_\epsilon\}_{\epsilon>0}$. Hence,
\begin{multline*}
\big|(\VV_\rho\circ\TT_{\varphi}^{\HH^{n}_{\Gamma}})f_2(x)-c\,\big|^\rho
\leq \sup_{\{\epsilon_m\searrow0\}}\sum_{m\in\Z}
|(K\varphi_{\epsilon_{m+1}}^{\,\epsilon_m}*(f_2\HH^n_\Gamma))(x)-
(K\varphi_{\epsilon_{m+1}}^{\,\epsilon_m}*(f_2\HH^n_\Gamma))(z_D)|^{\rho}.
\end{multline*}

Given $x\in\Gamma\cap D$, let $\{\epsilon_m\}_{m\in\Z}$ be a decreasing sequence of positive numbers (which depends on $x$) such that
\begin{equation*}
\big|(\VV_\rho\circ\TT_{\varphi}^{\HH^{n}_{\Gamma}})f_2(x)-c\,\big|^\rho\\
\leq 2\sum_{m\in\Z}|(K\varphi_{\epsilon_{m+1}}^{\,\epsilon_m}*(f_2\HH^n_\Gamma))(x)-
(K\varphi_{\epsilon_{m+1}}^{\,\epsilon_m}*(f_2\HH^n_\Gamma))(z_D)|^{\rho}.
\end{equation*}
Notice that $|(K\varphi_{\epsilon_{m+1}}^{\,\epsilon_m}*(f_2\HH^n_\Gamma))(x)-
(K\varphi_{\epsilon_{m+1}}^{\,\epsilon_m}*(f_2\HH^n_\Gamma))(z_D)|\leq\|f\|_{L^\infty(\HH^n_\Gamma)}(\Theta1_m+\Theta2_m)$, where $\Theta1_m$ and $\Theta2_m$ are as in (\ref{teorema localitzacio11}) but replacing $\mu$ by $\HH^n_\Gamma$. It is straightforward to check that the arguments and computations given in subsection \ref{ss localitzacio1} to estimate the two terms in (\ref{teorema localitzacio11}) (see (\ref{teorema localitzacio12}), (\ref{teorema localitzacio13}), and (\ref{teorema localitzacio14})) still hold if we replace $\mu$ by $\HH^n_\Gamma$. Therefore, we have $$\sum_{m\in\Z=\SSS\cup\LL}(\Theta1_m+\Theta2_m)^\rho\lesssim1,$$
which impies that
$\big|(\VV_\rho\circ\TT_{\varphi}^{\HH^{n}_{\Gamma}})f_2(x)-c\,\big|\lesssim\|f\|_{L^\infty(\HH^n_\Gamma)}$ and, by integrating in $D$, gives (\ref{teorema interpolacio Linf1}).

\subsection{The operator $\VV_\rho\circ\TT_{\varphi}^{\HH^{n}_{\Gamma}}
\,:\,L^p(\HH^{n}_{\Gamma})\to L^p(\HH^{n}_{\Gamma})$ is
bounded for all $1<p<\infty$}\label{ss Lp suau}

Since $\VV_\rho\circ\TT_{\varphi}^{\HH^{n}_{\Gamma}}$ is
sublinear, the $L^p$ boundedness follows by applying the results of
subsection \ref{teorema interpolacio s1} and subsection \ref{teorema
interpolacio s2}, and the interpolation theorem between the pairs
$(H^1(\HH^n_\Gamma),L^1(\HH^n_\Gamma))$ and
$(L^\infty(\HH^n_\Gamma),BMO(\HH^n_\Gamma))$ in \cite[page
43]{Journe}.

Given a v-cube $Q\subset\R^d$, set
$m_Q(f):=\HH^{n}_{\Gamma}(Q)^{-1}\int_Q f\,d\HH^{n}_{\Gamma}$, and let $M$ denote the Hardy-Littlewood maximal operator with respect to $\Gamma$, i.e. for $x\in\Gamma$, $M(f)(x):=\sup m_Q(|f|)$, where the supremum is taken over all v-cubes $Q\subset\R^d$ containing $x\in\Gamma$. Let $M^\sharp$ be the sharp maximal operator defined by $M^\sharp(f)(x):=\sup m_Q(|f-m_Q(f)|)$, where the supremum is also taken over all v-cubes $Q\subset\R^d$ containing $x\in\Gamma$.

One comment about the interpolation theorem in \cite[page 43]{Journe} is in order. Given an operator $F$ bounded form $H^1$ to $L^1$ and from $L^\infty$ to $BMO$, in the proof of the interpolation theorem applied to $F$, one uses that $M^\sharp\circ F$ is sublinear (i.e. $(M^\sharp\circ F)(f+g)\leq (M^\sharp\circ F)f+(M^\sharp\circ F)g$ for all functions $f,g$). This is the case when $F$ is linear.
However, $\VV_\rho\circ\TT_{\varphi}^{\HH^{n}_{\Gamma}}$ is not linear, and then it is not clear if $M^\sharp\circ\VV_\rho\circ\TT_{\varphi}^{\HH^{n}_{\Gamma}}$ is sublinear. 
Nevertheless, this problem can be fixed easily using that $\VV_\rho\circ\TT_{\varphi}^{\HH^{n}_{\Gamma}}$ is sublinear and positive (that is $(\VV_\rho\circ\TT_{\varphi}^{\HH^{n}_{\Gamma}})f(x)\geq0$ for all $f$ and $x$), as the following lemma shows.

\begin{lema}\label{lema interpolacio}
Let $F:L^1_{loc}(\HH^{n}_{\Gamma})\to L^1_{loc}(\HH^{n}_{\Gamma})$ be a positive and sublinear operator. Then $(M^\sharp\circ F)(f+g)\lesssim (M\circ F)f+(M^\sharp\circ F)g$ for all functions $f,g$.
\end{lema}
\begin{proof}
If $F$ is sublinear and positive, one has that $|F(f)(x)-F(g)(x)|\leq F(f-g)(x)$ for all functions $f,g\in L^1_{loc}(\HH^{n}_{\Gamma})$.  In particular, $|F(f+g)(x)-F(g)(x)|\leq F(f)(x)$. Then, for $x,y\in Q\cap\Gamma$,
\begin{equation*}
\begin{split}
|F(f+g)(y)-m_Q(Fg)|&\leq|F(f+g)(y)-Fg(y)|+|Fg(y)-m_Q(Fg)|\\
&\leq|Ff(y)|+|Fg(y)-m_Q(Fg)|.
\end{split}
\end{equation*}
Hence, $m_Q|F(f+g)-m_Q(Fg)|\leq m_Q|Ff|+m_Q|Fg-m_Q(Fg)|\leq (M\circ F)f(x)+(M^\sharp\circ F)g(x)$ and, by taking the supremum over all possible v-cubes $Q\ni x$, we conclude $(M^\sharp\circ F)(f+g)(x)\lesssim(M\circ F)f(x)+ (M^\sharp\circ F)g(x).$
\end{proof}

By using Lemma \ref{lema interpolacio} and the fact that $\|M f\|_{L^p(\HH^{n}_{\Gamma})}\lesssim\|M^\sharp f\|_{L^p(\HH^{n}_{\Gamma})}$ for $f\in L^{p_0}(\HH^{n}_{\Gamma})\cap L^{p}(\HH^{n}_{\Gamma})$ and $1\leq p_0\leq p<\infty$, one can reprove Journ\'{e}'s interpolation theorem applied to $\VV_\rho\circ\TT_{\varphi}^{\HH^{n}_{\Gamma}}$ with minor modifications in the original proof.

\subsection{The operator $\VV_\rho\circ\TT_{\varphi}^{\HH^{n}_{\Gamma}}\,:\,L^1(\HH^{n}_{\Gamma})\to L^{1,\infty}(\HH^{n}_{\Gamma})$ is bounded}
By adapting \cite[Theorem B]{CJRW-singular integrals} to our setting and using the smoothnes of the family $\varphi$, one can show that the $L^2(\HH^n_\Gamma)$ boundedness of $\VV_\rho\circ\TT_{\varphi}^{\HH^n_\Gamma}$ yields the boundedness of this operator from $L^{1}(\HH^n_\Gamma)$ to $L^{1,\infty}(\HH^n_\Gamma)$. The interested reader may see \cite{MT}, where a more general result is proved.

\end{document}